\documentclass[fontsize=12pt,a4paper,headings=normal,
twoside=false,leqno,parskip=half-,abstract=true]{scrartcl}
\usepackage[english]{babel}
\usepackage[utf8]{inputenc}
\usepackage{hyperref}
\hypersetup{
 pdftitle={Stabilized rapid oscillations in a delay equation:
 Feedback control by a small resonant delay},
 pdfauthor={Bernold Fiedler, Isabelle Schneider},
 colorlinks=true,
 linkcolor=blue,
 citecolor=blue,
 filecolor=blue,
 urlcolor=blue}

\usepackage{graphicx}
\usepackage[format=plain,labelfont=bf,font=small]{caption}
\usepackage{xcolor}
\usepackage[arrow, matrix, curve]{xy}
\usepackage{float}
\usepackage{caption}

\usepackage{amsmath,amsthm}
\usepackage{amssymb} 
\usepackage{latexsym}

\usepackage[notref,notcite,color,
final 
]{showkeys}

\definecolor{refkey}{rgb}{1,0,0}
\definecolor{labelkey}{rgb}{1,0,0}

\usepackage[textwidth=2cm,textsize=small,backgroundcolor=none]{todonotes}

  \mathchardef\ordinarycolon\mathcode`\:
  \mathcode`\:=\string"8000
  \begingroup \catcode`\:=\active
    \gdef:{\mathrel{\mathop\ordinarycolon}}
  \endgroup

\theoremstyle{plain}
\newtheorem{thm}{Theorem}[section]
\newtheorem{lem}[thm]{Lemma}

\newtheorem{prop}[thm]{Proposition}
\newtheorem{cor}[thm]{Corollary}

\begin{document}

\title{\Large{Stabilized rapid oscillations in a delay equation:\\
 Feedback control by a small resonant delay}}
\subtitle{	
	\vspace{4ex}
	\textit{{\normalsize -- Dedicated to Jürgen Scheurle 
	in gratitude and friendship --} }\vspace{4ex}
	}

\author{
Bernold Fiedler*\\
Isabelle Schneider*\\
\vspace{2cm}}
\date{version of \today}
\maketitle
\thispagestyle{empty}

\vfill

*\\
Institut für Mathematik\\
Freie Universität Berlin\\
Arnimallee 3/7\\ 
14195 Berlin, Germany\\


\newpage
\pagestyle{plain}
\pagenumbering{roman}
\setcounter{page}{1}

\begin{abstract}
We study scalar delay equations 
$$\dot{x} (t) = \lambda f(x(t-1)) + b^{-1} (x(t) + x(t -p/2))$$
with odd nonlinearity $f$, real nonzero parameters $\lambda, \, b$, and two positive time delays $1,\ p/2$.
We assume supercritical Hopf~bifurcation from $x \equiv 0$ in the well-understood single-delay case $b = \infty$.
Normalizing $f' (0)=1$, branches of constant minimal period $p_k = 2\pi/\omega_k$ are known to bifurcate from eigenvalues $i\omega_k = i(k+\tfrac{1}{2})\pi$ at $\lambda_k = (-1)^{k+1}\omega_k$, for any nonnegative integer $k$.
The unstable dimension of these rapidly oscillating periodic solutions is $k$, at the local branch $k$.
We obtain stabilization of such branches, for arbitrarily large unstable dimension $k$, and for, necessarily, delicately narrow regions $\mathcal{P}$ of scalar control amplitudes $b < 0$.

For $p$:= $p_k$ the branch $k$ of constant period $p_k$ persists as a solution, for any $b\neq 0$.
Indeed the delayed feedback term controlled by $b$ vanishes on branch $k$:
the feedback control is noninvasive there.
Following an idea of \cite{pyr92}, we seek parameter regions $\mathcal{P} = (\underline{b}_k,\overline{b}_k)$ of controls $b \neq 0$ such that the branch $k$ becomes stable, locally at Hopf~bifurcation.
We determine rigorous expansions for $\mathcal{P}$ in the limit of large $k$.
Our analysis is based on a \emph{2-scale covering lift} for the slow and rapid frequencies involved.

These results complement earlier results by \cite{fiol16} which required control terms 
$$b^{-1} (x(t-\vartheta) + x(t-\vartheta -p/2))$$
with a third delay $\vartheta$ near 1.
\end{abstract}

\tableofcontents
  

\newpage
\pagenumbering{arabic}
\setcounter{page}{1}

\section{Introduction and main result}
\label{sec:1}

\numberwithin{equation}{section}
\numberwithin{figure}{section}

In an ODE setting, \emph{delayed feedback control} is frequently studied for systems like
	\begin{equation}
	\dot{\mathbf{x}}(t) =
	\mathbf{F}(\mathbf{x} (t)) +
	\beta (\mathbf{x} (t) - \mathbf{x}(t-\tau))
	\label{eq:1.1}
	\end{equation}
with $\mathbf{x} \in \mathbb{R}^N$, smooth nonlinearities $\mathbf{F}$, and suitable $N \times N$ matrices $\beta$ mediating the feedback.
If the uncontrolled system $\beta = 0$ possesses a periodic orbit $\mathbf{x}_*(t)$ of (not necessarily minimal) period $p > 0$, then $\mathbf{x}_*(t)$ remains a solution of \eqref{eq:1.1} for time delays $\tau = p$ and any control matrix $\beta$.
In this sense, the delayed feedback control is noninvasive on $\mathbf{x}_*(t)$.
The linearized and nonlinear stability or instability of $\mathbf{x}_*(t)$, however, may well be affected by the control term $\beta$.

The above idea was first proposed by Pyragas, see~\cite{pyr92}.
It has gained significant popularity in the applied literature since then, with currently around 3000~publications listed.
See~\cite{fieetal08} and~\cite{pyr12} for more recent surveys.
In fact, no previous knowledge of the nonlinearity $\mathbf{F}$ is required to attempt this procedure, or any of its many variants.

One fundamental disadvantage of the Pyragas~method~\eqref{eq:1.1}, from a theoretical perspective, is the replacement of the ODE $\dot{\mathbf{x}} = \mathbf{F}(\mathbf{x})$ in finite-dimensional phase space $X = \mathbb{R}^N$ by the infinite-dimensional dynamical system~\eqref{eq:1.1} in a history phase space like 
$\mathbf{x}(t+ \cdot ) \in X = C^0 ([-\tau,\, 0],\, \mathbb{R}^N)$.
On the other hand, the very existence of a periodic solution $\mathbf{x}(t)$, for vanishing control $\beta = 0$, requires $N \geq 2$.

In \cite{fiol16} we therefore started to explore the Pyragas~method of delayed feedback control, in a slightly modified form, for the very simplistic scalar case
	\begin{equation}
	\dot{x}(t) =
	\lambda f (x(t-1)) +
	b^{-1} (x(t-\vartheta) +
	x(t-\vartheta - p/2))\,.
	\label{eq:1.2}
	\end{equation}
We consider nonzero real parameters $\lambda,\, b$ and positive delays $\vartheta,\ p/2$. The case $b = 0$ will only appear as a formal limit $\beta = \pm \infty$ of infinite feedback amplitudes. The identical cases $b = \pm \infty$ account for vanishing feedback $\beta=0$ and correspond to the scalar pure delay equation
	\begin{equation}
	\dot{x}(t) =
	\lambda f (x(t-1))
	\label{eq:1.3}
	\end{equation}
with $|\lambda |$ normalizing the remaining delay $\tau$ to unity.
See \cite{wri55} for an early analysis of a specific equation of this type, equivalent to the delayed logistic equation.

Throughout the paper we assume $f \in C^3$ to be odd, with normalized first derivative at $f(0) = 0$:
	\begin{equation}
	f(-x) =	-f (x), \quad
	f'(0) = 1, \quad
	f'''(0) < 0\,.
	\label{eq:1.4}
	\end{equation}
	
The characteristic equation for complex eigenvalues $\mu$ of the linearization of~\eqref{eq:1.2} at parameter $\lambda$ and the trivial equilibrium $x \equiv 0$ then reads
	\begin{equation}
	\mu =
	\lambda e^{-\mu} + b^{-1} 
	(e^{-\vartheta \mu} + e^{- (\vartheta+ p/2) \mu})\,.
	\label{eq:1.2ch}
	\end{equation}
See~\cite{kapyor74} for an analysis of odd periodic solutions $x_k(t)$ of the pure delay equation \eqref{eq:1.3} with constant minimal period	
	\begin{equation}
	p_k := 2 \pi/ \omega_k, \quad
	\omega_k := (k+\tfrac{1}{2}) \pi\,.
	\label{eq:1.5}
	\end{equation}
The periodic solutions originate by Hopf~bifurcation from imaginary eigenvalues $\pm\, i \omega_k$ at $x = 0$ for parameters	
	\begin{equation}
	\lambda = \lambda_k :=
	(-1)^{k+1} \omega_k\,.
	\label{eq:1.6}
	\end{equation}
Here $k \in \mathbb{N}_0$ is any nonnegative integer.
For $k >0$, these periodic solutions are called \emph{rapidly oscillating} because their minimal period $p_k$ is at most $4/3$.
\emph{Slowly oscillating} periodic solutions, in contrast, have minimal periods exceeding 2.
For example, $p_0=4$.
See \cite{wal14} for a survey of related results.
In particular see~\cite{wal83, dor89} for secondary bifurcations from these primary branches.

The case $\dot{x}(t) = g(x(t),\, x(t-1))$ of the general scalar delay equation with a single time lag has attracted considerable attention; see for example~\cite{beco63, hale77, halevl93, dieetal95, wu96, kolmysh99, nu02} and the many references there.
After early observations by Myshkis \cite{mysh49}, Mallet-Paret \cite{mp88} has discovered a discrete Lyapunov~functional for $g$ with monotone delayed feedback.
The global consequences of this additional structure are enormous; see \cite{fiemp89, kri08, lop17, mpse96a, mpse96b, wal95}.
For example, all rapidly oscillating periodic solutions are known to be unstable.
More recent developments study this scalar equation with state dependent delays, where the time delay $1$ is not constant but depends on the history $x(t+ \cdot)$ of the solution itself; see for example~\cite{haretal06,mpnu92, mpnu96, mpnu03, mpnu11, nu02}. An excellent survey article on the above developments for general scalar delay equations with a single delay is \cite{wal14}.

\begin{figure}[t!]
\centering \includegraphics[width=\textwidth]{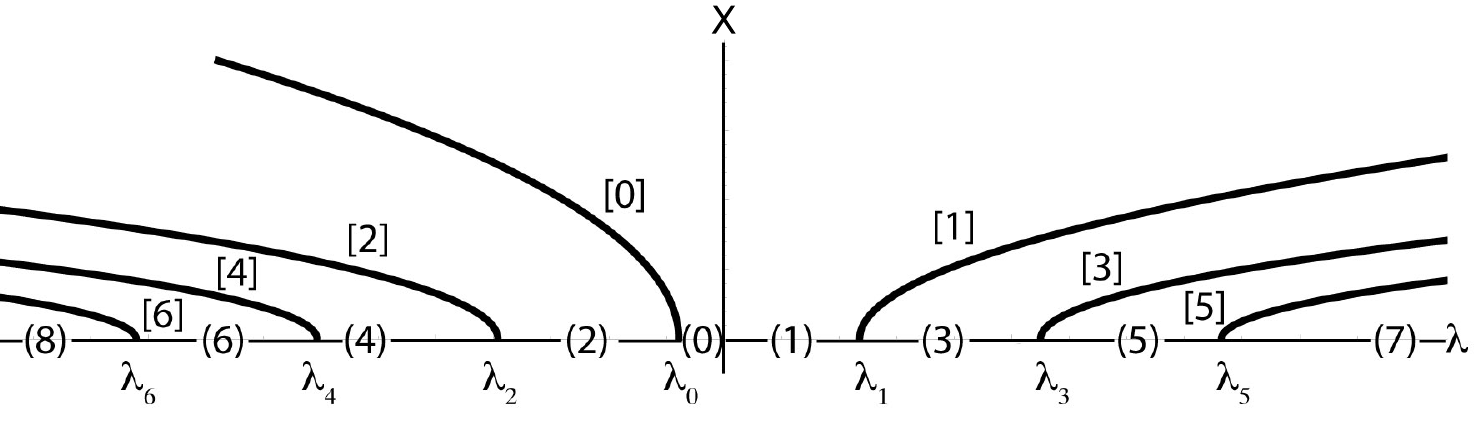}
\caption{\emph{
Supercritical Hopf~bifurcations of \eqref{eq:1.3} at $\lambda = \lambda_k$.
Note the strict unstable dimensions $E(\lambda_k) =k$ of the trivial equilibrium, in parentheses $(k)$, and the inherited unstable dimensions $[k]$, in brackets, of the local branches of bifurcating periodic orbits with constant minimal period $p_k = 4/(2k+1)$.
All branches consist of unstable rapidly oscillating periodic solutions, except for the stable slowly oscillating branch $k=0$.
See \cite{fiol16}.
}}
\label{fig:1.1}
\end{figure}

But let us return to the simple setting \eqref{eq:1.3} -- \eqref{eq:1.5} of a pure delay equation.
Pioneering analysis by~\cite{kapyor74} reduces the quest for periodic solutions near all Hopf bifurcations \eqref{eq:1.6} to a planar Hamiltonian ODE system.
This is due to an \emph{odd-symmetry}
	\begin{equation}
	x_k(t+p_k/2) =
	-x_k(t)
	\label{eq:1.7}
	\end{equation}
at half minimal period $p_k$, for all real $t$.
See also \cite{fiol16} for complete details, and \cite{yuguo14} for a survey on the Kaplan-Yorke idea.
Remarkably, global solution branches of constant minimal period $p_k$ emanate from each $\lambda = \lambda_k$ towards $\lambda$ of larger absolute value, in the soft spring case of strictly decreasing secant slopes $x \mapsto f(x)/x$, for $x >0$.
In particular all Hopf~bifurcations are locally nondegenerate and quadratically supercritical under the sign assumption $f'''(0) < 0$ of~\eqref{eq:1.4}. See fig.~\ref{fig:1.1} for a bifurcation diagram. 

\begin{figure}[t!]
\centering \includegraphics[width=\textwidth]{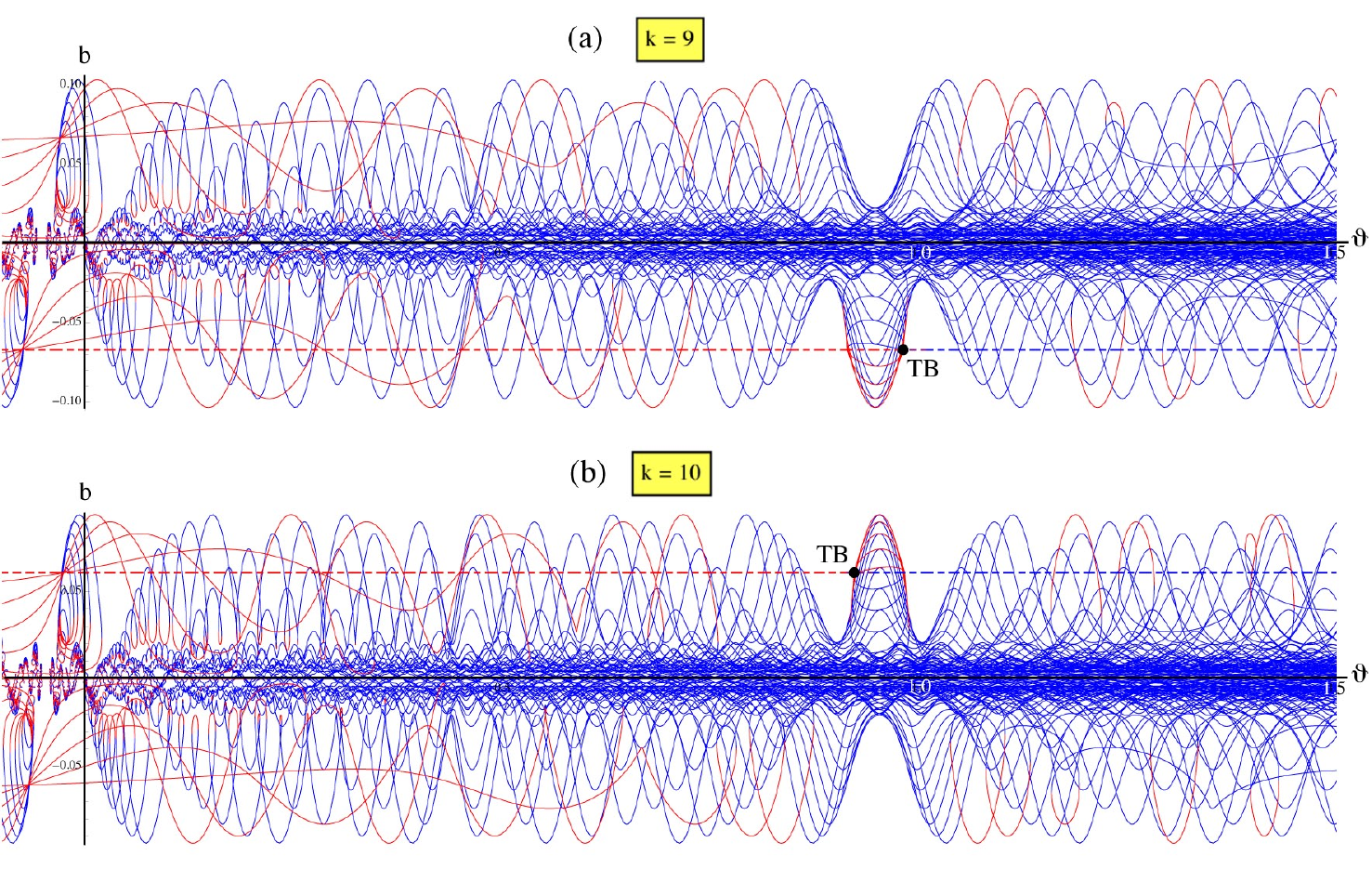}
\caption{\emph{
Additional Hopf~curves (colored solid), zero eigenvalue (red dashed), and Takens-Bogdanov~bifurcations (TB, black) at fixed $\lambda = \lambda_k$, for odd $k = 9$, (a) top, and even $k = 10$, (b) bottom. 
The Hopf curves are generated by the control parameters $\vartheta$ and $b$ of the delayed feedback terms in \eqref{eq:1.2}.
The more stable side is found towards smaller $|b|$, at red Hopf branches, and towards larger $|b|$, at blue branches. 
The same statement holds true at the zero eigenvalue; see the red dashed line. 
See \cite{fiol16} for further details.
}}
\label{fig:1.2}
\end{figure}

At supercritical Hopf~bifurcation it is easy to determine the unstable dimension $E$, i.e. the total algebraic multiplicity of Floquet~multipliers outside the complex unit circle, for the emanating local branch of periodic orbits.
It coincides with the total algebraic multiplicity
	\begin{equation}
	E = E (\lambda_k) = k
	\label{eq:1.8}
	\end{equation}
of the eigenvalues $\mu$ with strictly positive real part for the characteristic equation~\eqref{eq:1.2ch} at the Hopf~point $\lambda = \lambda_k\,$.
See for example~\cite{dieetal95, hale77, halevl93}.
Henceforth we skip the trivial case $k=0$, which leads to the bifurcation of stable, slowly oscillating solutions. Let $k\in \mathbb{N}$ be any strictly positive integer.

The local (and global) Hopf branches $(\lambda, x_k(t))$ which bifurcate at $\lambda = \lambda_k, \ x=0$ inherit constant period $p=p_k$ and odd-symmetry \eqref{eq:1.7}. 
Therefore our modified Pyragas control scheme~\eqref{eq:1.2} with $p$:= $p_k$ is noninvasive on the Hopf~branches with this symmetry.
For earlier applications of Pyragas control at half minimal periods in the presence of some involutive symmetry, although not in our present delay context, we refer to \cite{nakue98, fieetal10}.
It is our main objective to stabilize all bifurcating periodic orbits, for arbitrarily large unstable dimensions $k\in \mathbb{N}$, by suitable Pyragas controls \eqref{eq:1.2}.
For odd $k$, in particular, this again refutes the purported ``odd number limitation'' of Pyragas control \cite{fieetal07, juetal07, nak97, nakue98}.

We define a \emph{Pyragas~region} $\mathcal{P}$ to be a connected component of real control parameters $b \neq 0$ and $\vartheta \geq 0$ for which the periodic solutions $x_k(t)$ emanating by local Hopf~bifurcation from $\lambda=\lambda_k,\, x \equiv 0$ become linearly asymptotically stable for sufficiently small amplitudes.
Therefore, the boundaries of Pyragas regions are either certain curves where zero eigenvalues $\mu = 0$ occur, or else are Hopf curves characterized by purely imaginary eigenvalues $\mu =i\tilde{\omega}$ of the characteristic equation \eqref{eq:1.2ch}.

With this definition we can now formulate the main result of the previous paper \cite{fiol16}.
See fig.~\ref{fig:1.2} for an illustration of Hopf~curves in the cases $k = 9$ and $k =10$.

\begin{figure}
\centering \includegraphics[width=0.9\textwidth]{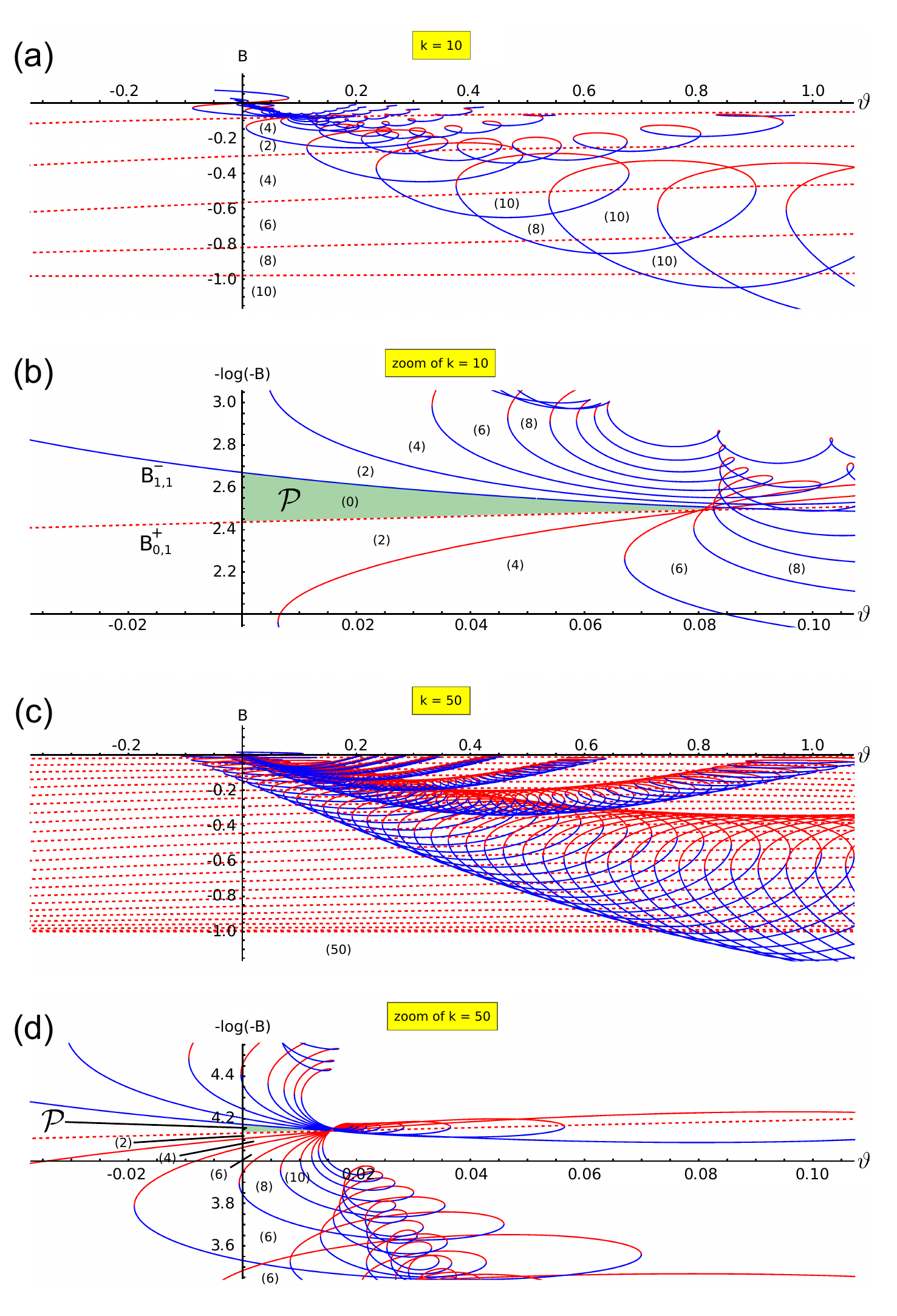}
\caption{\emph{
Control induced Hopf curves in parameters $(\vartheta,b)$, as in fig \ref{fig:1.2}, near $\vartheta = 0$. (a) $k=10$, (b) zoom into $k=10$, (c) $k=50$, (d) zoom into $k=50$.
Vertical coordinates are $B$, in (a), (c), and $-\log(-B)$, for the zooms (b),(d), with scaled $B=\frac{1}{2} b \omega_k$. Pyragas regions $\mathcal{P}$ are indicated in green. Hopf curves $\mu=i\tilde{\omega}$ with Hopf frequencies $0 < \tilde{\omega} < \omega_k$ are dashed (red), and Hopf curves with $\tilde{\omega} > \omega_k$ are solid (red,blue).
For color coding see fig. \ref{fig:1.2}. Unstable dimensions $E(b,\vartheta)$ of $x\equiv 0$, and of bifurcating periodic orbits, are indicated in parentheses. 
}}
\label{fig:1.3}
\end{figure}

\begin{thm}
\cite{fiol16}
Consider the system~\eqref{eq:1.2} of delayed feedback control for the scalar pure delay equation~\eqref{eq:1.3}.
Let assumptions~\eqref{eq:1.4} of oddness and normalization hold for the soft spring nonlinearity $f \in C^3$.
Then the following assertions hold for large enough $k \geq k_0$.

There exist Pyragas~regions $\mathcal{P} = \mathcal{P}_k^+ \, \cup \, \mathcal{P}_k^-$ composed of two disjoint open sets $\mathcal{P}_k^\pm \neq \varnothing$.
Each region $\mathcal{P}_k^\iota,\, \iota = \pm$, is bounded by the horizontal \emph{zero line}
	\begin{equation}
	b =b_k := -2/\lambda_k
	=(-1)^k \cdot 2/\omega_k
	\label{eq:1.9}
	\end{equation}	
and three other analytic curves $\gamma_k^0$ and $\gamma_{k,\,\pm}^\iota$, all mutually transverse.
The zero line \eqref{eq:1.9} indicates a zero eigenvalue $\mu$ of the characteristic equation \eqref{eq:1.2ch} of~\eqref{eq:1.2} at $x \equiv 0$.
The other curves indicate additional purely imaginary eigenvalues $\mu$.

Define $\varepsilon$:= $1/\omega_k$.
Then an approximation of the Pyragas~regions $(\vartheta,\,b) \in \mathcal{P}_k^\iota,\, \iota = \pm$, up to error terms of order $\varepsilon^3$, is given by two parallelograms.
One exact horizontal boundary is $b=b_k := (-1)^k\, 2 \varepsilon$; see~\eqref{eq:1.9}.
The other horizontal boundary $\gamma_k^0$ is approximated by
	\begin{equation}
	b = (-1)^k\, 2 \varepsilon +
	b_k^\iota\, \varepsilon^2 + \ldots\,.
	\label{eq:1.10}
	\end{equation}
The sides $\gamma_{k,\,\pm}^\iota$ are given, up to order $\varepsilon^3$, by the parallel slanted lines through the four points at $b=b_k$,
	\begin{equation}
	\vartheta_{k,\,\pm}^\iota = 1 -
	(\tfrac{\pi}{2} - \iota\,q) \varepsilon
	\pm \Theta_{k+2}\, \varepsilon^2 + \ldots\,,	
	\label{eq:1.11}
	\end{equation}
with slopes $\sigma_k^\iota$.
The offsets $q,\,\Theta_{k+2},\, b_k^\iota$ and the slopes $\sigma_k^\iota$ of the Pyragas parallelograms are given by
	\begin{equation}
	\begin{aligned}
	q\  &= \ 
	\arccos \, (2/\pi)
	&&= \  0.88 \ldots\,,\\
	\Theta_{k+2}\  &= \ 
	\pi (\sqrt{(\tfrac{\pi}{2})^2-1} - 2q)
	&&=  \ - 1.73 \ldots\,,\\
	b_k^\iota \  &= \ 
	(\pi+2 \Phi_k^\iota )\, 
	\cos\, \Phi_k^\iota
	&&= \  1.49 \ldots + (-1)^k \iota\cdot 1.02 \ldots\,,\\
	\sigma_k^\iota \  &= \ 
	2(-1)^{k+1} \iota \sqrt{(\tfrac{\pi}{2})^2-1}
	&&= \  (-1)^{k+1} \iota \cdot 2.42 \ldots
	\end{aligned}
	\label{eq:1.13a}
	\end{equation}
Here we have used the abbreviations
	\begin{equation}
	\Phi_k^\iota \ = 
	\  (-1)^k \iota \, \arcsin \, q
	\label{eq:1.13b}
	\end{equation}
In particular the areas $| \mathcal{P}_k^\pm |$ of the Pyragas~regions are of very small order $\varepsilon^4$.
The relative areas are approximately reciprocal,
	\begin{equation}
	\lim_{k \mapsto \infty} |\mathcal{P}_k^+|
	/ |\mathcal{P}_k^-| =
	\lim_{k \mapsto \infty} b_k^+ / b_k^- = 
	\left\lbrace 
	\begin{aligned}
	&5.37\ldots 
	\qquad &\mathrm{for}\ &\mathrm{even}\ k \,,\\
	&0.19 \ldots 
	\qquad &\mathrm{for}\ &\mathrm{odd}\ k\,.
	\end{aligned}
	\right.
	\label{eq:1.14}
	\end{equation}
\label{thm:1.1}
\end{thm}

See \cite{fiol16} for full details and proofs.

The above stabilization result for rapidly oscillating periodic solutions requires two additional delays in the control term of \eqref{eq:1.2}:
the half-period delay $p_k/2$, and a joint offset $\vartheta$ near 1, in addition to the normalized delay 1 of the reference system.
In the present paper we achieve the same goal with the single delay $p_k/2$ in the control term, i.e. with vanishing delay offset:
	\begin{equation}
	\vartheta =0\,.
	\label{eq:1.17}
	\end{equation}
From now on, and for the rest of the paper, we therefore replace \eqref{eq:1.2} by
	\begin{equation}
	\dot{x} (t) = \lambda f(x(t-1)) +b^{-1} (x(t)+x(t-p/2))\,,
	\label{eq:1.18}
	\end{equation}
with nonlinearities $f \in C^3$ which satisfy assumptions \eqref{eq:1.4}.
Our main result identifies unique nonempty Pyragas intervals $\mathcal{P}_k =(\underline{b}_k, \overline{b}_k)$ of control parameters $b$. The $k$-dimensionally unstable, rapidly oscillating periodic solutions of constant minimal period $p_k=4/(2k+1)$ are stabilized near Hopf bifurcation at $\lambda=\lambda_k$,$\ x=0$, for $b\in \mathcal{P}_k$ and all sufficiently large $k\in \mathbb{N}$. Via $\varepsilon$:= $1/\omega_k = p_k/2\pi$, we also provide $\varepsilon$-expansions, alias $k$-expansions, for the Pyragas boundaries $\underline{b}_k$ and $\overline{b}_k$.
Taylor expansions with respect to $\varepsilon$ again amount to rapid oscillation expansions at $k=\infty$.

\begin{thm}
Consider the system \eqref{eq:1.18} of delayed feedback control for the scalar pure delay equation \eqref{eq:1.3}.
Let assumptions \eqref{eq:1.4} of oddness and normalization hold for the soft spring nonlinearity $f\in C^3$.
Then the following assertions hold for large enough $k \geq k_0$, i.e. for small enough $0<\varepsilon$:= $1/\omega_k= ((k+\tfrac{1}{2})\pi)^{-1} \leq \varepsilon_0$.

The only Pyragas region of nonzero control amplitudes $b$ is the open interval
	\begin{equation}
	\mathcal{P}_k:= \lbrace \underline{b}_k < b< \overline{b}_k \rbrace\,.
	\label{eq:1.19}
	\end{equation}
Up to error terms of order $\varepsilon^4$, the lower and upper boundaries of the Pyragas interval satisfy
	\begin{equation}
	\begin{aligned}
	\underline{b}_k &=
	-\tfrac{1}{2}\pi^2\varepsilon^2 
	-\tfrac{3}{4}\pi^3\varepsilon^3+\ldots\,,\\
	\overline{b}_k &=
	-\tfrac{1}{2}\pi^2\varepsilon^2 
	+\tfrac{1}{4}\pi^3\varepsilon^3+\ldots
	\end{aligned}
	\label{eq:1.20}
	\end{equation}

\label{thm:1.2}
\end{thm}

Although our proofs and expansions only address Hopf bifurcations at sufficiently large unstable dimensions $k$, and sufficiently rapid oscillation frequencies $\omega_k\,,$ numerical evidence suggests a single Pyragas interval $\mathcal{P}_k\,,$ for any $k\geq 1$. 
We do not pursue these cases here, beyond the evidence provided in figs. \ref{fig:1.2} and \ref{fig:1.3}.

The remaining sections disentangle the elements of the proof of theorem~\ref{thm:1.2}.
We give a brief outline here.
For a summary of sections~\ref{sec:2} -- \ref{sec:5}, on a precise technical level, we refer to the proof of theorem~\ref{thm:1.2} in the concluding section~\ref{sec:6}.

Section~\ref{sec:2}, and most of the remaining sections, address the characteristic equation \eqref{eq:1.2ch}, at $\vartheta=0$, for the linearization at the original Hopf bifurcation points $\lambda = \lambda_k,\ x\equiv 0$.
Elementary as this task may appear, the rapidly oscillatory terms which appear in the limit $\varepsilon = 1/\omega_k \searrow 0$ cause substantial and worthwhile difficulties.

In section~\ref{sec:2} we first recall some elementary results from \cite{fiol16} which address the crossing direction of an additional simple eigenvalue $0$ induced by the control term.
We also introduce a 2-\emph{scale lift}, which artificially represents the large, rapidly oscillatory imaginary parts of $b$-induced Hopf eigenvalues $\mu = i\tilde{\omega}$ by, both, $\tilde{\omega}$ itself and a scaled slow frequency
	\begin{equation}
	\tilde{\Omega} = \varepsilon \tilde{\omega}\,.
	\label{eq:1.21}
	\end{equation}
Note how $\tilde{\Omega} =1,\ \tilde{\omega} = \omega_k$ correspond to the reference Hopf bifurcation at $\lambda = \lambda_k$.
We observe $\tilde{\Omega} \neq 2m$ cannot be at even integer resonance. We introduce a new local scaled slow frequency
	\begin{equation}
	\Omega:= \tilde{\Omega} -\Omega_m\,,\qquad
	\Omega_m := 2m+1\,,
	\label{eq:1.22}
	\end{equation}
near each odd integer resonance $\Omega_m$. Below, in fact, we will be able to focus on $-1 < \Omega \leq 0$.
Since
	\begin{equation}
	\omega \equiv \tilde{\omega} \pmod{2\pi}
	\label{eq:1.23}
	\end{equation}
rotates rapidly through $S^1$, for small $\varepsilon$, we treat the two frequencies $\Omega, \omega$ as independent variables, formally.
They remain related by the \emph{hashing relation} 
\begin{equation}
	\Omega = \varepsilon(\omega +\tfrac{\pi}{2} (1-(-1)^k-(-1)^m)
	-2\pi j)\,,\qquad
	-\tfrac{\pi}{2} \leq \omega <\tfrac{3}{2}\pi\,,
	\label{eq:1.24}
	\end{equation}
$j \in \mathbb{N}$, first discussed in lemma~\ref{lem:2.3}.
We will be able to restrict attention to the case of odd $k$, in this setting.

The above hashing trick, first used in \cite{fiol16}, will be of central importance in our analysis.
In the limit $\varepsilon \searrow 0$, alias $h \nearrow \infty$, the hatching by the hashing lines (\ref{eq:1.24}) fills the $(\Omega,\omega)$-cylinder, densely.
Indeed, the hashing lines define steeply slanted (non-military) ``barber pole'' stripes of horizontal distance $\varepsilon$ around the cylinder.
The hashing trick (\ref{eq:1.24}) allows us to consider $\Omega$ and $\omega$ as independent cylinder variables, temporarily. 
This eliminates $\varepsilon > 0$ from the characteristic equation, altogether, as follows.

For the control amplitude $b$ we proceed with the same scaling
	\begin{equation}
	B:= \tfrac{1}{2} b \varepsilon^{-1}
	\label{eq:1.25}
	\end{equation}
as in \cite{fiol16}.
Inserting the scalings \eqref{eq:1.21} -- \eqref{eq:1.23} into the characteristic equation \eqref{eq:1.2ch} for $\mu=i\tilde{\omega}$ at $\vartheta=0$,
we thus arrive at the 2-\emph{scale characteristic equation}
	\begin{equation}
	0= -ie^{i\omega}+\Omega_m+\Omega +(-1)^mB^{-1}
	\cos (\tfrac{\pi}{2} \Omega) e^{i\tfrac{\pi}{2}\Omega}\,.
	\label{eq:1.26}
	\end{equation}	
In lemmata \ref{lem:2.4}-\ref{lem:2.6} we solve the $\varepsilon$-independent (!) complex characteristic equation \eqref{eq:1.26} for the real variables
	\begin{equation}
	\begin{aligned}
	\omega &= \omega^\pm (\Omega)\,,\\
	B &= B^\pm (\Omega)\,.
	\end{aligned}
	\label{eq:1.27}
	\end{equation}
See figs.~\ref{fig:3.1}, \ref{fig:3.2} below for illustration.

In section \ref{sec:3} we observe how the imaginary parts $\tilde{\omega}$ of unstable eigenvalues $\mu=\mu_{R}+i\tilde{\omega}$ are trapped in certain strips indexed by $m,j$. Instability in such a strip can be induced by Hopf bifurcation at control parameters $B=B_{m,j}^{-}$, and be reduced again at control parameters $B^+_{m,j}$.
This involves an analysis of the crossing directions of $\mu$, transversely to the imaginary axis $\mu_R=0$, as $B$ increases through $B^{\pm}_{m,j}$. See  fig.~\ref{fig:3.3} and theorem \ref{thm:3.4}. 
In corollary \ref{cor:3.6} we conclude the absence of any region of Pyragas stabilization for $B>0$.
Corollary \ref{cor:3.9} summarizes the results of section \ref{sec:3}: we reduce the proof of theorem \ref{thm:1.2} to the three inequalities \eqref{eq:3.67} -- \eqref{eq:3.69} among the Hopf parameter values $B^{\pm}_{m,j}$.

In section \ref{sec:4}, we insert the solution $\omega =\omega^\pm (\Omega)$ from \eqref{eq:1.27} into the hashing \eqref{eq:1.24}.
Inverting the resulting maps $\Omega \mapsto \varepsilon= \varepsilon(\Omega)$, uniformly for bounded $m \leq m_0,\ j \leq j_0$, we obtain $\varepsilon$-expansions
	\begin{equation}
	\begin{aligned}
	\Omega &= \Omega^\pm_{m,j}(\varepsilon)\,,\\
	\omega &= \omega^\pm_{m,j}(\varepsilon)\,,\\
	B&= B^\pm_{m,j}(\varepsilon)
	\end{aligned}
	\label{eq:1.28}
	\end{equation}		
for the frequencies $\tilde{\omega} \equiv \omega$ and the control amplitudes $b=2B\varepsilon$ of the resulting control-induced Hopf bifurcations.
In particular the crossing directions of the induced imaginary Hopf pairs with respect to $b$ sum up such that
	\begin{equation}
	\underline{b} = 2\varepsilon B_{0,1}^+(\varepsilon)\,,\qquad
	\overline{b} =2\varepsilon B_{1,1}^-(\varepsilon)
	\label{eq:1.29}
	\end{equation}		
provide the boundaries of the Pyragas region $\mathcal{P}$ claimed in theorem~\ref{thm:1.2}.
We illustrate the relative location of $B^{\pm}_{m,j}$, in view of the crucial inequalities required in corollary \ref{cor:3.9}, at the end of section \ref{sec:4}; see also fig \ref{fig:4.1}.

It remains to show, however, that the candidate interval $b \in (\underline{b},\overline{b})$ does not suffer any destabilization, due to any other Hopf points $B_{m,j}^\pm$.
This turns out to be equivalent to the estimates
	\begin{equation}
	B_{m, j_{m}+1}^+ < B_{0,1}^+\qquad \text{and} \qquad
	B_{m,j_{m}}^- >B_{1,1}^-
	\label{eq:1.30}
	\end{equation}
at $j_m$:= $[(m+1)/2]$.
See \eqref{eq:3.67}, \eqref{eq:3.69} and corollary \ref{cor:3.9} again.
For bounded $m \leq m_0$, these estimates are suggested by the explicit expansions \eqref{eq:1.28}. In section~\ref{sec:5} we begin to settle the delicate case of large $m, \Omega_m$ by expansions with respect to
	\begin{equation}
	\delta:= \Omega_m^{-1} = 1/(2m+1)\,.
	\label{eq:1.31}
	\end{equation}
Here our second small parameter $\delta > 0$ expands the odd integer resonance regions around $\Omega = \Omega_m = 2m+1$, for large $m$, in much the same way as our first small parameter $\varepsilon$ expanded the discrete parameter $k$, for large $k \geq k_0$, which enumerated the original Hopf bifurcations of more and more rapidly oscillating periodic solutions with higher and higher unstable dimension.

This time, we solve the characteristic equation \eqref{eq:1.26} to obtain expansions
	\begin{equation}
	\begin{aligned}
	\Omega &= \Omega^\pm(\delta, \omega)\,,\\
	B&= B^\pm(\delta, \omega)
	\end{aligned}
	\label{eq:1.32}
	\end{equation}
with respect to $\delta$, uniformly in $|\omega | \leq \pi /2$.
Here $\Omega^-, B^-$ refer to the case $j=j_m$ and $\Omega^+, B^+$ refer to $j=j_{m}+1$.
In section~\ref{sec:5}, the hashing relation \eqref{eq:1.24} then provides a $\delta$-expansion for
	\begin{equation}
	\varepsilon^\pm = \varepsilon^\pm (\delta, \omega)\,.
	\label{eq:1.33}
	\end{equation}
Inserting this into the already established expansions \eqref{eq:1.29} for $\underline{b}, \overline{b}$, and comparing the results, for small $\delta$, we obtain
	\begin{equation}
	B^+(\delta, \omega ) < B_{0,1}^+(\varepsilon^+(\delta, \omega))
	\qquad	\text{and} \qquad 
	B^-(\delta, \omega )> B_{1,1}^-(\varepsilon^-(\delta,\omega))
	\label{eq:1.34}
	\end{equation}
as claimed in \eqref{eq:1.30}.
Well, nontrivial differences only appear at order $\delta^3$ and after additional linearization with respect to $\omega$, at $\omega = \pm \tfrac{\pi}{2}$.

The proof of theorem~\ref{thm:1.2} only involves some discussion of a characteristic equation with two exponential terms of different scales.
Nevertheless, the elementary ingredients to the proof turn out to be surprisingly involved.
Therefore we summarize the various elements of the proof, as scattered across sections~\ref{sec:2} -- \ref{sec:5}, in our final section~\ref{sec:6}.


\textbf{Acknowledgments.}
For many helpful comments and suggestions, as well as most of the figures, we are much indebted to Alejandro~López~Nieto.
Delightful discussions were provided by several participants of the conference in honor of J\"urgen Scheurle, and in particular by P.S. Krishnaprasad.
We are particularly grateful for the lucid remarks of our referees, which helped us improve the somewhat messy presentation.
Ulrike~Geiger typeset the original manuscript, with expertise and diligence.
The authors have been supported by the CRC~910 \emph{``Control of Self-Organizing Nonlinear Systems:
Theoretical Methods and Concepts of Application''} of the Deutsche Forschungs\-gemeinschaft.


\section{The 2-scale characteristic equation}
\label{sec:2}

The characteristic equation \eqref{eq:1.2ch} of the delay equation \eqref{eq:1.2} with vanishing time shift $\vartheta = 0$ reads
	\begin{equation}
	\mu = -(-1)^k \varepsilon^{-1} e^{-\mu} +
	b^{-1} (1+e^{-\pi\varepsilon\mu})\,,
	\label{eq:2.1}
	\end{equation}
at Hopf bifurcation parameter $\lambda = \lambda_k=(-1)^{k+1} \omega_{k}$, minimal period $p_k=2\pi / \omega_k$, and with the abbreviation
	\begin{equation}
	\varepsilon = \omega_k^{-1} = 1/((k+\tfrac{1}{2})\pi)
	\label{eq:2.2}
	\end{equation}
for $k\in \mathbb{N}$.
We decompose the eigenvalue $\mu =\mu_R +i\tilde{\omega}$ into real and imaginary parts and define the auxiliary slow frequency $\tilde{\Omega}$:= $ \varepsilon \tilde{\omega}$; see \eqref{eq:1.21}.
For the convenient choice of
	\begin{equation}
	\omega :\equiv \tilde{\omega} \pmod{2\pi}\,,\qquad 
	-\tfrac{1}{2}\pi \leq \omega < \tfrac{3}{2} \pi\,,
	\label{eq:2.3}
	\end{equation}
we obtain the crucially important 2-scale characteristic equation
	\begin{equation}
	\begin{aligned}
	0 &= \chi (\varepsilon, \delta, \omega, \Omega, B, \mu_R):=\\
	&= -\varepsilon \mu_R + i \tilde{\Omega}- (-1)^k e^{-\mu_R+i\tilde{\omega}}
	-\tfrac{i}{B} \sin(\tfrac{\pi}{2}\Omega)
	e^{-\pi\varepsilon\mu_R+i\tfrac{\pi}{2}\Omega}
	+ \tfrac{1}{2B}(1-e^{-\pi\varepsilon\mu_R}),
	\end{aligned}
	\label{eq:2.4}
	\end{equation}
by some elementary arithmetic and with the abbreviations
	\begin{equation}
	\begin{aligned}
	\tilde{\Omega}&:= \varepsilon\tilde{\omega}\,;\\
	B&:= \tfrac{1}{2}b\varepsilon^{-1}\,;\\
	\delta&:= \Omega_m^{-1} = 1/(2m+1)\,;\\
	\Omega&:= \tilde{\Omega}-\Omega_m\,.
	\end{aligned}
	\label{eq:2.5}
	\end{equation}
In particular we have utilized the complex conjugate of (\ref{eq:2.1}).
Evidently, the solutions of \eqref{eq:2.4} for even parity of $k$ are trivially obtained from the solutions for odd parity if we replace $\omega$ by $\omega +\pi \pmod{2\pi}$.
For later use we note the relation
	\begin{equation}
	\sin( \tfrac{\pi}{2} \Omega ) = -(-1)^m \cos(\tfrac{\pi}{2} \tilde\Omega)\,;
	\label{eq:derivation}
	\end{equation}
see also \eqref{eq:2.42}.

For interpretation we recall that $\tilde{\Omega}=1$ indicates $\tilde{\omega}=\omega_{k}$, i.e. a $1:1$ resonance $\tilde{\omega}=\omega_{k}$ of the imaginary part $i \tilde{\omega}$ of $\mu$, under feedback control, with the original Hopf eigenvalue $\mu=i \omega_{k}$ at parameter $\lambda=\lambda_{k}$. 
Similarly, $\tilde{\Omega}=\Omega_{m}=(2m+1)$ indicates an odd integer $(2m+1):1$ resonance. The parameters $\varepsilon, \delta$ make $k,m$ look continuous, respectively, and replace them eventually.

For complex nonreal eigenvalues $\mu$, the imaginary part $\tilde{\omega} = \tilde{\Omega}/\varepsilon$ can be taken positive, without loss, and we may assume
	\begin{equation}
	\begin{aligned}
	-1 < \Omega \leq 0\,, \qquad &\text{for } m=0\,,\\
	-2 <\Omega \leq 0\,,\qquad &\text{for } m\in \mathbb{N}\,.
	\end{aligned}
	\label{eq:2.7}
	\end{equation}
We also note that $\chi$ is real analytic in all variables, for $B \ne 0$. 
Since $\mu_{R}=0$, for purely imaginary eigenvalues, the characteristic function $\chi$ of \eqref{eq:2.4} simplifies and becomes
	\begin{equation}
	\chi_0 (\delta, \omega, \Omega, B):=
	-i \chi(\varepsilon, \delta, \omega, \Omega,\mu_R=0)=
	\tilde{\Omega}+(-1)^k i e^{i\omega}-
	\tfrac{1}{B}\sin (\tfrac{\pi}{2}\Omega) e^{i\tfrac{\pi}{2}\Omega}\,,
	\label{eq:2.8}
	\end{equation}
in the Hopf case. 
Note how $\varepsilon$ has disappeared from the characteristic equation \eqref{eq:2.4}, in \eqref{eq:2.8}, at the price of a hidden hashing relation between $\omega$ and $\Omega$; see lemma \ref{lem:2.3} below.

In the present section we collect some elementary facts about the 2-scale characteristic equation \eqref{eq:2.4} -- \eqref{eq:2.5}.
Real eigenvalues $\mu = \mu_R$, i.e. the case $\tilde{\omega} = \tilde{\Omega} =0$, are addressed in lemma~\ref{lem:2.1}.
As a corollary we eventually obtain how $B$ has to be negative in any Pyragas region; see corollaries~\ref{cor:2.2} and \ref{cor:3.6}.
With a brief interlude on hashing in lemma~\ref{lem:2.3}, we embark on our discussion of purely imaginary eigenvalues $\mu_R =0$.
In lemma~\ref{lem:2.4} we show how to eliminate any two of the three variables $\omega, \Omega, B$ from the resulting $\varepsilon$-independent 2-scale characteristic equation
	\begin{equation}
	0 = \chi_0(\delta, \omega, \Omega, B)\,.
	\label{eq:2.9}
	\end{equation}
In particular we identify a quadratic loop
	\begin{equation}
	Q(\delta, \Omega, B) =0\,,
	\label{eq:2.10}
	\end{equation}
by elimination of $\omega$, such that purely imaginary eigenvalues $\mu_R=0$ can occur only if \eqref{eq:2.10} is satisfied.
In section~\ref{sec:3} we will observe how positive real parts, $\mu_R>0$, can only occur inside the loop, and negative real parts, i.e. linear stability as required in Pyragas regions, are confined to the exterior.
Via hashing, this leads to the definition of crucial pairs of Hopf parameter values
		\begin{equation}
		B_{m,j}^- < B_{m,j}^+ <0
		\label{eq:2.11}
		\end{equation}
such that the loss of stability caused at $B_{m,j}^-$, by a pair of purely imaginary eigenvalues, is recovered when the control parameter $B<0$ increases further to pass the matching Hopf value $B_{m,j}^+$.

\begin{lem}
The characteristic equation \eqref{eq:2.1} possesses a real zero eigenvalue, $\mu = \mu_R=0$, if and only if
	\begin{equation}
	B = B_0=(-1)^k\,.
	\label{eq:2.12}
	\end{equation}
The zero eigenvalue $\mu$ is algebraically simple, and its continuation $\mu = \mu(B)$ satisfies
	\begin{equation}
	\mathrm{sign}\, \operatorname{Re} \mu'(B_0) = (-1)^k =B_0\,,
	\label{eq:2.13}
	\end{equation}
i.e. $\mu(B)$ increases towards larger $|B|$.
\label{lem:2.1}
\end{lem}

\begin{proof}[\textbf{Proof.}]
For real eigenvalues $\mu = \mu_R$, $\omega=\tilde{\Omega}=0$, and $\Omega=-1$, the characteristic equation \eqref{eq:2.4} reads
	\begin{equation}
	\varepsilon \mu_R= -(-1)^k e^{-\mu_R}+
	\tfrac{1}{2B}(1+e^{-\pi\varepsilon\mu_R})\,.
	\label{eq:2.14}
	\end{equation}
Inserting $\mu_R=0$ proves claim \eqref{eq:2.12}.
Partial differentiation with respect to $\mu_R$ shows simplicity of $\mu_R=0$ at $B= B_0=(-1)^k$.
Implicit differentiation with respect to $B$ at $B= B_0,\ \mu_R=0$ shows
	\begin{equation}
	\begin{aligned}
	\varepsilon\mu'_R= (-1)^k &\mu'_R +\tfrac{1}{2}(-1)^k
	\cdot(-\pi\varepsilon)\mu'_R-1\,,\quad \text{i.e.}\\
	&\mu'_R(1-(\tfrac{\pi}{2}+B_0)\varepsilon) = B_0\,.
	\end{aligned}
	\label{eq:2.15}
	\end{equation}
For $k\in \mathbb{N}$, $\varepsilon=((k+\tfrac{1}{2})\pi)^{-1}$, the coefficient of $\mu'_R$ is positive, and the lemma is proved.
\end{proof}

Of course we may solve the real characteristic equation \eqref{eq:2.14} for $B=B(\mu_R)$, explicitly, to obtain
	\begin{equation}
	B=B(\mu_R)=\frac{1}{2} \cdot
	\frac{1+\exp (-\pi\varepsilon\mu_R)}
	{\varepsilon\mu_R+(-1)^k\exp (-\mu_R)}	\ .
	\label{eq:2.16}
	\end{equation}
For odd $k$, vanishing denominator indicates the unique positive real eigenvalue $\mu = \mu_R$ of the original problem \eqref{eq:1.3} without control.
For even $k$, the denominator is positive for all real $\mu_R$, because $\varepsilon=((k+\tfrac{1}{2} )\pi)^{-1} <e$ for $k\in \mathbb{N}_0\,$.
Moreover, $\pi \varepsilon \leq2/3$ for all $k \in \mathbb{N}$ implies $\lim B(\mu_R)=0$ for $\mu_R \rightarrow \pm \infty$.
Let
	\begin{equation}
	1 < B_{\max} := \max B(\mu_R) <\infty
	\label{eq:2.17}
	\end{equation}
denote the maximum over $\mu_R \in \mathbb{R}$, for even $k$.
Indeed $B_{\max}> 1=B_{0}=B(0)$, by lemma~\ref{lem:2.1}.
This allows us to determine the \emph{even/odd parity} of the total algebraic count $E(B)$ of all eigenvalues $\mu$, real or complex, with strictly positive real part. 
We write $E(B) \equiv 0\, (\mathrm{mod}\, 2)$ or $E(B) \equiv \mathrm{even}$, for even parity of $E(B)$, 
and analogously $E(B) \equiv 1\, (\mathrm{mod}\, 2) \equiv \mathrm{odd}$, for odd parity.

\begin{cor}
Let $k\in \mathbb{N}$ be odd.
Then the unstable parities \emph{(mod 2)} are given by
	\begin{equation}
	E(B) \equiv 
	\left\lbrace 
	\begin{aligned}
	&\mathrm{odd} 
	\qquad & -&\infty < B <-1\,,\\
	&\mathrm{even}
	\qquad &\mathrm{for}\qquad -&1\leq B<0\,,\\
	&\mathrm{odd}
	\qquad & &0 \leq B<+\infty\,.
	\end{aligned}
	\right.
	\label{eq:2.18}
	\end{equation}
For even $k\in \mathbb{N}$, the unstable parities \emph{(mod 2)} are given by
	\begin{equation}
	E(B) \equiv 
	\left\lbrace 
	\begin{aligned}
	&\mathrm{even}
	\qquad & -&\infty < B <0\,,\\
	&\mathrm{odd}
	\qquad &\mathrm{for}\qquad &0<B<1\,,\\
	&\mathrm{even}
	\qquad & &1\leq B< +\infty\,.
	\end{aligned}
	\right.
	\label{eq:2.19}
	\end{equation}
Pyragas regions $E(B)=0$ require even parity $0 \,(\mathrm{mod} \,2)$, of course.
For even $k$, they also require
	\begin{equation}
	B> B_{\max} >1\,,
	\label{eq:2.20}
	\end{equation}
in case $B >0$.
\label{cor:2.2}
\end{cor}

\begin{proof}[\textbf{Proof.}]
Since nonreal complex eigenvalues occur in complex conjugate pairs, the real eigenvalues alone determine the parity.
For odd $k$ and at $B = \pm \infty$, i.e. at vanishing control, instability by a simple positive real eigenvalue follows from the vanishing denominator in \eqref{eq:2.16}.
Lemma~\ref{lem:2.1} then implies claim \eqref{eq:2.18}.

For even $k$, real eigenvalues are absent if $B <0$ or $B>B_{\max}$; see \eqref{eq:2.16}, \eqref{eq:2.17}.
At $B=B_{\max}>1$, a pair of complex eigenvalues merges and forms a positive double real eigenvalue.
Decreasing $B$ further, one of these two positive eigenvalues becomes negative, at $B=B_0=1$, and the other real eigenvalue remains positive and simple; see lemma~\ref{lem:2.1}.
This proves claims \eqref{eq:2.19}, \eqref{eq:2.20}, and the corollary.
\end{proof}

We study the case of purely imaginary nonzero eigenvalues $\mu = i\tilde{\omega} \neq 0$, $\mu_R=0$ next.
The 2-scale characteristic equation \eqref{eq:2.4} then simplifies to \eqref{eq:2.8}, \eqref{eq:2.9}, i.e.
	\begin{equation}
	\chi_0 (\delta, \omega,\Omega, B) :=
	-i \chi(\varepsilon, \delta, \omega,\Omega, B, \mu_R=0)=0\,.
	\label{eq:2.21}
	\end{equation}
Strictly speaking, however, the frequency $\tilde{\omega}$ and the slow frequency $\tilde{\Omega}$ are still related by the linear hashing relation $\tilde{\Omega} = \varepsilon\tilde{\omega}$; see \eqref{eq:1.21} -- \eqref{eq:1.24}.
We clarify this relation next.

\begin{lem}
Consider $\tilde{\omega}>0,\ \Omega = \tilde{\Omega}-\Omega_m$ with $\Omega_m=2m+1$, and $-1<\Omega\leq0$ for $m=0$, but $-2<\Omega \leq0$ for $m\in \mathbb{N}$.
Then the hashing relation $\tilde{\Omega} = \varepsilon\tilde{\omega}$, with $\varepsilon= \omega_k^{-1}$, is equivalent to
	\begin{equation}
	\Omega = \varepsilon(\omega +\tfrac{\pi}{2}(1-(-1)^k-(-1)^m)-2\pi j)\,.
	\label{eq:2.22}
	\end{equation}
Here the representative $\omega\equiv\tilde{\omega}$ \emph{(mod 2$\pi$)} is chosen such that
	\begin{equation}
	-\tfrac{\pi}{2} \leq \omega <\tfrac{3}{2}\pi\,,
	\label{eq:2.23a}
	\end{equation}
$j\in \mathbb{N}$ is chosen such that
\begin{equation}
	j= 
	\left\lbrace 
	\begin{aligned}
	&0 
	&\qquad &\mathrm{for}\ k,m \ \mathrm{both}\ \mathrm{even}\,,\\
	&1
	&\qquad &\mathrm{otherwise}\,,
	\end{aligned}
	\right.
	\label{eq:2.23b}
	\end{equation}
holds at $\Omega =0$, and
	\begin{equation}
	\tilde{\omega} = \omega+2\pi(km+[(k+1)/2]+[(m+1)/2]-j)
	\equiv \omega \pmod{2\pi}\,.
	\label{eq:2.24}
	\end{equation}
\label{lem:2.3}
\end{lem}
\begin{proof}[\textbf{Proof.}]
The hashing relations $\tilde{\Omega}=\varepsilon\tilde{\omega}$ and \eqref{eq:2.22} are both affine linear in $\tilde{\omega},\ \omega$ with slope $\varepsilon$.
To show their equivalence, via $\tilde{\Omega}=\Omega +\Omega_m$ in \eqref{eq:2.5} and definition \eqref{eq:2.24} of $\omega$, we only have to check \eqref{eq:2.22} at $\varepsilon \tilde{\omega}=\tilde{\Omega} = \Omega_m$ and $\Omega =0$.
Indeed we obtain
	\begin{equation}
	\begin{aligned}
	\omega:
	&= 
	\tilde{\omega}-2\pi(km+[(k+1)/2]+[(m+1)/2]-j)=\\
	&=
	\varepsilon^{-1}\cdot \Omega_m-\pi(2km+2[(k+1)/2]+2[(m+1)/2])+2\pi j=\\
	&=
	(k+\tfrac{1}{2})\pi\cdot(2m+1)
		-\pi(2km+k+m)+2\pi j-\\
	& 
	\phantom{=(k+\tfrac{1}{2})\pi\cdot(2m+1)\;}
		-\pi((2[(k+1)/2]-k)+(2[(m+1)/2]-m))=\\
	&=
	\tfrac{\pi}{2}+2\pi j-\pi(\tfrac{1}{2}(1-(-1)^k)
	+\tfrac{1}{2}(1-(-1)^m))=\\
	&=
	2\pi j -\tfrac{\pi}{2}(1-(-1)^k-(-1)^m)\,,
	\end{aligned}
	\label{eq:2.25}
	\end{equation}
as required by \eqref{eq:2.22}.
The choice of $j$ in \eqref{eq:2.23b} ensures the ranges \eqref{eq:2.23a} for $\Omega < 0$ near $\Omega=0$.
This proves the lemma.
\end{proof}

\begin{figure}[t!]
\centering \includegraphics[width= 1.05\textwidth]{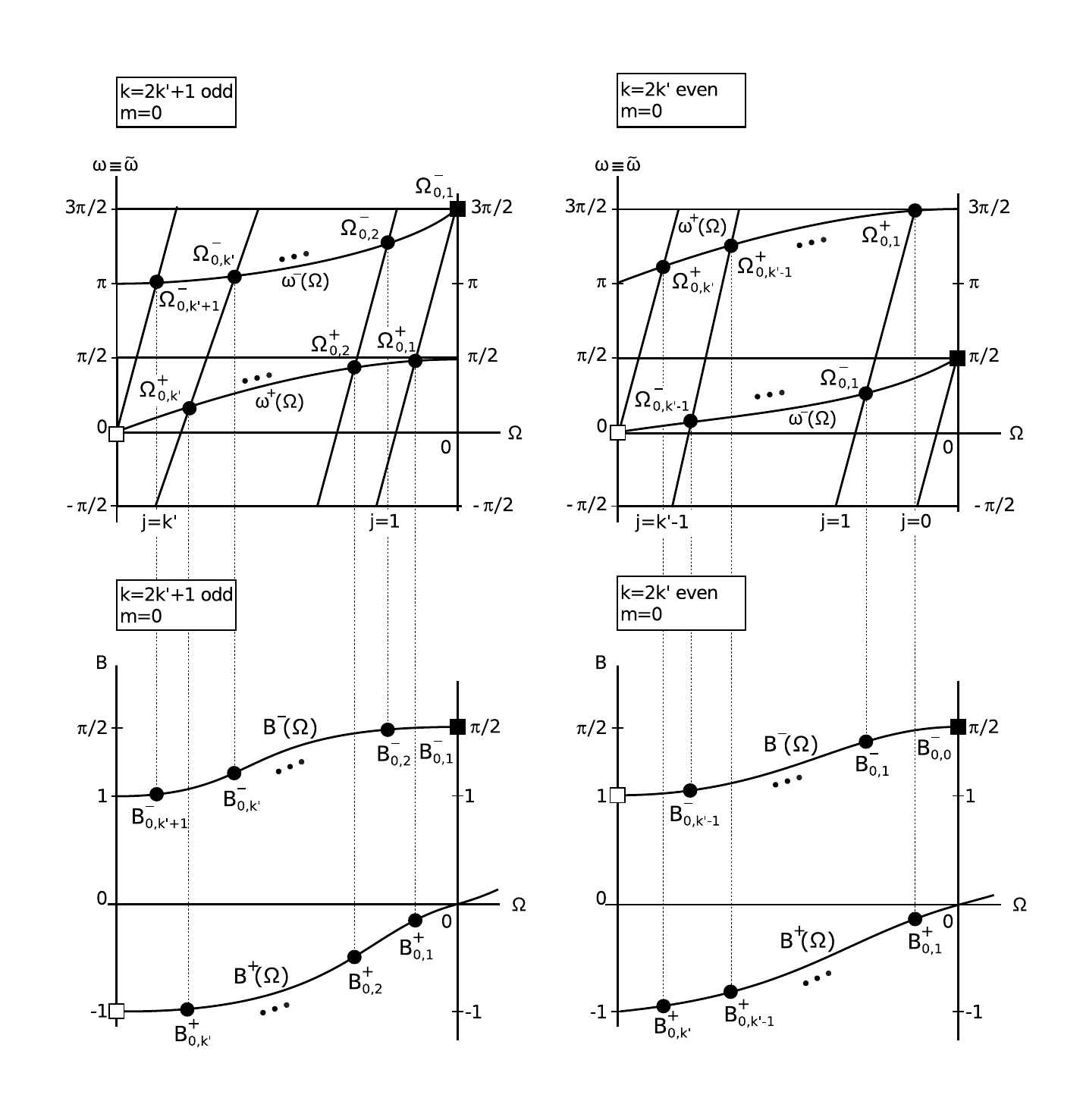}
\caption{\emph{
Purely imaginary eigenvalues $\mu = i \tilde{\omega} = i\tilde{\omega}_{0,j}^\pm$ and Hopf control parameters $B = B_{0,j}^\pm$ at $\varepsilon= \omega_k^{-1} = ((k+\tfrac{1}{2})\pi)^{-1}$.
The horizontal axis is $-1 \leq \Omega = \tilde{\Omega}-\Omega_0 \leq 0$, with $\Omega_0=1$.
Left: odd $k$.
Right: even $k$.
Top row: hashing $\tilde{\Omega} = \varepsilon \tilde{\omega}$ alias $\Omega = \varepsilon(\omega + \ldots)$ according to lemma~\ref{lem:2.3}, \eqref{eq:2.22} -- \eqref{eq:2.24} and \eqref{eq:3.36}.
Note how $\tilde{\Omega} = \tilde{\Omega}_{0,j}^\pm = \varepsilon\tilde{\omega}_{0,1}^\pm$ enumerate the Hopf frequencies defined by the intersections of the slanted hashing lines, of slope $1/\varepsilon$, with the relations $\tilde{\omega} = \tilde{\omega}^\pm (\tilde{\Omega})$, induced by the 2-scale characteristic equation; see lemma~\ref{lem:2.4} and \eqref{eq:3.40}.
Bottom row: the resulting control parameters $B=B_{0,j}^\pm = B^\pm (\tilde{\Omega}_{0,j}^\pm)$, also induced by the 2-scale characteristic equation according to lemma~\ref{lem:2.4}.
Solid dots $\bullet$ indicate transverse Hopf bifurcations, where the Hopf pair $\mu = \pm i\omega$ crosses towards the stable side for decreasing $|B|$, see lemma~\ref{lem:2.4}(iv).
Note the zero real eigenvalue $\square$ at ``Hopf'' frequency $\tilde{\omega} =0$, for $B=(-1)^k$.
Also note the non-crossing trivial Hopf pair $\blacksquare$ at $\mu= \pm i\omega_k$, which terminates the curves $B^-(\tilde{\Omega})$ at $\tilde{\Omega} = \varepsilon\omega_k=1$.
}}
\label{fig:3.1}
\end{figure}

The Hopf points $B\in \mathbb{R}$, where purely imaginary eigenvalues $\mu =i\tilde{\omega}>0$ arise, are therefore defined by the system of the 2-scale characteristic equation \eqref{eq:2.21} and the hashing \eqref{eq:2.22}, in the precise sense of lemma~\ref{lem:2.3}.
For the moment we ``forget'' hashing and address the complex 2-scale equation \eqref{eq:2.21} first, in its own right.
See also the bottom rows of figs.~\ref{fig:3.1} and \ref{fig:3.2}.

\begin{lem}
The 2-scale characteristic equation $\chi_0(\delta, \omega, \Omega, B)=0$ for purely imaginary eigenvalues, i.e. equation \eqref{eq:1.26}, \eqref{eq:2.21}, is equivalent to the system
	\begin{equation}
	\begin{aligned}
	0
	&=H(\delta, \Omega,\omega)
	:=\tilde{\Omega}
	\sin (\tfrac{\pi}{2}\Omega)-(-1)^k
	\cos(\omega-\tfrac{\pi}{2}\Omega)=\\
	&=(-1)^{m+1}(\tilde{\Omega}\cos (\tfrac{\pi}{2}\tilde{\Omega})
	-(-1)^k\sin(\omega-\tfrac{\pi}{2}\tilde{\Omega}))
	\end{aligned} \label{eq:2.27}\end{equation}\\[-1cm]
	\begin{equation}
	B
	= (-1)^k \sin^2(\tfrac{\pi}{2}\Omega)/\cos \omega 
	= (-1)^k \cos^2(\tfrac{\pi}{2}\tilde{\Omega})/\cos \omega\hphantom{-}
	\label{eq:2.28}
	\end{equation}
with parameter $\delta =1/(2m+1),\ m\in \mathbb{N}_0$.
As always, the case of even $k$ results from odd $k$ by addition of $\pi$ to $\omega\pmod{2\pi}$.
Here we have also used the previous notation\\[-0.5cm]
	\begin{equation}
	\tilde{\Omega} = \Omega +\Omega_m = \Omega+(2m+1)=\Omega +1/\delta\,;
	\label{eq:2.29}
	\end{equation}
see \eqref{eq:2.5}.
Eliminating $\omega \in (-\tfrac{\pi}{2},\ 3\tfrac{\pi}{2})$ we obtain the quadratic relation
	\begin{equation}
	0=Q(\delta,\Omega,B):= (\tilde{\Omega}^2-1)B^{2}
	+\tilde{\Omega} \sin (\pi\tilde{\Omega})B
	+\cos^2(\tfrac{\pi}{2}\tilde{\Omega})\,.
	\label{eq:2.30}
	\end{equation}
The discriminant $D$ of \eqref{eq:2.30} is given by
	\begin{equation}
	D= \cos^2(\tfrac{\pi}{2}\tilde{\Omega})
	\cdot (1-(\tilde{\Omega} \cos (\tfrac{\pi}{2}\tilde{\Omega}))^2)\,,
	\label{eq:2.31}
	\end{equation}
and the explicit solutions $B=B^\pm$ of \eqref{eq:2.30} are
	\begin{equation}
	B^\pm = (\tilde{\Omega}^2-1)^{-1} 
	(-\tfrac{1}{2}\tilde{\Omega}\sin (\pi\tilde{\Omega})\pm \sqrt{D})\,.
	\label{eq:2.32}
	\end{equation}
For $m\in \mathbb{N}$ and $\Omega_m=2m+1$, let ${\tilde{\underline{\Omega}}}_m\in (2m,\Omega_{m})$ and $\tilde{\Omega}_m^{\max} \in (\Omega_m, 2m+2)$ denote the unique solutions $\tilde{\Omega}$ of $D=0$, i.e. of
	\begin{equation}
	\tilde{\Omega} \cdot (-1)^m\cos (\tfrac{\pi}{2}\tilde{\Omega})=1\,,
	\label{eq:2.34}
	\end{equation}
in the respective intervals.
In terms of $\Omega = \tilde{\Omega}-\Omega_m$ and $0<\delta = \Omega_m^{-1}$ this defines unique solution branches $(\Omega, B^\pm)$ of \eqref{eq:2.30} with
	\begin{equation}
	\begin{aligned}
	B^+<0&<B^-\,,\qquad &\mathrm{for}\ 
	0=:\tilde{\underline{\Omega}}_0<\tilde{\Omega}<&\ \Omega_0=1\,, &m=0;\\
	B^\pm<0&\,,\qquad &\mathrm{for}\ \quad \ \
	\tilde{\underline{\Omega}}_m<\tilde{\Omega}<&\ \Omega_m\,, &m\geq 1\,;\\
	0&<B^\pm\,,\qquad \quad &\mathrm{for}\ \quad \ \
	\Omega_m<\tilde{\Omega}<&\ \tilde{\Omega}_m^{\max}\,, &m\geq 1\,.
	\end{aligned}
	\label{eq:2.33}
	\end{equation}
\label{lem:2.4}
\end{lem}

\begin{figure}[t!]
\centering \includegraphics[width= 1 \textwidth]{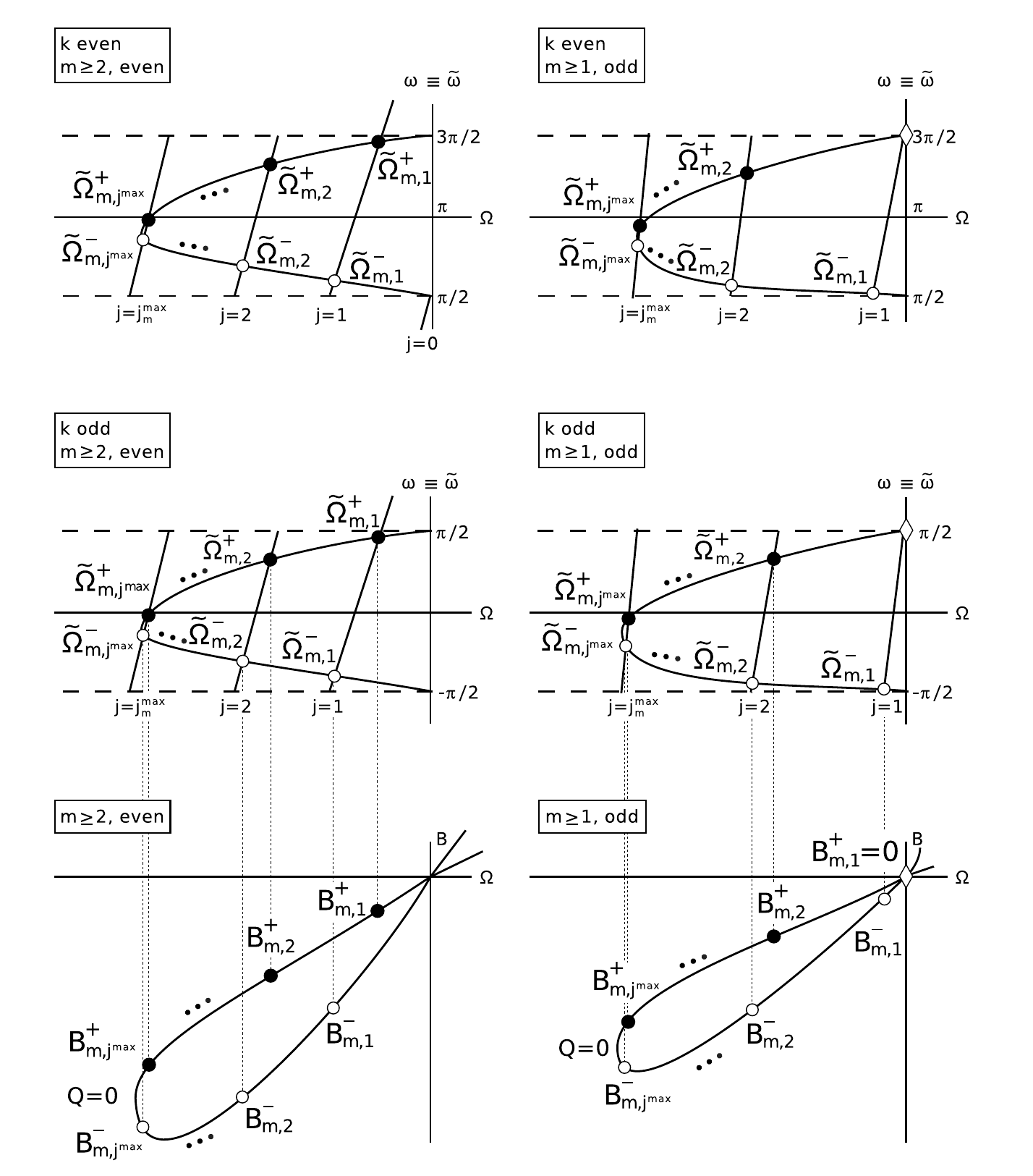}
\caption{\emph{
Purely imaginary eigenvalues $\mu = i\,\tilde{\omega} = i\,\tilde{\omega}_{m,j}^\pm$, two top rows, and Hopf control parameters $B= B_{m,j}^\pm <0$, bottom rows, at $\varepsilon= \omega_k^{-1} = ((k+\tfrac{1}{2})\pi)^{-1}$.
The horizontal axis is $-1<\underline{\Omega}_m \leq \Omega = \tilde{\Omega}-\Omega_m\leq0$ with $\Omega_m=2m+1$.
Left: even $m$.
Right: odd $m$.
Layout and legends as in figure~\ref{fig:3.1}.
Again, solid dots $\bullet$ indicate transverse Hopf stabilization towards smaller control parameters $|B|$, i.e. towards larger $B <0$, at $B_{m,j}^+$.
Circles $\circ$, in contrast, indicate transverse Hopf destabilization towards the same side, at $B_{m,j}^-$.
Note how destabilization by each $B_{m,j}^- <0$ is annihilated when $B<0$ increases through the subsequent stabilization at $B_{m,j}^+ <0$.
See theorem~\ref{thm:3.4}(iv). 
Only for odd $m$ and $j=1$, the subsequent stabilization at $B_{m,1}^+=0, \ \diamond$, fails to occur at any finite control amplitude $\beta = 1/b <0$.
}}
\label{fig:3.2}
\end{figure}

\begin{proof}[\textbf{Proof.}]
For vanishing real part $\mu_R=0$, the 2-scale characteristic equation \eqref{eq:2.4} simplifies to
	\begin{equation}
	0=i \chi_0(\delta,\omega,\Omega,B)
	= i\tilde{\Omega}-(-1)^k e^{i\omega}
	-\tfrac{i}{B}\sin(\tfrac{\pi}{2}\Omega) e^{i\tfrac{\pi}{2}\Omega}\,;
	\label{eq:2.35}
	\end{equation}
see \eqref{eq:2.8}, \eqref{eq:2.21}.
To prove \eqref{eq:2.27}, we multiply by $\exp(-i\tfrac{\pi}{2}\Omega)$ and take real parts.
To prove \eqref{eq:2.28} we take real parts directly.
To eliminate $\omega$, as in \eqref{eq:2.30}, we solve $\chi_0=0$ in \eqref{eq:2.35} for the only term $\exp (i\omega)$ which contains $\omega$, and calculate the square of the absolute values of both sides.
The remaining claims \eqref{eq:2.31} -- \eqref{eq:2.33} concerning the quadratic relation \eqref{eq:2.30} are plain high school calculus.
This proves the lemma.
\end{proof}

In the following analysis we will skip the third case $m\geq 1,\ B^\pm>0,\ \Omega_m<\tilde{\Omega}<\tilde{\Omega}_m^{\max}$ of \eqref{eq:2.33} which is completely analogous to the second case $\underline{\tilde{\Omega}}_m<\tilde{\Omega}<\Omega_m,\ B^\pm <0$.
Indeed that third case will turn out to be irrelevant anyway, in section~\ref{sec:3}; see corollary \ref{cor:3.6}.

\begin{lem}
The 2-scale relation \eqref{eq:2.27} between slow and fast frequencies $\Omega$ and $\omega$ can be solved for $\Omega = \Omega(\delta,\omega)$, implicitly, and for the inverse function $\omega = \omega(\delta,\Omega)$, explicitly:
	\begin{equation}
	\omega = \omega^\pm:\equiv \tfrac{\pi}{2}\Omega \pm 
	\arccos (-\tilde{\Omega}\sin (\tfrac{\pi}{2}\Omega)) \quad
	\pmod{2\pi}\,,
	\label{eq:2.36}
	\end{equation}
for $k$ odd.
Even $k$ require addition of $\pi \pmod{2\pi}$.

For $m=0$, where $ \delta=1$ and $0<\tilde{\Omega} = \Omega +1<1$, both functions $\omega^\pm$ have strictly positive and bounded derivatives with respect to $\tilde{\Omega}$ or $\Omega$, equivalently, in the interior domain.
For odd $k$, their boundary values and ranges are, accordingly,
	\begin{equation}
	\begin{aligned}
	\omega^+ \in [0,\tfrac{1}{2}\pi]\,,
	\quad &\mathrm{with} \ \omega^+=0, \tfrac{1}{2}\pi \ \mathrm{at}\
	\tilde{\Omega} =0,1\,;\\
	\omega^- \in [\pi,\tfrac{3}{2}\pi]\,,
	\quad &\mathrm{with} \ \omega^-=\pi, \tfrac{3}{2}\pi \ \mathrm{at}\
	\tilde{\Omega} =0,1\,.
	\end{aligned}
	\label{eq:2.37}
	\end{equation}
The ranges and boundary values are interchanged for even $k$.
See the top row of fig.~\ref{fig:3.1}.

Let $m\geq 1$, where $0<\delta = 1/(2m+1)\leq 1/3$, and consider $\underline{\tilde{\Omega}}_m \leq \tilde{\Omega}\leq 2m+1 = \Omega_m$; see \eqref{eq:2.34}.
Then we observe ranges
	\begin{equation}
	\begin{aligned}
	\omega \in [-\tfrac{\pi}{2},\tfrac{\pi}{2}]\,,
	\quad &\mathrm{for}\ \mathrm{odd} \ k\geq 1\,,\\
	\omega \in [\tfrac{\pi}{2},\tfrac{3}{2}\pi]\,,
	\quad &\mathrm{for}\ \mathrm{even} \ k\geq 2\,.
	\end{aligned}
	\label{eq:2.38}
	\end{equation}
Moreover, the implicit inverse function $\Omega = \Omega(\delta,\omega)$ is strictly piecewise monotone in $\omega$ with unique local and global minimum
	\begin{equation}
	-1<\underline{\tilde{\Omega}}_m -\Omega_m =:
	\underline{\Omega}_m \leq \Omega \leq 0
	\label{eq:2.39}
	\end{equation}
and boundary values $\Omega =0$ at $\omega\equiv \pm \tfrac{\pi}{2}$.
The minimal value $\underline{\Omega}_m$ occurs at
	\begin{equation}
	\begin{aligned}
	\omega = \underline{\omega} &= \tfrac{\pi}{2} \underline{\Omega}_m
	\quad &\mathrm{for}\ &k\ \mathrm{odd}\,,\\
	\omega = \underline{\omega} &= \pi +\tfrac{\pi}{2} \underline{\Omega}_m
	\quad &\mathrm{for}\ &k\ \mathrm{even}\,.
	\end{aligned}
	\label{eq:2.40}
	\end{equation}
The two explicit branches $\omega = \omega^\pm(\delta,\Omega)$  possess
strictly nonzero, but only locally bounded, derivatives with respect to $\tilde{\Omega}$ or $\Omega$, equivalently, in the domain 
$\underline{\Omega}_m < \Omega \leq 0$. 
They
merge at $\Omega=\underline{\Omega}_m$, where the discriminant of the quadratic $B$-relation \eqref{eq:2.30} vanishes.
See the two upper rows of fig.~\ref{fig:3.2}.
\label{lem:2.5}
\end{lem}

\begin{proof}[\textbf{Proof.}]
As always, we may consider odd $k$, without loss.
The explicit solutions $\omega = \omega^\pm$ of \eqref{eq:2.36} follow directly from the first line of the 2-scale equation \eqref{eq:2.27}.
Recall that nonnegative discriminants $D$ in \eqref{eq:2.31}, \eqref{eq:2.32} require $|\tilde{\Omega} \sin (\tfrac{\pi}{2} \Omega)|=|\tilde{\Omega} \cos (\tfrac{\pi}{2} \tilde{\Omega})|\leq1$; see also \eqref{eq:derivation}.
Hence $-1\leq \Omega\leq0$ implies $-1\leq \tilde{\Omega}\sin (\tfrac{\pi}{2} \Omega)\leq 0$. 
In particular $\arccos(- \tilde{\Omega} \sin (\tfrac{\pi}{2} \Omega) )  \in [0,\tfrac{\pi}{2}]$, for all $m \in \mathbb{N}_{0}$.
This proves claim \eqref{eq:2.36} and the range claims \eqref{eq:2.37}, \eqref{eq:2.38}.
Moreover the functions $\omega=\omega^{\pm}(\Omega)$ are differentiable with bounded derivatives, except for the vertical tangent at the discriminant loci $\Omega=\underline{\Omega}_{m}$, $m \geq1$.

For the inverse function $\Omega=\Omega(\omega)$, we study the monotonicity  claims, for all $m$,  and the minimizer claims, for $m\geq1$, next.
Here we suppress $\delta$, for a while.
Recall $H(\Omega,\omega)=0$, from \eqref{eq:2.27}.
With the abbreviations
	\begin{equation}
	\begin{aligned}
	S&:=\sin(\tfrac{\pi}{2}\Omega) 
	= -(-1)^m\cos (\tfrac{\pi}{2}\tilde{\Omega})\,,\\
	C&:=\cos(\tfrac{\pi}{2}\Omega) 
	= (-1)^m\sin (\tfrac{\pi}{2}\tilde{\Omega})\,,\\
	s&:=\sin(\omega-\tfrac{\pi}{2}\Omega) \,,\qquad
	c:= \cos (\omega-\tfrac{\pi}{2}\Omega)\,,
	\end{aligned}
	\label{eq:2.42}
	\end{equation}
$H$ and its partial derivatives, for odd $k$, are
	\begin{equation}
	\begin{aligned}
	&H& &= \quad \tilde{\Omega}S+c\ ,\\
	&H_\Omega& &= \quad S+\tfrac{\pi}{2}\tilde{\Omega}C+\tfrac{\pi}{2}s\ ,\\
	&H_\omega& &= \quad -s \ .
	\end{aligned}
	\label{eq:2.43}
	\end{equation}
Elementary  arguments show that $H=0$ is a regular value of $H$.
Let $\dot{\Omega}$ denote the derivative of $\Omega(\omega)$.
By implicit differentiation of $H(\Omega, \omega)=0$ with respect to $\omega$, we obtain
\begin{equation}
	-H_\Omega \dot{\Omega} =  H_\omega \,.
	\label{eq:2.44} 
	\end{equation}
Suppose $\dot{\Omega} =0$ at $\omega=\underline{\omega}$.
Then
	\begin{equation}
	0= H_\omega =-s=-\sin (\omega-\tfrac{\pi}{2}\Omega)
	\label{eq:2.57}
	\end{equation}
at $\Omega = \Omega(\underline{\omega})$ implies $\omega =\underline{\omega}\equiv \tfrac{\pi}{2}\Omega\pmod{\pi}$.
Insertion into $H(\Omega,\underline{\omega})=0$ implies
	\begin{equation}
	\pm 1= -c= \tilde{\Omega}S 
	=-(-1)^m\tilde{\Omega}\cos(\tfrac{\pi}{2}\tilde{\Omega})\,,
	\label{eq:2.58}
	\end{equation}
using \eqref{eq:2.42}.
Since $S<0<\tilde{\Omega}$, we obtain $c=+1$.
This implies $m\geq1$ and
	\begin{equation}
	\tilde{\Omega} = \underline{\tilde{\Omega}}_m
	\label{eq:2.59}
	\end{equation}
as defined in \eqref{eq:2.34}.
For $m=0$, in fact, $0\leq \tilde{\Omega}\leq 1$ prevents any solution of \eqref{eq:2.58}.
This proves the strong monotonicity claims for $m=0$, and completes the proof of lemma~\ref{lem:2.5}.
\end{proof}

\begin{lem}
The functions $B^\pm=B^\pm(\delta, \Omega)$ for the control parameter $B$ in \eqref{eq:2.28}, \eqref{eq:2.30},\eqref{eq:2.32} have the following properties.

For $m=0$, where $\delta=1$ and $0<\tilde{\Omega}= \Omega+1<1$, we have
	\begin{equation}
	B^\pm(\delta,\Omega)=(-1)^k
	\frac{\sin^2(\tfrac{\pi}{2}\Omega)}{\cos\omega^\pm(\delta, \Omega)}.
	\label{eq:2.61}
	\end{equation}
Here $\omega^\pm$ have been defined in \eqref{eq:2.36}.
Moreover $B^+$ increases strictly with respect to $\Omega$, or $\tilde{\Omega}$,
	\begin{equation}
	\begin{aligned}
	&B^+<0<B^-
	\quad &\mathrm{for}\quad &0<\tilde{\Omega}<1\,;\\
	&B^+=-1,0
	\quad &\mathrm{for}\quad &\tilde{\Omega}=0,1\,;\\
	&B^-=1,\,\pi/2
	\quad &\mathrm{for}\quad &\tilde{\Omega}=0,1\,.
	\end{aligned}
	\label{eq:2.62}
	\end{equation}
See the bottom row of fig.~\ref{fig:3.1}.

Let $m\geq 1$, where $0<\delta <1/(2m+1)\leq 1/3$ and $\underline{\tilde{\Omega}}_m\leq \tilde{\Omega} < 2m+1=\Omega_m$; see \eqref{eq:2.33}.
Then
	\begin{equation}
	\begin{aligned}
	&B^-<B^+<0 &\mathrm{for} \quad &\underline{\tilde{\Omega}}_m<\tilde{\Omega}< \Omega_m\,;\\
	&B^\pm = -\tfrac{1}{2}(\underline{\tilde{\Omega}}_m^2-1)^{-1}
	\underline{\Omega}_m\sin(\pi\underline{\tilde{\Omega}}_m)\,,
	\quad &\mathrm{for}\quad &\tilde{\Omega}
	=\underline{\tilde{\Omega}}_m\,;\\
	&B^\pm=0\,,
	\quad &\mathrm{for}\quad &\tilde{\Omega}= \Omega_m\,.
	\end{aligned}
	\label{eq:2.63}
	\end{equation}
In terms of \eqref{eq:2.40}, the branches $B^\pm$ are parametrized over $\omega$, instead of $\Omega$, as follows:
	\begin{equation}
	\begin{aligned}
	&B^+= B(\delta,\omega)\,,
	\quad &\mathrm{for}\quad &\omega >\underline{\omega}:=
	\tfrac{\pi}{2}\underline{\Omega}_m\,;\\
	&B^-= B(\delta,\omega)\,,
	\quad &\mathrm{for}\quad &\omega <\underline{\omega}:=
	\tfrac{\pi}{2}\underline{\Omega}_m\,.
	\end{aligned}
	\label{eq:2.64}
	\end{equation}
Moreover $B^+$ is strictly increasing with respect to $\omega$, or $\Omega$ alias $\tilde{\Omega}$.

For $B^-$ and $m\geq 1$ we encounter a zero derivative $\partial_\omega B^-$ at $\omega= \omega_*,\ \Omega= \Omega_*$ if, and only if,
	\begin{equation}
	0=d(\Omega,\omega):=\sin^2(\tfrac{\pi}{2}\Omega)\sin\omega+
	\tfrac{\pi}{2}\cos\omega\sin(\omega-\pi\Omega)\,.
	\label{eq:2.64a}
	\end{equation}
Critical points, and in particular the minimum, of $\omega \mapsto B$ (and of $B^-$) occur at certain $\omega= \omega_*,\ \Omega=\Omega_*$ which satisfy
	\begin{equation}
	\pi\Omega_{*} < \omega_{*} < \tfrac{1}{2}\pi\Omega_{*} <0\,.
	\label{eq:2.64b}
	\end{equation}
In particular $\partial_\omega B^-<0$ for $-\tfrac{\pi}{2}<\omega\leq\pi\Omega_{*}$.
Moreover $B>-1$.
\label{lem:2.6}
\end{lem}

\begin{proof}[\textbf{Proof.}]
Again we consider odd $k$, without loss.
We suppress the parameter $\delta$ and differentiate \eqref{eq:2.28} with respect to $\omega$, implicitly, analogously to the proof of lemma~\ref{lem:2.5}:
	\begin{equation}
	\dot{B}= -
	(\tfrac{\pi}{2}\sin(\pi\Omega)\cos\omega\cdot
	\dot{\Omega}+\sin^2(\tfrac{\pi}{2}\Omega)\sin\omega)/\cos^2\omega\,.
	\label{eq:2.65}
	\end{equation}
We invoke lemma~\ref{lem:2.5}.

Consider $m=0$, first.
Then the sign of the right hand side of \eqref{eq:2.61} is
	\begin{equation}
	-\text{sign}\,\cos\omega^\pm = \mp 1
	\label{eq:2.66}
	\end{equation}
by \eqref{eq:2.37}.
These signs agree with our definition \eqref{eq:2.32} of $B^\pm$, for $m=0$, because $0<\tilde{\Omega}<1$ implies $B^+<B^-$, there.
More specifically, we have shown $B^+<0<B^-$.

To settle monotonicity of $B^+$, for all $m$, we keep considering odd $k$, without loss.
We aim to show $\dot{B}^{+}>0$ in the interior domain of definition.
We rewrite \eqref{eq:2.65} as
	\begin{equation}
	\dot{B}^+= -S (\pi C\cos\omega\cdot\dot
	{\Omega}+S\sin\omega)/\cos^2\omega\,,
	\label{eq:2.67}
	\end{equation}
with $C>0>S$ for $-1<\Omega <0$; see \eqref{eq:2.42} for this notation.
We also recall $|\omega |<\tfrac{\pi}{2},\ \cos \omega >0,\ \dot{\Omega}>0$ and $s=\sin(\omega-\tfrac{\pi}{2}\Omega)>0$, for $B^+$ and odd $k$; see the proof of lemma~\ref{lem:2.5}.
Therefore \eqref{eq:2.67} implies $\dot{B}^+>0$, if $\sin\omega \leq 0$.
It remains to show interior positivity of $\dot{B}^{+}$ for $\sin \omega>0$, i.e. for $0<\omega <\tfrac{\pi}{2}$.

Suppose $\dot{B}=0$.
We derive the relation \eqref{eq:2.64a} at such a zero, first.
We differentiate $\chi_0=0$ in \eqref{eq:2.8}, \eqref{eq:2.9} with respect to $\omega$, implicitly, to obtain
	\begin{equation}
	0=\dot{\Omega} +e^{i\omega}-
	\tfrac{\pi}{2B} e^{i\pi\Omega}\,\dot{\Omega}\,.
	\label{eq:2.71}
	\end{equation}
We have used the assumption $\dot{B} =0$ here.
We multiply \eqref{eq:2.71} by the complex conjugate coefficient of $\dot{\Omega}$ and take imaginary parts to eliminate the derivative 
$\dot{\Omega}$:
	\begin{equation}
	\begin{aligned}
	0 &= \mathrm{Im}(e^{i\omega}(1-\tfrac{\pi}{2B}e^{-i\pi\Omega}))=\\
	&= \sin\omega-\tfrac{\pi}{2B}\sin(\omega-\pi\Omega)\,.
	\end{aligned}
	\label{eq:2.72}
	\end{equation}
Substitution of $B$ from \eqref{eq:2.28} and multiplication by the resulting denominator $\sin^2(\tfrac{\pi}{2}\Omega)$ proves claim \eqref{eq:2.64a}.

We can now prove interior positivity $\dot{B}^{+}>0$ for the remaining case $0<\omega<\tfrac{\pi}{2}$.
Suppose  $\dot{B}^{+}=0$, indirectly. 
We then use  trigonometric addition and the abbreviations of \eqref{eq:2.42} to rewrite \eqref{eq:2.64a} as 
\begin{equation}
0=d(\Omega,\omega)=S^{2} \sin \omega + \tfrac{\pi}{2} \cos \omega \ (-Sc+Cs).
\label{eq:2.68}
\end{equation}
To reach a contradiction, we check positivity of each individual term.
Our assumption $0<\omega<\tfrac{\pi}{2}$ implies $\sin\omega>0,\ \cos\omega>0$.
Moreover, $-1<\Omega<0$ implies $S^2>0,\ C>0$.
By subtraction, we also obtain $s>0$ because $0<\omega-\tfrac{\pi}{2}\Omega<\pi$.
It only remains to check positivity of $-Sc$.
Indeed $H=0$ in \eqref{eq:2.27}, \eqref{eq:2.43} implies $-Sc = \tilde{\Omega} S^2 >0$, since $S^2>0$ and $\tilde{\Omega}>0$.
Hence all terms on the right hand side of \eqref{eq:2.68} are strictly positive.
This contradiction establishes $\dot{B}^{+}>0$ in all cases.

To prove claim \eqref{eq:2.64b}, we first observe that $\omega <\tfrac{\pi}{2} \underline{\Omega}_m < \tfrac{\pi}{2}\Omega<0$ holds all along the $\omega^-$-branch, because $-\tfrac{\pi}{2}<\omega<\underline{\omega} = \tfrac{\pi}{2}\underline{\Omega}_m$ defines the domain of $B^-$; see lemma~\ref{lem:2.5}.
To show $\omega_* > \pi\Omega_*$ at $\dot{B}=0$, indirectly, suppose $-\tfrac{\pi}{2}<\omega_* \leq\pi\Omega_*<0$.
We then claim $d(\Omega_*,\omega_*)<0$.
Indeed
	\begin{equation}
	\begin{aligned}
	d(\Omega_*,\omega_*) &=
	S_*^{2}\sin\omega_*+
	\tfrac{\pi}{2}\cos\omega_*\sin(\omega_*-\pi\Omega_*)\leq\\
	&\leq S_*^{2}\sin\omega_* <0\,.
	\end{aligned}
	\label{eq:2.73}
	\end{equation}
This contradiction proves \eqref{eq:2.64b}.

It remains to prove $B>-1$, for $m\geq1$. 
By continuity of $B$, and because $B^{\pm}=0$ at $\tilde{\Omega}=\Omega_{m}=2m+1$, it is sufficient to show $B^{-}\neq-1$, indirectly.
Suppose $B^{-}=-1$. 
Then $S^{2}=\cos(\omega)$, by \eqref{eq:2.28}.
Therefore \eqref{eq:2.27} implies
\begin{equation}
	\begin{aligned}0=H/S & = \tilde{\Omega}+\cos(\omega-\tfrac{\pi}{2}\Omega)/S=\\ 
& = \tilde{\Omega}+(C \cos\omega+S \sin\omega)/S=\\
& = \tilde{\Omega}+CS + \sin\omega \geq \tilde{\Omega}-2>0\,,
\end{aligned}
	\label{eq:2.74}
	\end{equation}
since $\tilde{\Omega} > \Omega_{m}-1=2m\geq 2$.
This proves $B>-1$, and completes the proof of the lemma.
\end{proof}


\section{Control-induced Hopf bifurcation}
\label{sec:3}

In absence of control, i.e. in the limit $b=2\varepsilon B\rightarrow\pm\infty$, the original delay equation \eqref{eq:1.4} possesses a trivial simple Hopf eigenvalue
	\begin{equation}
	\mu=i\tilde{\omega}= i\omega_k=i\varepsilon^{-1}
	\label{eq:3.1}
	\end{equation}
of the characteristic equation \eqref{eq:2.1}, at the original parameter
	\begin{equation}
	\lambda = \lambda_k=-(-1)^k/\varepsilon\,.
	\label{eq:3.2}
	\end{equation}
The (scaled) control parameter $B$ induces further purely imaginary Hopf eigenvalues $\mu=\mu_R+i\tilde{\omega},\ \mu_R=0$, of the 2-scale characteristic equation \eqref{eq:2.4}.
We fix and suppress $\delta=\Omega_m^{-1}$ in this section and rewrite \eqref{eq:2.1} -- \eqref{eq:2.5} as
	\begin{equation}
	\begin{aligned}
	0=
	\psi(\mu,\varepsilon,B):&=
	-\varepsilon B\mu-(-1)^k B e^{-\mu}+
	\tfrac{1}{2}(1+e^{-\pi\varepsilon\mu})=\\
	&=-\varepsilon B\mu_R-iB\tilde{\Omega}-
	(-1)^k B e^{-\mu_R-i\tilde{\omega}}+
	\tfrac{1}{2}(1+e^{-\pi\varepsilon\mu_R-i\pi\tilde{\Omega}})\,.
	\end{aligned}
	\label{eq:3.3}
	\end{equation}
As before we have abbreviated $\tilde{\Omega}=\Omega_m+\Omega = \delta^{-1}+\Omega$ here, and $\mu=\mu_R+i\tilde{\omega}$ with $\tilde{\omega}\equiv \omega\pmod{2\pi},\ \varepsilon\tilde{\omega}=\tilde{\Omega}$.
We also recall the hashing relation \eqref{eq:1.24}, i.e.
	\begin{equation}
	0=
	h(\tilde{\omega}, \Omega,\varepsilon)=
	-\tilde{\Omega}+\varepsilon\tilde{\omega}=
	-\Omega+\varepsilon(\omega+\tfrac{\pi}{2}(1-(-1)^k-(-1)^m)-2\pi j)\,,
	\label{eq:3.4}
	\end{equation}
for$-\tfrac{\pi}{2}\leq\omega<\tfrac{3}{2}\pi\,$; see lemma~\ref{lem:2.3}.
In other words, Hopf bifurcation is governed by the three real equations \eqref{eq:3.3}, \eqref{eq:3.4} for vanishing $(\Psi,h) \in \mathbb{C} \times \mathbb{R}$, in the five not quite independent real variables $(\mu, \Omega, \varepsilon,B) \in \mathbb{C}\times \mathbb{R}^3$.

In trapping lemma~\ref{lem:3.1} we observe absence of nontrivial eigenvalues $\mu = \mu_R+i\tilde{\omega}$ with imaginary parts $\tilde{\omega}\equiv \pm \tfrac{\pi}{2} \pmod{2\pi}$.
This traps imaginary parts in eigenvalue \emph{strips}:
an old and efficient idea already present in \cite{beco63,nu78}.
It establishes the crucial sequences of Hopf bifurcations at scaled control parameters
	\begin{equation}
	B=B_{m,j}^\pm\,,
	\label{eq:3.5}
	\end{equation}
where $m,j$ label specific strips of the Hopf eigenvalues $\mu =i\tilde{\omega}$, with $\Omega_{m-1} < \tilde{\Omega} = \varepsilon\tilde{\omega}<\Omega_m=2m+1$.

We also observe how eigenvalues $\mu$ cannot appear from, or disappear towards, $\mathrm{Re} \ \mu=+\infty$.
Proposition~\ref{prop:3.2} examines eigenvalues at vanishing control $b=2B\varepsilon=\pm \infty$ to establish simplicity of eigenvalues, in each strip.
With some estimates for Jacobian determinants involving $\Psi$ and $h$, in proposition~\ref{prop:3.3}, we establish the transverse crossing directions of the simple Hopf eigenvalues $\mu$, as $B$ increases through $B_{m,j}^\pm <0$; see the central crossing theorem~\ref{thm:3.4} of the present section.
In fact we observe a gain of stability, i.e.  decrease of the unstable dimensions $E=E(B)$ by $2$, at $B=B_{m,j}^+<0$, and destabilization at $B_{m,j}^-<0$.
Corollary~\ref{cor:3.5} concludes that Pyragas stabilization is impossible, for $B>0$.
Corollary~\ref{cor:3.6} concludes that unstable eigenvalues in the complex $(m,j)$-strips are present if, and only if,
	\begin{equation}
	B_{m,j}^- < B< B_{m,j}^+\,.
	\label{eq:3.6}
	\end{equation}
Corollary~\ref{cor:3.7} studies the case $m=0$ of slow frequencies $0<\tilde{\Omega} = \varepsilon\tilde{\omega}<1$, as well as the simplest case $m=1$.
It concludes stability of the strip $m=0,\ j=1$ for $B_{0,1}^+<B<0$, but instability of the strip $m=1,\ j=1$ for $B_{1,1}^-<B<0$.
With the orderings of $B_{m,j}$ with respect to $j$, for each fixed $m\leq 1$, as collected in proposition~\ref{prop:3.8} we arrive at the conclusion of the present section, in corollary~\ref{cor:3.9}:
the region $\cal{P}$:= $\lbrace B\ |\ B_{0,1}^+< B<B_{1,1}^-\rbrace$ is a nonempty Pyragas region, provided that
	\begin{equation}
	B_{m,j_{m}+1}^+ <B_{0,1}^+< B_{1,1}^-<B_{m,j_m}^-
	\label{eq:3.7}
	\end{equation}
holds for $j_m:= [  (m+1)/2]$  and all $m\geq1$.
The delicate ordering \eqref{eq:3.7} will only be established in sections~\ref{sec:4} and \ref{sec:5} below.

\begin{lem}
\label{lem:3.1}
For any fixed $\varepsilon >0$, consider strictly complex eigenvalues $\mu = \mu_R +i\tilde{\omega}$, i.e. solutions $\mu \in \mathbb{C} \smallsetminus \mathbb{R}$ of the characteristic equation \eqref{eq:3.3}.

(i)~Assume
	\begin{equation}
	\mu_R \geq 0>B\,.
	\label{eq:3.8a}
	\end{equation}
Then the only eigenvalues $\mu=\mu_R+i\tilde{\omega}$ such that
	\begin{equation}
	0<\tilde{\omega} = \mathrm{Im} \,\mu \equiv \tfrac{\pi}{2} \pmod{\pi}
	\label{eq:3.8b}
	\end{equation}
are the trivial eigenvalues $\mu = i\omega_{\tilde{k}}=i(\tilde{k}+\tfrac{1}{2})\pi$ at $\varepsilon= \omega_{\tilde{k}}^{-1}$, where $\tilde{k} \in \mathbb{N}_0$ has the same parity as $k$.

(ii)~Assume
	\begin{equation}
	\varepsilon=\omega_k^{-1},\ B\neq 0\,.
	\label{eq:3.8c}
	\end{equation}
Then the only eigenvalue $\mu =\mu_R+i\tilde{\omega}$ such that
	\begin{equation}
	\tilde{\omega} = \omega_k=(k+\tfrac{1}{2}) \pi
	\label{eq:3.8d}
	\end{equation}
is the algebraically simple eigenvalue $\mu=i\omega_k$.

(iii)~Fix $\varepsilon= \omega_k^{-1}$, for some $k\in \mathbb{N}_0$, and fix any constant $K >1$.
Consider any sequence of (scaled) control parameters $B_n$ and nontrivial eigenvalues $\mu_n=\mu_{R,n}+i\tilde{\omega}_n \neq i\omega_k$ such that
	\begin{equation}
	\begin{aligned}
	0 &\leq \mu_{R,n}\,,\\
	1/K &\leq \tilde{\omega}_n\leq K\,, \quad \mathrm{and}\\
	B_n &\rightarrow 0\,.
	\end{aligned}
	\label{eq:3.8e}
	\end{equation}
Then $|\mu_n|$ remains bounded.
Moreover, for any convergent subsequence $\mu_n$ there exists a positive integer $m$ such that
	\begin{equation}
	\begin{aligned}
	\varepsilon\lim \mu_n =
	\lim \varepsilon i\tilde{\omega}_n&=
	\lim i \tilde{\Omega}_n = i\Omega_m=i(2m+1)\,,\quad \mathrm{and}\\
	\lim \,\tilde{\omega}_n \ (\mathrm{mod}\, 2\pi) &\equiv
	\begin{cases}
	\phantom{3}\tfrac{\pi}{2} \quad \mathrm{for} \quad k+m \quad \mathrm{even,}\\
	3 \tfrac{\pi}{2} \quad \mathrm{for} \quad k+m\quad \mathrm{odd.}
	\end{cases}
	\end{aligned}
	\label{eq:3.8f}
	\end{equation}
\end{lem}

\begin{proof}[\textbf{Proof.}]
To prove claim~(i), suppose $\mu=\mu_R+i\tilde{\omega}$ with $\tilde{\omega}=\omega_{\tilde{k}}$:= $(\tilde{k}+\tfrac{1}{2})\pi$, for some nonnegative integer $\tilde{k}$.
We have to conclude $\mu_R =0$ and $\varepsilon= \omega_{\tilde{k}}^{-1}$.

Abbreviating $\varepsilon\tilde{\omega} =: \tilde{\Omega}$, we decompose the characteristic equation \eqref{eq:3.3} into real and imaginary parts at $\tilde{\omega} = \omega_{\tilde{k}}$ to obtain
	\begin{align}
	-\varepsilon B\mu_R+ 
	\tfrac{1}{2}
	(1+ e^{-\pi\varepsilon\mu_R}\cos(\pi \tilde{\Omega}))=0\,;
	\label{eq:3.9}\\
	-B\tilde{\Omega} +(-1)^{k+\tilde{k}}Be^{-\mu_R}-
	\tfrac{1}{2}e^{-\pi\varepsilon\mu_R}\sin(\pi \tilde{\Omega})=0\,.
	\label{eq:3.10}	
	\end{align}
The real part \eqref{eq:3.9} can be solved for $B\mu_R$ as
	\begin{equation}
	0\geq B\mu_R=
	\tfrac{1}{2}\varepsilon^{-1}(1+e^{-\pi\varepsilon\mu_R}
	\cos (\pi \tilde{\Omega}))\geq 0\,.
	\label{eq:3.11}
	\end{equation}
Indeed, the right inequality follows because we have assumed $\mu_R \geq 0$ in \eqref{eq:3.8a}, and the left inequality follows from our assumption $B<0$.
In particular we conclude
	\begin{equation}
	\mu_R =0\qquad \text{and}\qquad \cos(\pi\tilde{\Omega})=-1\,.
	\label{eq:3.12}
	\end{equation}
Insertion of $\mu_R=0$ and $\cos (\pi\tilde{\Omega})=-1,\ \sin(\pi \tilde{\Omega})=0$ in the imaginary part \eqref{eq:3.10} of the characteristic equation  \eqref{eq:3.3} then implies
	\begin{equation}
	0> B\tilde{\Omega} =(-1)^{k+\tilde{k}} B e^{-\mu_R}=(-1)^{k+\tilde{k}}B\,.
	\label{eq:3.13}
	\end{equation}
This proves $1 = \tilde{\Omega} = \varepsilon \omega_{\tilde{k}}$ and $k\equiv\tilde{k}$ (mod 2), as claimed.

To prove claim~(ii), we first note
	\begin{equation}
	\tilde{\Omega} = \varepsilon\tilde{\omega} = \varepsilon\omega_k=1\,,
	\qquad \exp (-i\tilde{\omega})=-(-1)^ki\,,
	\label{eq:3.14a}
	\end{equation}
by assumptions \eqref{eq:3.8c}, \eqref{eq:3.8d}.
For the imaginary part \eqref{eq:3.10} of the characteristic equation \eqref{eq:3.3} at $\mu= \mu_R+i\tilde{\omega}$ this implies
	\begin{equation}
	0= -B+Be^{-\mu_R}\,,
	\label{eq:3.14b}
	\end{equation}
i.e. $\mu_R=0$.

To show algebraic simplicity of the resulting trivial eigenvalue $\mu = i\omega_k$, we differentiate the right hand side of the characteristic equation \eqref{eq:3.3} with respect to $\mu$, there.
A vanishing derivative would require
	\begin{equation}
	0=-\varepsilon B-iB+\tfrac{\pi}{2}\varepsilon\,.
	\label{eq:3.14c}
	\end{equation}
This contradiction proves algebraic simplicity of the trivial Hopf eigenvalue $\mu = \pm i\omega_k$ at $\varepsilon= \omega_k^{-1}$, for any $B$.

To prove claim~(iii), we rewrite the characteristic equation \eqref{eq:2.1} in the form 
\begin{equation}
	b_{n}(\varepsilon \mu_{n}+(-1)^{k}\exp(-\mu_{n}))=1 + \exp(-\pi\varepsilon \mu_{n}).
	\label{eq:3.15a}
	\end{equation}
To show $|\mu_{n}|$ remains bounded, indirectly, we first suppose
	\begin{equation}
	|\mu_n|\rightarrow \infty\,,
	\label{eq:3.16}
	\end{equation}
for some subsequence.
In \eqref{eq:3.8e} we have assumed bounded imaginary parts $\tilde{\omega}_n = \mathrm{Im}\,\mu_n$.
Therefore \eqref{eq:3.8a}, \eqref{eq:3.8e}, and \eqref{eq:3.16} imply
	\begin{equation}
	\mu_{R,n} = \text{Re}\, \mu_n \rightarrow + \infty\,.
	\label{eq:3.17}
	\end{equation}
From \eqref{eq:3.15a} we then obtain, more precisely,
	\begin{equation}
	\varepsilon \,\lim \, b_n\mu_n = 1\,.
	\label{eq:3.18}
	\end{equation}
Taking imaginary parts of \eqref{eq:3.15a} and passing to a convergent subsequence of $\tilde{\omega}_n$, we also obtain
	\begin{equation}
	\begin{aligned}
	\lim \, \tilde{\omega}_n &=
	\lim \tfrac{\mu_n}{\varepsilon b_n\mu_n}\,\mathrm{Im}\,
	(-(-1)^k b_{n}e^{-\mu_n}+e^{-\pi\varepsilon\mu_n})=\\
	&=
	\lim \mu_n\, \mathrm{Im}\,(-(-1)^k b_n e^{-\mu_n}+e^{-\pi\varepsilon\mu_n})=0\,.
	\end{aligned}
	\label{eq:3.19a}
	\end{equation}
This contradicts our lower bound \eqref{eq:3.8e} on $\tilde{\omega}_n$.
Therefore the sequence $|\mu_n|$ remains uniformly bounded.

Next we divide the original unscaled characteristic equation \eqref{eq:2.1} for $\mu = \mu_n$ by $\varepsilon\mu_n-i\neq0$, at fixed $\varepsilon= \omega_k^{-1}$, to obtain
	\begin{equation}
	b_n\frac{\varepsilon\mu_n+(-1)^k\exp (-\mu_n)}{\varepsilon\mu_n-i}=
	\varepsilon \frac{1+\exp (-\pi\varepsilon\mu_n)}{\varepsilon\mu_n-i}\,.
	\label{eq:3.15}
	\end{equation}
Both sides extend to entire functions of $\mu_n$.
Since \eqref{eq:3.15} is entire, and $|\mu_{n}|$ remain bounded, $b_n \rightarrow 0$ then implies
	\begin{equation}
	\varepsilon \, \lim \, \frac{1+\exp (-\pi\varepsilon\mu_n)}
	{\varepsilon\mu_n-i}=0\,.
	\label{eq:3.19b}
	\end{equation}
The denominator cancels the simple zero $\varepsilon\mu_n = i\Omega_0=i$ of the numerator.
The remaining zeros $\varepsilon \mu_{n}=i \Omega_{m}=i (2m+1)$ of the numerator prove claim \eqref{eq:3.8f}, and the trapping lemma.
\end{proof}

We now recall the location of eigenvalues $\mu$ of the characteristic equation \eqref{eq:3.3} in the limit $B\rightarrow \pm \infty$ of vanishing control.
This is well-known material; see e.g. \cite{beco63, hale77, nu78}.
We include a short proof for the convenience of the reader.

\begin{prop}
Let $\varepsilon > 0$.
Consider eigenvalues $\mu \in \mathbb{C}$ at vanishing control $B =\pm \infty$, i.e. solutions of
	\begin{equation}
	0= \varepsilon\mu +(-1)^ke^{-\mu}\,.
	\label{eq:3.20}
	\end{equation}
Then the following claims~(i) -- (iv) hold true.

(i) If $0<\, \mathrm{Im}\, \mu \equiv \tfrac{\pi}{2}\pmod{\pi}$, then
	\begin{equation}
	\mathrm{Re}\, \mu_R=0\,,\ \mu=i\omega_{\tilde{k}} =i(\tilde{k}+1/2)\pi
	 \quad \mathrm{and} \quad \varepsilon= \omega_{\tilde{k}}^{-1}\,,
	\label{eq:3.21}
	\end{equation}
where $\tilde{k} \in \mathbb{N}_0$ has the same parity as $k$.

(ii) At $\varepsilon=1/\omega_k$ the eigenvalue $\mu=i\omega_k$ is algebraically simple.
The local continuation $\mu=\mu(\varepsilon)$ satisfies
	\begin{equation}
	\frac{d}{d\varepsilon} \mu(0) = -\frac{1}{\varepsilon(1+\varepsilon^{2})}(1+i\varepsilon)\,.
	\label{eq:3.22}
	\end{equation}

(iii) For $1/\omega_k<\varepsilon < 1/\omega_{k-2}$ the nontrivial complex eigenvalues $\mu \in \mathbb{C}\smallsetminus \mathbb{R}$ with $\mathrm{Re}\, \mu \geq 0,\ \mathrm{Im}\, \mu>0$ are given by $[k/2]$ algebraically simple eigenvalues $\mu_{0,j},\ j=1, \ldots, [k/2]$, one in each strip
	\begin{equation}
	\begin{aligned}
	0 &< \mathrm{Re}\, \mu_{0,j}\,;\\
	\omega_k-2j\pi &< \mathrm{Im}\, \mu_{0,j} < \omega_k -2j\pi +\tfrac{\pi}{2}\,.
	\end{aligned}
	\label{eq:3.23}
	\end{equation}

(iv) For even $k$, there do not exist real eigenvalues $\mu \geq 0$. 
If $k$ is odd, the only real eigenvalue $\mu \geq 0$ is the algebraically simple eigenvalue defined by the unique positive solution of
	\begin{equation}
	\varepsilon \mu = e^{-\mu}\,.
	\label{eq:3.24}
	\end{equation}
\label{prop:3.2}
\end{prop}
$\hphantom{--}$\\[-2.5cm]

\begin{proof}[\textbf{Proof.}]
Let $\mu = \mu_R+i\tilde{\omega}$.
To prove claim~(i), we assume $\tilde{\omega}= \omega_{\tilde{k}}$.
We take real parts of the complex characteristic equation \eqref{eq:3.20} to see that $\cos \tilde{\omega}=0$ implies $\mu_R=0$.
Taking imaginary parts,
	\begin{equation}
	\begin{aligned}
	0 &= \varepsilon \omega_{\tilde{k}}-(-1)^k 
	e^{-\mu_R}\sin\omega_{\tilde{k}}=\\
	&= \varepsilon \omega_{\tilde{k}}-(-1)^{k+\tilde{k}}
	\end{aligned}
	\label{eq:3.25}
	\end{equation}
shows the remaining claims of (i).

Claim~(ii) follows by implicit differentiation of \eqref{eq:3.20} with respect to $\varepsilon$:
	\begin{equation}
	0=\mu+(\varepsilon-(-1)^ke^{-\mu})\mu'
	\label{eq:3.26}
	\end{equation}
where $\mu'$ abbreviates the implicit derivative with respect to $\varepsilon$.
Inserting $\varepsilon = \omega_k^{-1}$ and $\mu=i\omega_k$ proves claim~(ii).

Claim~(iii) follows by global continuation of the simple Hopf eigenvalues $\mu=\mu(\varepsilon)$ with respect to decreasing $\varepsilon$, alias increasing $|\lambda |$.
We may proceed by induction on $k$.
For large $\varepsilon$, i.e. for small $\lambda$ alias small rescaled delay, already \cite{kur71} observed $\text{Re}\, \mu \rightarrow-\infty$ for all complex eigenvalues.
At $\varepsilon= \omega_{\tilde{k}}^{-1}$, with $\tilde{k}$ of the same parity as $k$ and $j$:= $[\tilde{k}/2]+1$, a simple Hopf eigenvalue $\mu= \mu_{0,j}$:= $i\omega_{\tilde{k}}$ appears on the imaginary axis.
By property~(ii) it progresses, locally, for decreasing $\varepsilon$, into the strip \eqref{eq:3.23}.
By property~(i) that simple eigenvalue $\mu_{0,j}(\varepsilon)$ can never leave that trapping strip again, because $\text{Re}\, \mu_{0,j}>0$ remains bounded above for $\varepsilon>0$ bounded below.
By standard complex analysis, therefore, each $\mu_{0,j}(\varepsilon)$ remains simple and continues globally in its strip, for $0<\varepsilon<\omega_{\tilde{k}}^{-1}$.
The last value $\tilde{k}$ encountered for $\varepsilon> \omega_k^{-1}$ is $\tilde{k}=k-2$.
This proves claim~(iii).

Claim~(iv) on real eigenvalues $\mu$ has been addressed in lemma~\ref{lem:2.1} already.
This proves the proposition.
\end{proof}

The following proposition collects the partial derivatives of the \emph{characteristic function}
	\begin{equation}
	\psi (\mu, \varepsilon,B):=
	-\varepsilon B \mu-(-1)^k Be^{-\mu} +\tfrac{1}{2}
	(1+e^{-\pi\varepsilon\mu})
	\label{eq:3.27}
	\end{equation}
introduced in \eqref{eq:3.3}.

\begin{prop}
The partial derivatives of $\psi = \psi(\mu,\varepsilon,B)$ satisfy
	\begin{align}
	\psi_\mu &= 
	-\varepsilon B +(-1)^k Be^{-\mu}-
	\tfrac{\pi}{2}\varepsilon e^{-\pi\varepsilon\mu}\,;
	\label{eq:3.28}\\
	\psi_\varepsilon &=
	-\mu(B +\tfrac{\pi}{2} e^{-\pi\varepsilon\mu})\,;
	\label{eq:3.29}\\
	B\psi_B &=
	-\varepsilon B \mu-(-1)^kBe^{-\mu}= 
	\psi-\tfrac{1}{2}(1+e^{-\pi\varepsilon\mu})\,.
	\label{eq:3.30}
	\end{align}
At imaginary eigenvalues $\mu= i\tilde{\omega}$, where $\psi (\mu,\varepsilon;B)=0$, and with the abbreviation $\tilde{\Omega}$:= $\varepsilon\tilde{\omega}$, we also obtain the Jacobian determinant
	\begin{equation}
	-B \det \psi_{(\varepsilon, B)} = \tilde{\omega}\cdot 
	(B+\tfrac{\pi}{2})\cdot \cos^2(\tfrac{\pi}{2}\tilde{\Omega})\,.
	\label{eq:3.31}
	\end{equation}
\label{prop:3.3}
\end{prop}
\begin{proof}[\textbf{Proof.}]
The calculations of \eqref{eq:3.28} -- \eqref{eq:3.30} are trivial.
For $\psi=0$ in \eqref{eq:3.27} we also obtain
	\begin{equation}
	\psi_B= -\tfrac{1}{2B}(1+e^{-\pi\varepsilon\mu})\,.
	\label{eq:3.32}
	\end{equation}
Since $\psi\in \mathbb{C}\cong \mathbb{R}^2$, the Jacobian $\psi_{(\varepsilon, B)}$ can be written abstractly as
	\begin{equation}
	\det \psi_{(\varepsilon, B)} = \det
	\begin{pmatrix}
	\text{Re}\; \psi_\varepsilon & \text{Re}\; \psi_B\\
	\text{Im}\; \psi_\varepsilon & \text{Im}\;\psi_B
	\end{pmatrix}
	= \text{Im}\;(\overline{\psi}_\varepsilon\cdot \psi_B)\,.
	\label{eq:3.33}
	\end{equation}
Insertion of \eqref{eq:3.29}, \eqref{eq:3.30}, $\psi=0$ at imaginary eigenvalues $\mu =i\tilde{\omega} = i\varepsilon^{-1} \tilde{\Omega}$ provides the Jacobian determinant
	\begin{equation}
	\begin{aligned}
	-B\det \psi_{(\varepsilon, B)} &=
	\text{Im}\;(\overline{\psi}_\varepsilon \cdot (-B\psi_B)) =\\
	&=
	\tfrac{1}{2}\tilde{\omega}\;\text{Im}\,((Bi+\tfrac{\pi}{2}
	i e^{i\pi\tilde{\Omega}}) (1+e^{-i\pi\tilde{\Omega}}))=\\
	&=
	\tfrac{1}{2}\tilde{\omega}\,(B+\tfrac{\pi}{2})
	(1+\cos (\pi\tilde{\Omega}))\,.
	\end{aligned}
	\label{eq:3.34}
	\end{equation}
This proves \eqref{eq:3.31} and the proposition.
\end{proof}

With these lengthy preparations we can now address transverse crossings at the relevant Hopf eigenvalues, from two viewpoints.
Fix $\varepsilon= \omega_k^{-1}$ and consider nontrivial Hopf eigenvalues $\mu= i\tilde{\omega},\ 0<\tilde{\omega}\neq \omega_k=(k+\tfrac{1}{2})\pi$, at control parameter $B\neq 0$.
Our viewpoint above was to study, equivalently,
	\begin{equation}
	\psi(\mu,\varepsilon, B) =0
	\label{eq:3.35}
	\end{equation}
for $\psi$ defined in \eqref{eq:3.27}; see \eqref{eq:3.3}.

In section~\ref{sec:2} our viewpoint was slightly different; see \eqref{eq:2.21} and lemma \ref{lem:2.3}.
Hashing with the shifted slow frequency
	\begin{equation}
	\Omega = \tilde{\Omega} -\Omega_m =
	\varepsilon\tilde{\omega}-(2m+1)=
	\varepsilon(\omega +\tfrac{\pi}{2}(1-(-1)^k-(-1)^m)-2\pi j)\,,
	\label{eq:3.36}
	\end{equation}
$-\tfrac{\pi}{2}\leq \omega< \tfrac{3}{2}\pi$, $\omega \equiv \tilde{\omega}\pmod{2\pi}$, equivalently, we wrote the characteristic equation for the eigenvalue $\mu=i\tilde{\omega}$ as
	\begin{equation}
	0= \chi_0 (\delta, \omega, \Omega, B)=
	\tilde{\Omega}+(-1)^ki e^{i\omega}-
	\tfrac{1}{B}\sin(\tfrac{\pi}{2}\Omega) e^{i\tfrac{\pi}{2}\Omega}\,;
	\label{eq:3.37}
	\end{equation}
see \eqref{eq:2.4} -- \eqref{eq:2.9}.
In lemma~\ref{lem:2.5} we have described the solutions of \eqref{eq:3.37} by functions
	\begin{equation}
	\Omega = \Omega(\omega)\,,
	\label{eq:3.38}
	\end{equation}
for any fixed $\delta = \Omega_m^{-1}$ and nonnegative integers $m$.
That description was completed with $B= B^\pm (\Omega) = B^\pm (\omega, \Omega(\omega))$ as	
	\begin{equation}
	B=(-1)^k\sin^2(\tfrac{\pi}{2}\Omega)/\cos \omega=
	(-1)^k\cos^2 (\tfrac{\pi}{2}\tilde{\Omega})/\cos \omega\,;
	\label{eq:3.39}
	\end{equation}
see \eqref{eq:2.28} and lemmata~\ref{lem:2.4}, \ref{lem:2.6}.

We now combine these results, for $\varepsilon= \omega_k^{-1}$, with the hashing \eqref{eq:3.36} to define the \emph{Hopf frequencies} $\omega_{m,j}^\pm \in (-\tfrac{\pi}{2}, \tfrac{3}{2}\pi)$ as the intersections of \eqref{eq:3.38} with \eqref{eq:3.36}, i.e.
	\begin{equation}
	\varepsilon(\omega_{m,j}^\pm + \tfrac{\pi}{2}(1-(-1)^k-(-1)^m)-2\pi j)=
	\Omega(\omega_{m,j}^\pm)\,,
	\label{eq:3.40}
	\end{equation}
for $j=1, \ldots, j_m^{\max}\,$. Here and below we restrict attention to the case $\underline{\tilde{\Omega}}_m\leq \tilde{\Omega}\leq \Omega_m$, i.e. $\underline{\Omega}_m\leq\Omega \leq 0$. 
Indeed, the opposite case of $m\geq 1$ and $B^\pm >0$ in \eqref{eq:2.33} will turn out irrelevant in corollary~\ref{cor:3.6} below.

We have to comment on the precise meaning of \eqref{eq:3.40}, in view of lemma \ref{lem:2.5}.
Consider the case $m=0$ first; see fig.~\ref{fig:3.1}.
Since the derivatives of the two branches $\omega=\omega^{\pm}(\Omega)$ in \eqref{eq:2.37} are bounded, their intersections with the near-vertical hashing lines $\tilde{\omega}=\tilde{\Omega}/\varepsilon$ of slope $1/\varepsilon$ are transverse, for $0<\varepsilon\leq\varepsilon_0$ small enough. 
This provides two intersections $\omega=\omega_{0,j}^{\pm}$, one pair for each $j,$ as indicated.

The cases $m\geq 1$ of fig. \ref{fig:3.2} are slightly more involved. 
The strictly decreasing lower branch $\omega = \omega^-(\Omega), \ \underline{\Omega}_m < \Omega < 0,$ is characterized by
	\begin{equation}
	-\tfrac{\pi}{2} < \omega = \omega^-(\Omega) < \underline{\omega} = \tfrac{\pi}{2} \underline{\Omega}_m < 0\,.
	\label{eq:3.40a}
	\end{equation}
This decreasing branch provides unique, transverse intersections $\omega = \omega_{m,j}^-$ with the increasing hashing lines.
The strictly increasing upper branch $\omega = \omega^+(\Omega), \ \underline{\Omega}_m < \Omega < 0,$ on the other hand, may exhibit non-transverse, and even multiple, intersections $\omega \in \omega_{m,j}^+$ with the near-vertical hashing lines, for near-minimal $\Omega\gtrsim\underline{\Omega}_m$.
To simplify our presentation, mostly, we will think of intersection \emph{points} $\omega = \omega_{m,j}^+$, rather than intersection \emph{sets} $\omega \in \omega_{m,j}^+$.
Of course we will proceed with the appropriate care to address the general set case whenever necessary.
Eventually, we will be able to exclude cases where the minimal intersection $\omega_{m,j}^-$ of the hashing line $(m,j)$ belongs to the upper branch $\omega^+(\Omega)$;
see lemma \ref{lem:5.4} below.

With these cautioning remarks in mind, we may proceed, mostly, with the additional requirement
	\begin{equation}
	\omega_{m,j}^- <\omega_{m,j}^+\,,
	\label{eq:3.41}
	\end{equation}
for $m\geq 1$.
The only exception may arise by a tangency of the hashing, at maximal $j=\overline{j}_m,\ m\geq 1$, where $\omega_{m,j}^-=\min \omega_{m,j}^+\,.$ 
From lemma \ref{lem:2.5}, and in particular from \eqref{eq:2.37}, \eqref{eq:2.40} we also recall the boundary values
for the functions $\omega^{\pm}(\Omega)$ at $\Omega = \underline{\Omega}_m$ and $\Omega = 0$.
Let us keep in mind how $\omega_{m,j}^\pm$ come with their shifted and slow variants
	\begin{equation}
	\tilde{\omega}_{m,j}^\pm,\ 
	\tilde{\Omega}_{m,j}^\pm,\
	\Omega_{m,j}^\pm
	\label{eq:3.43}
	\end{equation}
as an entourage; see \eqref{eq:3.36}.
These also define the control parameters
	\begin{equation}
	B =B_{m,j}^\pm =(-1)^k\cos^2(\tfrac{\pi}{2}\tilde{\Omega}_{m,j}^\pm)/
	\cos \omega_{m,j}^\pm
	\label{eq:3.44}
	\end{equation}
where the nontrivial, control-induced Hopf bifurcations with eigenvalues $\mu = \pm i \tilde{\omega}_{m,j}^\pm$ actually occur.

\begin{thm}
With the above notation, the following holds, at $\varepsilon= \omega_k^{-1}$.

(i) The values $B_{0,j}^\pm$ of the control parameter enumerate all nontrivial Hopf bifurcations with eigenvalues $\mu = \pm i\tilde{\omega}$ of frequencies $0<\tilde{\omega}<\omega_k=(k+\tfrac{1}{2})\pi$.

(ii) The values $B_{m,j}^\pm$ with $m\geq 1$ enumerate all nontrivial Hopf bifurcations with eigenvalue frequencies $\tilde{\omega} > \omega_k$ and strictly negative control parameter $B<0$.

(iii) All enumerated Hopf eigenvalues are algebraically simple.

(iv) The local continuations $\mu = \mu_{m,j}^-(B)$ of all enumerated Hopf eigenvalues $i\omega_{m,j}^- = \mu_{m,j}^-(B_{m,j}^-)$ cross the imaginary axis transversely with
	\begin{equation}
	\frac{d}{dB}\mathrm{Re}\, \mu_{m,j}^-(B) > 0
	\label{eq:3.45}
	\end{equation}
at $B=B_{m,j}^-$.
At $B = \max B_{m,j}^+\,,$ generated by the frequencies $\omega_{m,j}^+\,,$ that unstable eigenvalue recovers stability, at the latest.
\label{thm:3.4}
\end{thm}
\begin{proof}[\textbf{Proof.}]
The proof of claims~(i) and (ii) follows from our detailed analysis of the characteristic equation $\chi_0=0$ in section~\ref{sec:2}; see in particular lemma~\ref{lem:2.4}.

To prove simplicity of Hopf eigenvalues, (iii), we partially substitute the explicit expression \eqref{eq:2.28} for $B$ into $\psi_\mu$ of \eqref{eq:3.28}, and take real parts:
	\begin{equation}
	\begin{aligned}
	\text{Re}\, \psi_\mu &=
	-\varepsilon B+(-1)^kB\cos\tilde{\omega}-
	\tfrac{\pi}{2}\varepsilon\cos(2\tfrac{\pi}{2}\tilde{\Omega})=\\
	&=
	-\varepsilon B+\cos^2(\tfrac{\pi}{2}\tilde{\Omega})-
	\tfrac{\pi}{2}\varepsilon(2\cos^2(\tfrac{\pi}{2}\tilde{\Omega})-1)=\\
	&=
	\varepsilon(\tfrac{\pi}{2}-B)+
	(1-\pi\varepsilon)\cos^2(\tfrac{\pi}{2}\tilde{\Omega})>0	
	\end{aligned}
	\label{eq:3.46}
	\end{equation}
at $\pi\varepsilon=\pi\omega_k^{-1}=1/(k+\tfrac{1}{2}),\ k\geq 1$, and for all enumerated $B$.
Indeed, by lemma~\ref{lem:2.6}, the only exception to $B\leq \tfrac{\pi}{2}$ arises for $B=\tfrac{\pi}{2}$ at the excluded trivial eigenvalue with frequency $\tilde{\omega}=\tilde{\Omega}/\varepsilon=1/\varepsilon =\omega_k$.
This proves simplicity claim~(iii).

\begin{figure}[t!]
\centering \includegraphics[width=\textwidth]{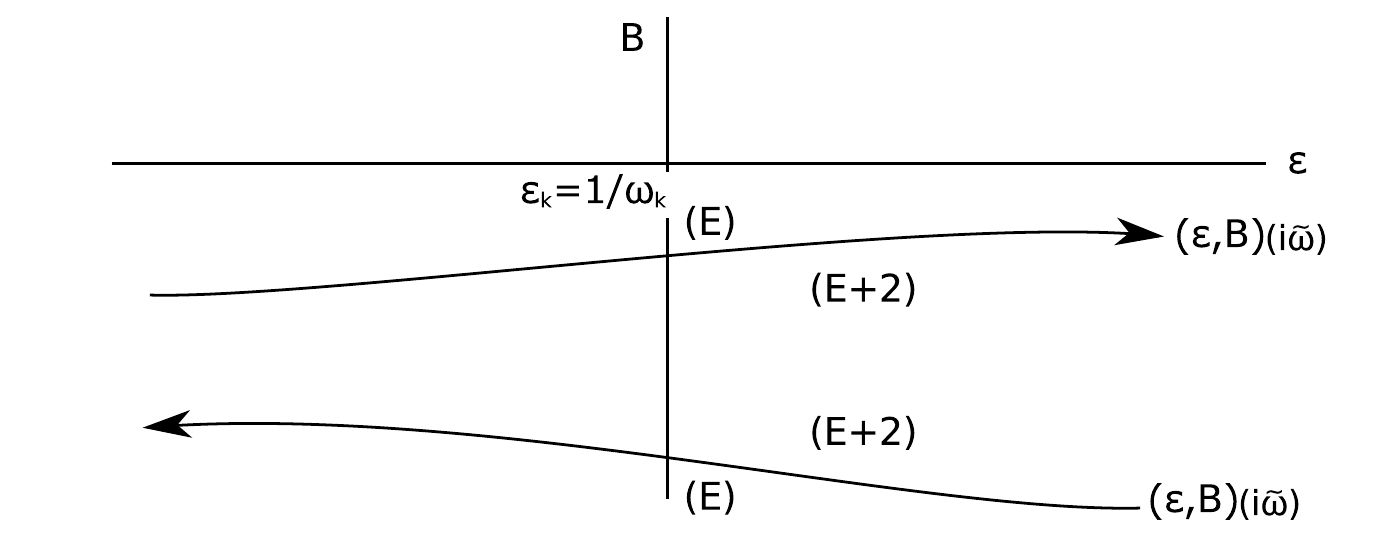}
\caption{\emph{
Hopf curves $\tilde{\omega} \mapsto (\varepsilon(i\tilde{\omega}), B(i\tilde{\omega}))$, oriented along increasing $\tilde{\omega}$.
Note the resulting unstable dimensions $E$, in parantheses, to the left, and $E+2$ to the right, of the Hopf curves.
}}
\label{fig:3.3}
\end{figure}

For $m\geq1$, our proof of the remaining crossing and (de)stabilization claims~(iv) will be based on the following three ingredients.
We will first invoke the implicit function theorem to show that the local continuation map
	\begin{equation}
	(\varepsilon, B) \mapsto \mu=\mu(\varepsilon,B)
	\label{eq:3.47}
	\end{equation}
is an orientation preserving diffeomorphism, near $\varepsilon=\omega_k^{-1}$ and the enumerated Hopf eigenvalues $\mu=i\omega_{m,j}^\pm$ at $B=B_{m,j}^\pm$.
In a second step we will then show
	\begin{equation}
	\varepsilon_\omega(i\omega_{m,j}^-)<0
	\label{eq:3.48}
	\end{equation}
for the partial derivatives, with respect to $\omega$, of the local inverse function $(\varepsilon,B)=(\varepsilon(\mu), B(\mu))$ to \eqref{eq:3.47}, at $\mu=i\omega_{m,j}^-$. The third ingredient describes the necessary adaptations at $\omega_{m,j}^+$.

We first show how claims \eqref{eq:3.47} and \eqref{eq:3.48} imply the crossing direction
\begin{equation}
	\text{Re}\, \mu_B>0\,,
	\label{eq:3.49}
	\end{equation}
for the partial derivative of $\mu(\varepsilon, B)$ with respect to $B$ at $\mu=i\omega_{m,j}^-$.
Indeed consider the oriented Hopf curve  $\tilde{\omega} \rightarrow(\varepsilon(i\tilde{\omega}), B(i\tilde{\omega}))$ in the $(\varepsilon,B)$ plane.
See fig. \ref{fig:3.3}.
By the orientation preserving transformation \eqref{eq:3.47}, the region $\text{Re}\, \mu<0$ lies to the left of the Hopf curve.
By \eqref{eq:3.48}, the tangent to the Hopf curve at $\omega=\omega_{m,j}^-$ points strictly to the left of the vertical $B$-axis at the fixed value $\varepsilon=\varepsilon_k$.
Therefore, the $B$-axis crosses the Hopf curve transversely, at $\omega=\omega_{m,j}^-$, and into the unstable region Re $\mu>0$, for increasing $B$.
Thus the diffeomorphism \eqref{eq:3.47} implies the crossing direction \eqref{eq:3.49}.

The cases $B=B_{m,j}^+$ can be treated analogously, with a little extra care.
In the case of a single transverse crossing of the hashing line $(m,j)$ with the upper branch $\omega=\omega^+(\Omega)$, at $\omega=\omega_{m,j}^+$, we now have $\varepsilon_\omega(i\omega_{m,j}^+)>0$. 
Therefore, the tangent to the Hopf curve at $\omega=\omega_{m,j}^+$ now points strictly to the right of the vertical $B$-axis at the fixed value $\varepsilon=\varepsilon_k$.
The previous arguments then show $\text{Re}\, \mu_B<0,$ i.e. stabilization towards increasing $B$.
In case of multiple crossings, possibly involving tangents, we can prove stabilization from the last crossing (or tangency) onwards, as claimed in (iv), via generic approximation by an odd number of transverse crossings of the hashing line with the upper branch $\omega=\omega^+(\Omega)$. 
Put simply, destabilization occurs whenever the increasing hashing lines in the top rows of fig. \ref{fig:3.2} enter the interior region of the 2-scale relation $\Omega = \Omega(\delta,\omega)$, and stabilization ensues as soon as the hashing lines leave towards the exterior region; see lemma \ref{lem:2.5}.

To prove claim~(iv) for $m\geq1$ it therefore remains to verify claims \eqref{eq:3.47} and \eqref{eq:3.48}.
We will address the analogous, but simpler, case $m=0$ at the end of the proof.

To verify the orientation claim \eqref{eq:3.47} we invoke the implicit function theorem for $\psi(\mu,\varepsilon, B)=0$; see \eqref{eq:3.3}, \eqref{eq:3.27}.
Indeed the Jacobian determinants, which determine the local orientations, satisfy
	\begin{equation}
	\det \psi_\mu\cdot \det\mu_{(\varepsilon,B)}=
	\det(-\psi_{(\varepsilon,B)})=
	\det(\psi_{(\varepsilon,B)})\,.
	\label{eq:3.50}
	\end{equation}
Our enumeration of cases $B=B_{m,j}^\pm$ for $m\geq 1$ above has skipped any positive $B^\pm$ of lemma~\ref{lem:2.4}, \eqref{eq:2.33}.
By lemma \ref{lem:2.6} we know $-1<B^-\leq0$.
Therefore proposition~\ref{prop:3.3}, \eqref{eq:3.31} asserts strict positivity of $\det (\psi_{(\varepsilon, B)})$ in \eqref{eq:3.50}.
The enumerated Hopf eigenvalues $\mu=i\omega_{m,j}^\pm$ are simple zeros of the complex analytic characteristic function $\psi\in \mathbb{C}$, by claim~(iii).
Therefore $\det\psi_\mu$ on the left is also strictly positive, by the Cauchy-Riemann equations.
This proves strict positivity of $\det\mu_{(\varepsilon,B)}$ and establishes the orientation claim \eqref{eq:3.47}.

To determine the signs of the tangent partial derivatives $\varepsilon_\omega$, as claimed in \eqref{eq:3.48}, we recall the definition
	\begin{equation}
	\varepsilon = \tilde{\Omega}/\tilde{\omega} =
	\tilde{\Omega}(\tilde{\omega})/\tilde{\omega}\,,
	\label{eq:3.51}
	\end{equation}
where $\tilde{\Omega}(\tilde{\omega})=\Omega_m+\Omega(\omega)$ follows from the $\varepsilon$-independent characteristic equation $\chi_0(\delta, \omega, \Omega, B)=0$ at $\omega \equiv \tilde{\omega}\pmod{2\pi}$, for fixed $\delta= \Omega_m^{-1}$.
See \eqref{eq:3.37}, \eqref{eq:3.38} and lemmata~\ref{lem:2.4}, \ref{lem:2.5}.
Straightforward differentiation of \eqref{eq:3.51} with respect to $\tilde{\omega}>0$ or $\omega$ yields
	\begin{equation}
	\begin{aligned}
	\varepsilon_\omega =
	(\dot{\Omega}-\tilde{\Omega}/\tilde{\omega})/\tilde{\omega}
	=(\dot{\Omega}-\varepsilon)/\tilde{\omega}\,,
	\end{aligned}
	\label{eq:3.52}
	\end{equation}
in the notation of lemma~\ref{lem:2.5}, where $\dot{\Omega} =d\Omega/d\omega=d\tilde{\Omega}/d\tilde{\omega}$.
The definition of $\omega_{m,j}^-$ as the unique intersection of the hashing line $(m,j)$ with the lower branch $\omega=\omega^-(\Omega)$ in \eqref{eq:3.36},  and \eqref{eq:3.40}, \eqref{eq:3.40a} show that \eqref{eq:3.52} implies the sign of $\varepsilon_\omega$ claimed in \eqref{eq:3.48}.
See the two top rows of fig. \ref{fig:3.2}. The case of a single transverse crossing at $\omega=\omega_{m,j}^+$ leads to $\varepsilon_\omega(i\omega_{m,j}^+)>0,$ analogously.
The required adaptations for multiple and/or non-transverse crossings have been described above.

It remains to address the case $m=0,\ 0<\tilde{\Omega}<1$.
Consider $\omega_{0,j}^+,\ B_{0,j}^+<0$ first; see lemmata~\ref{lem:2.4} -- \ref{lem:2.6}.
Here each crossing $\omega=\omega_{m,j}^\pm$ is transverse and unique. 
Transformation \eqref{eq:3.47} remains orientation preserving, by \eqref{eq:3.31} and \eqref{eq:3.50}, verbatim as for $m\geq 1$.
Furthermore \eqref{eq:3.52} implies $\varepsilon_\omega >0$, for small enough $0<\varepsilon<\varepsilon_0$, by an upper bound on the positive derivatives
	\begin{equation}
	0<1/\dot{\Omega}=\frac{d}{d\Omega} \omega^\pm (\Omega) < 1/\varepsilon_0
	\label{eq:3.53}
	\end{equation}
in lemma~\ref{lem:2.5}.
This shows claim~(iv), \eqref{eq:3.45} at $B=B_{0,j}^+<0$, for $m=0$.
To show claim~(iv), \eqref{eq:3.45} at $B=B_{0,j}^->0$, for $m=0$, we first note that \eqref{eq:3.31}, \eqref{eq:3.50} now imply orientation reversal in \eqref{eq:3.47}, because
	\begin{equation}
	\det \psi_{(\varepsilon,B)} <0\,.
	\label{eq:3.54}
	\end{equation}
The argument \eqref{eq:3.53}, however, remains intact at $\omega^-(\Omega)$.
This shows how the sign reversal claimed in~(iv), \eqref{eq:3.45} remains valid for $m=0$, proving the lemma.
\end{proof}

Figs.~\ref{fig:3.1} and \ref{fig:3.2} already summarized our results, so far, separately for $m=0$ and for $m\geq 1,\ B<0$.
Consider the case $m=0$ first, i.e. slow Hopf frequencies $0<\tilde{\Omega} <\Omega_0=1$, alias $-1<\Omega=\tilde{\Omega}-\Omega<0$.
The branch $\omega = \omega^+(\Omega)$ of Hopf frequencies $\omega\equiv \tilde{\omega} \pmod{2\pi}$ provides
	\begin{equation}
	k'=[k/2]
	\label{eq:3.55}
	\end{equation}
Hopf bifurcations at control parameters $B= B_{0,j}^+,\ j=1,\ldots,k'$.
Hashing and strong monotonicity of $B=B^+(\Omega)$, lemma~\ref{lem:2.6}, imply
	\begin{equation}
	\begin{aligned}
	-1<\Omega_{0,k'}^+&< \ldots <
	\Omega_{0,2}^+<\Omega_{0,1}^+<1\,;\\
	-1<B_{0,k'}^+ &<\ldots <B_{0,2}^+<B_{0,1}^+<0\,.
	\end{aligned}
	\label{eq:3.56}
	\end{equation}
By lemma~\ref{lem:3.1}, unstable eigenvalues $\mu=\mu_R+i\tilde{\omega}$ cannot cross any of the lines $\tilde{\omega}\equiv \tfrac{\pi}{2}\pmod{\pi}$, for $B<0$.
By proposition~\ref{prop:3.2}, each of the $k'=[k/2]$ resulting strips
	\begin{equation}
	\text{Re}\, \mu>0\,,\qquad 0<\omega_k-2j\pi<\tilde{\omega}<\omega_k
	-2j\pi+\pi\,,
	\label{eq:3.57}
	\end{equation}
$j=1,\ldots,k'$, contains exactly one simple eigenvalue $\mu_{0,j}$ inherited from $B=-\infty$.
By theorem~\ref{thm:3.4} and analytic continuation, this simple eigenvalue persists as $B$ increases, until it disappears into $\text{Re}\, \mu<0$ by simple transverse Hopf bifurcation at
	\begin{equation}
	B=B_{0,j}^+\,,\qquad \tilde{\omega}= \omega_{0,j}^+\,.
	\label{eq:3.58}
	\end{equation}
Indeed $\tilde{\omega}=\omega_{0,j}^+$ belongs to the same strip \eqref{eq:3.57}, for each $j=1,\ldots , k'$.
For even $k=2k'$, this eliminates all unstable eigenvalues generated at $B=-\infty$, once
	\begin{equation}
	B_{0,1}^+<B<0\,.
	\label{eq:3.59}
	\end{equation}
For odd $k=2k'+1$, the same statement remains true, because $-1<B<0$ renders the additional real eigenvalue stable; see lemma~\ref{lem:2.1} and corollary~\ref{cor:2.2}.
Note $-1<B_{0,1}^+$ here, by lemma~\ref{lem:2.6}; see also fig.~\ref{fig:3.1}.
These remarks prove the following corollary.

\begin{cor}
Let $\varepsilon=1/\omega_k$ and assume $B<0$.
Then $\mu_R<0$ for any eigenvalue $\mu=\mu_R +i\tilde{\omega}$ with $0\leq \tilde{\omega}<\omega_k$, if and only if \eqref{eq:3.59} holds.
\label{cor:3.5}
\end{cor}

We study $B>0$ next.
In corollary~\ref{cor:2.2} we have already observed instability, by parity due to the presence of a real eigenvalue $\mu>0$, in case $k$ was odd.
Let us therefore consider even $k=2k'$.
At $\varepsilon=1/\omega_k$ and $B=+\infty$ we encounter the same $k'$ unstable simple complex eigenvalues $\mu_{0,j}\,,$ one in each of the $k'$ strips \eqref{eq:3.57}, as before.
This time, however, only $k'-1$ simple transverse Hopf bifurcations at $B=B_{0,j}^->0$ offer their assistance for stabilization by decreasing $B>0$.
Indeed $B_{0,j}^-$ cancels the instability of $\mu_{0,j+1}$, for $j=1, \ldots ,k'-1$, but $\mu_{0,k'}$ remains unstable for all $B>0$.
This proves the following corollary.

\begin{cor}
Let $\varepsilon =1/\omega_k$ and assume $B>0$.
Then there exists an unstable eigenvalue $\mu$, i.e. $\mathrm{Re}\, \mu>0$.
For odd $k$, the unstable eigenvalue can be taken to be real.
For even $k \geq 2$, the unstable eigenvalue $\mu = \mu_R +i\tilde{\omega}$ can be taken to be strictly complex with
	\begin{equation}
	0<\omega_k-2\pi<\tilde{\omega}<\omega_k-\pi\,.
	\label{eq:3.60}
	\end{equation}
In particular, there does not exist any region of Pyragas stabilization (near Hopf bifurcation) for control parameters $B>0$.
\label{cor:3.6}
\end{cor}

Henceforth we restrict attention to the remaining case $B<0$.
We fix $\varepsilon=1/\omega_k$.
All unstable real eigenvalues, or complex eigenvalues $\mu =\mu_R +i\tilde{\omega}$ with $0<\tilde{\omega} <\omega_k$ come from $B=-\infty$ and have been taken care of in corollary~\ref{cor:3.5}.
By the trapping lemma~\ref{lem:3.1} for imaginary parts, all remaining changes of stability must arise from the simple transverse Hopf bifurcations at
	\begin{equation}
	B=B_{m,j}^\pm\,,\quad m\geq 1\,,
	\label{eq:3.61}
	\end{equation}
as enumerated in theorem~\ref{thm:3.4}.
Note how the imaginary parts $\tilde{\omega} >0$ of any unstable eigenvalues $\mu=\mu_R+i\tilde{\omega}$ induced by these Hopf bifurcations are confined to the disjoint strips
	\begin{equation}
	\tilde{\omega} = \omega +2\pi(km+[(k+1)/2]+
	[(m+1)/2]-j)
	\label{eq:3.62}
	\end{equation}
where $\tfrac{\pi}{2}<\omega <3\tfrac{\pi}{2}$ for even $k$, and $-\tfrac{\pi}{2}<\omega <\tfrac{\pi}{2}$ for odd $k$.
This follows from hashing \eqref{eq:2.24}, \eqref{eq:3.40} at Hopf bifurcation frequencies $\tilde{\omega} = \tilde{\omega}_{m,j}^\pm$, and persists with instability of $\mu$, by trapping lemma~\ref{lem:3.1} of imaginary parts.
In particular, the strips are disjoint, for different $(m,j)$, and $B_{m,j}^\pm$ generate frequencies $\tilde{\omega}_{m,j}^\pm$ which remain in the same strip.
See also fig.~\ref{fig:3.2}.
This proves the following corollary.

\begin{cor}
Let $0<\varepsilon =1/\omega_k\leq\varepsilon_0$ be sufficiently small, and assume $B<0$.
Then there exists an unstable eigenvalue $\mu$, i.e. $\mathrm{Re}\, \mu >0$, if at least one of the following conditions holds:
	\begin{equation}
	\begin{aligned}
	B<B_{0,1}^+ &<0\,,\qquad \mathrm{or}\\
	B_{m,j}^-<B &< \min B_{m,j}^+\leq0,
	\end{aligned}
	\label{eq:3.63}
	\end{equation}
for some $m\geq 1$ and some $j$ enumerated in theorem~\ref{thm:3.4}.
Here we define $B_{m,1}^+:=0$ for odd $m$.
\label{cor:3.7}
\end{cor}

We have implicitly excepted hashing tangencies $\Omega_{m,j}^-= \min \tilde{\Omega}_{m,j}^+$ in equation \eqref{eq:2.63}.
Indeed this case corresponds to a nontransverse Hopf point, at (scaled) frequency $\tilde{\Omega}_{m,j}^- = \min \tilde{\Omega}_{m,j}^+$, from the stable side. 
This does not contribute to the strict unstable dimension $E(B)$.

Hence we may assume $\tilde{\Omega}_{m,j}^- < \tilde{\Omega}_{m,j}^+$.
Then hashing \eqref{eq:3.40} implies $B_{m,j}^-<B_{m,j}^+$.
Indeed this follows from strong monotonicity of $B^+(\Omega)$ in case $\Omega_{m,j}^\pm= \Omega(\omega_{m,j}^\pm)$ with $\omega_{m,j}^\pm \geq \underline{\omega}$; see lemma~\ref{lem:2.6}.
If $\omega_{m,j}^-<\underline{\omega}<\omega_{m,j}^+$, then we reach the same conclusion, again by strong monotonicity of $B^+(\tilde{\Omega})$ and lemma~\ref{lem:2.4}:
	\begin{equation}
	B_{m,j}^-=B^-(\Omega_{m,j}^-)<B^+(\Omega_{m,j}^-)<
	B^+(\Omega_{m,j}^+)=B_{m,j}^+\,.
	\label{eq:3.64}
	\end{equation}

More systematically, these arguments are collected in the following proposition.

\begin{prop}
Let $\varepsilon=1/\omega_k\leq\varepsilon_0$ be sufficiently small.
For any $m \geq 1$ consider the enumeration of Hopf bifurcation parameters $B_{m,j}^\pm$ of theorem~\ref{thm:3.4}.
Then
	\begin{equation}
	B_{m,j}^- < B_{m,j}^+ <0\,,
	\label{eq:3.65}
	\end{equation}
except at a possible hashing tangency $\Omega_{m,j}^-= \min \Omega_{m,j}^+$.
Moreover the series $B_{m,j}^+$ decreases strictly monotonically in $j=1,\ldots, j_m^{\max}$, for each fixed $m\geq 1$.
\label{prop:3.8}
\end{prop}

\begin{proof}[\textbf{Proof.}]
Claim \eqref{eq:3.65} has been proved in \eqref{eq:3.64}.
Strict monotonicity of $B_{m,j}^+= B^+(\Omega_{m,j}^+)$ in $j$ follows from strict monotonicity of $\tilde{\Omega}_{m,j}^+=\tilde{\omega}_{m,j}^+/\varepsilon$ in the $j$-strips \eqref{eq:3.62} and from strict monotonicity of $\Omega \mapsto B^+(\Omega)$ in lemma~\ref{lem:2.6}.
This proves the proposition.
\end{proof}

We summarize the results of this section in a final corollary.
Define
	\begin{equation}
	j_m:= [(m+1)/2]\,,
	\label{eq:3.66}
	\end{equation}
for integer $m\geq 1$.
In the following sections we will show, for small enough $0<\varepsilon =1/\omega_k\leq \varepsilon_0$, that $B_{0,1}^+$ is unique, and
	\begin{equation}
	\max B_{m,j_m+1}^+ <B_{0,1}^+<0\qquad
	\text{for} \quad j_m+1\leq j_m^{\max}\,,
	\label{eq:3.67}
	\end{equation}
i.e. as long as $\Omega_{m,j_m+1}^+$ exists.
The maximum is taken over all $m\geq 1$.
See \eqref{eq:3.40} for the delimiter $j=1, \ldots , j_m^{\max}$ of the enumeration $B_{m,j}^\pm$.
On the other hand, we will also show
	\begin{align}
	B_{0,1}^+ &<B_{1,1}^-<0\,,\quad \ \text{and} \quad 	
	\label{eq:3.68}\\
	B_{1,1}^- &\leq B_{m,j}^-<0\,,\quad \text{for all} \ m\geq 1 \ \text{and}\ 
	j=1, \ldots , 	\min\lbrace j_m,j_m^{\max}\rbrace\,.
	\label{eq:3.69}
	\end{align}
This identifies the Pyragas region $\mathcal{P}$ as follows.

\begin{cor}
Let $0<\varepsilon=1/\omega_k\leq \varepsilon_0$ be chosen small enough and assume the orderings \eqref{eq:3.67} -- \eqref{eq:3.69}, for all $m\geq 1$.
Then the nonempty Pyragas region
	\begin{equation}
	\mathcal{P} = \lbrace B<0\,|\quad B_{0,1}^+<B<B_{1,1}^-\rbrace	
	\label{eq:3.70}
	\end{equation}
is the only region of control parameters $B=\tfrac{1}{2}b/\varepsilon$ in the delay equation \eqref{eq:1.18}, such that Pyragas stabilization succeeds for the Hopf bifurcation at $\lambda = \lambda_k=(-1)^{k+1}\varepsilon^{-1}$.
\label{cor:3.9}
\end{cor}

\begin{proof}[\textbf{Proof.}]
By corollary~\ref{cor:3.6}, instability prevails for all $B>0$.
By corollary~\ref{cor:3.7}, instability holds for $B<B_{0,1}^+$, and for $B_{1,1}^- <B<0$.
It therefore remains to show that
	\begin{equation}
	(B_{m,j}^-\, , \max B_{m,j}^+)\cap (B_{0,1}^+, B_{1,1}^-)= \emptyset\,,
	\label{eq:3.71}
	\end{equation}
for all $m\geq 1$ and $j=1, \ldots , j_m^{\max}$.

Strong monotonicity of $B_{m,j}^+$ with respect to $j$, as in proposition~\ref{prop:3.8}, and assumption \eqref{eq:3.67} imply
	\begin{equation}
	\max B_{m,j}^+ \leq \max B_{m,j_m+1}^+ < B_{0,1}^+
	\label{eq:3.72}
	\end{equation}
for all $m\geq 1$ and $j>j_m$.
This establishes claim \eqref{eq:3.71} for $j_m <j\leq j_m^{\max}$, provided that $j_m <j_m^{\max}$.

It remains to show claim \eqref{eq:3.71} for $1\leq j\leq \min \lbrace j_m, j_m^{\max}\rbrace$.
In this case we invoke assumption \eqref{eq:3.69} to conclude an empty intersection \eqref{eq:3.71}, again.
Since the Pyragas region \eqref{eq:3.70} is nonempty, by assumption \eqref{eq:3.68}, this proves the corollary.
\end{proof}


\section{Locally uniform expansions in $\varepsilon = 1/\omega_k$}
\label{sec:4}

To locate and understand the Pyragas region
	\begin{equation}
	\mathcal{P} = \lbrace B\ |\ B_{0,1}^+ < B < B_{1,1}^-\rbrace\,,
	\label{eq:4.1}
	\end{equation}
in the limit $\varepsilon= 1/\omega_k \rightarrow 0$ of Hopf bifurcations with large unstable dimensions $k$, it remains to establish the precise locations of the control induced Hopf bifurcations $B_{m,j}^\pm$ relative to the gap \eqref{eq:4.1}.
See assumptions \eqref{eq:3.67} -- \eqref{eq:3.69} of corollary~\ref{cor:3.9}.
In the present section we accomplish this task, by expansions with respect to small $\varepsilon$, for arbitrarily bounded $m=1,\ldots , m_0\,$, alias $\delta \geq \delta_0$:= $1/\Omega_{m_0}=1/(2m_0+1)$.

To be precise, we first fix any $m_0\in\mathbb{N},$ alias $\delta_0>0.$
In section \ref{sec:5}, we will choose $\delta_0$ sufficiently small.
In the present section we will then consider $0<\varepsilon\leq\varepsilon_0=\varepsilon_0(\delta_0)$ small enough for certain $\varepsilon$-expansions of $B_{m,j}^\pm$ to hold, uniformly for all $m\leq m_0$ and $j\leq j_m+1,\ j_m=[(m+1)/2].$
Note how the derivatives of $\omega^\pm(\Omega)$ remain bounded in the relevant region; see \eqref{eq:4.2}.
Hence $\omega_{m,j}^\pm$ and $B_{m,j}^\pm$ are defined uniquely by transverse intersections of the 2-scale characteristic equation \eqref{eq:2.27} with the hashing lines $\tilde{\omega} = \tilde{\Omega}/\varepsilon$ of \eqref{eq:4.5}, \eqref{eq:4.6}, in the present section.

In particular, the $\varepsilon$-expansions of $B_{0,1}^+$ and $B_{1,1}^-$ will establish the expansions \eqref{eq:1.20} of $\underline{b}_k = 2\varepsilon B_{0,1}^+$ and $\overline{b}_k=2\varepsilon B_{1,1}^-$; see \eqref{eq:1.29}.
The limit $\delta \rightarrow 0$ of large $m \rightarrow \infty$ requires a different approach, and will therefore be deferred to the next section.

Our strategy has been outlined in \eqref{eq:1.24} -- \eqref{eq:1.29}.
We first solve the $\varepsilon$-independent 2-scale characteristic equation $H(\Omega,\omega)=0$ of \eqref{eq:2.27} for $\omega = \omega (\Omega)$ explicitly:
	\begin{equation}
	\omega = \omega^\pm(\Omega)= 
	\tfrac{\pi}{2}\Omega \pm 
	\arccos (-\tilde{\Omega}\sin (\tfrac{\pi}{2}\Omega))\,;
	\label{eq:4.2}
	\end{equation}
see lemma~\ref{lem:2.5}.
Here and below we only consider odd $k$, without loss of generality.
Even $k$ add $\pi$ to $\omega^\pm$.
Recall $|\omega | \leq \tfrac{\pi}{2}$, for odd $k$.
We also recall $\tilde{\Omega} =1/\delta +\Omega$ and
	\begin{equation}
	0\leq -\tilde{\Omega} \sin (\tfrac{\pi}{2}\Omega)\leq 1\,,
	\label{eq:4.3}
	\end{equation}
by the discriminant condition \eqref{eq:2.34}.
Equivalently
	\begin{equation}
	0\geq \Omega \geq \underline{\Omega}_m\,.
	\label{eq:4.4}
	\end{equation}
Note bounded derivatives of $\omega(\Omega)$, locally uniformly for $0\geq \Omega > \underline{\Omega}_m$.
We suppress explicit dependence on $\delta$, viz. $m$, in the present section.

Next we insert \eqref{eq:4.2} into the hashing relation \eqref{eq:2.22} of lemma~\ref{lem:2.3}:
	\begin{equation}
	\Omega = \varepsilon (\omega^\pm (\Omega) -a \pi)
	\label{eq:4.5}
	\end{equation}
with the abbreviation
	\begin{equation}
	a=a_{m,j}=2j-1+\tfrac{1}{2}(-1)^m\,.
	\label{eq:4.6}
	\end{equation}
Here we have used that $k$ is odd; the relevant modifications for even $k$ cancel in \eqref{eq:4.5} and below.
By the implicit function theorem we can solve \eqref{eq:4.5} for
	\begin{equation}
	\Omega = \Omega(\varepsilon, a)= \Omega_{m,j}^\pm (\varepsilon)\,,
	\label{eq:4.7}
	\end{equation}
uniquely, for small enough $0 <\varepsilon \leq \varepsilon_0 = \varepsilon_0(\delta_0).$
Inserting the result into $\omega^\pm$ of \eqref{eq:4.2} and $B$ of \eqref{eq:2.28} we obtain expansions
	\begin{equation}
	\begin{aligned}
	\omega &=& \omega_{m,j}^\pm(\varepsilon)& :=
	\omega^\pm (\Omega_{m,j}^\pm (\varepsilon))\,,\\
	B &=& B_{m,j}^\pm(\varepsilon)& :=
	-\sin^2(\tfrac{\pi}{2}\Omega_{m,j}^\pm(\varepsilon))/
	\cos \omega_{m,j}^\pm(\varepsilon)\,.
	\end{aligned}
	\label{eq:4.8}
	\end{equation}
We collect these straightforward expansions in the following proposition.

\begin{prop}
\label{prop:4.1}
For any fixed $\delta_0=1/\Omega_{m_0} >0$ consider $0< \varepsilon \leq \varepsilon_0 = \varepsilon_0(\delta_0)$ small enough.
We abbreviate the coefficients
\begin{equation}
\alpha^\pm_{m,j} := 4j + (-1)^m-2\mp1
\label{eq:4.9a}
\end{equation}
Then the expansions for $\Omega_{m,j}^\pm,\ \omega_{m,}^\pm$, and $B_{m,j}^\pm$ with respect to $\varepsilon$ are
	\begin{align}
	\begin{split}
	\Omega_{m,j}^-(\varepsilon)&=
	\alpha^-_{m,j}(-\tfrac{\pi}{2}\varepsilon+ 2m(\tfrac{\pi}{2}\varepsilon)^2+\ldots\,)\,;\\ 
	\Omega_{m,j}^+(\varepsilon)&=
	\alpha^+_{m,j}(-\tfrac{\pi}{2}\varepsilon- 2(m+1)(\tfrac{\pi}{2}\varepsilon)^2+\ldots\,)\,;
	\end{split}\label{eq:4.9}\\
	\intertext{~}
	\begin{split}
	\omega_{m,j}^-(\varepsilon)&=
	\tfrac{\pi}{2}(-1+2m\,\alpha^-_{m,j}\,\tfrac{\pi}{2}\varepsilon+
	\ldots\,)\,;\\
	\omega_{m,j}^+(\varepsilon)&=
	\tfrac{\pi}{2}(+1-2(m+1)\,\alpha^+_{m,j}\,\tfrac{\pi}{2}\varepsilon+
	\ldots\,)\,;
	\end{split}\label{eq:4.10}\\
	\intertext{~}
	\begin{split}
	B_{m,j}^-(\varepsilon) &=
	\tfrac{\pi}{4m}\,\alpha^-_{m,j}\,(-\tfrac{\pi}{2}\varepsilon\, + \tfrac{1}{2m}(4m^2-\alpha^-_{m,j})(\tfrac{\pi}{2}\varepsilon)^2+\ldots\,)\,;\\
	B_{m,j}^+(\varepsilon) &=
	\tfrac{\pi}{4(m+1)}\,\alpha^+_{m,j}\,(-\tfrac{\pi}{2}\varepsilon\, - \tfrac{1}{2(m+1)}(4(m+1)^2+\alpha^+_{m,j})(\tfrac{\pi}{2}\varepsilon)^2+\ldots\,)\,.
	\end{split}\label{eq:4.11}
	\end{align}
The expansions for $B^\pm,\ \Omega^\pm$ hold for even and odd $k$, alike.
The expansions for $\omega^\pm$ are given for odd $k$.
For even $k$, we have to add $\pi$; see lemma~\ref{lem:2.5}.
\end{prop}

\begin{proof}[\textbf{Proof.}]
We omit the obvious and tedious calculations.
\end{proof}

\begin{cor}
In the setting of proposition~\ref{prop:4.1} we obtain the following expansions and inequalities, for $0< \varepsilon\rightarrow 0:$
	\begin{align}
	B_{1,1}^- &=
	-(\tfrac{\pi}{2})^2\varepsilon+
	(\tfrac{\pi}{2})^3 \varepsilon^2+\ldots\,;
	\label{eq:4.12} \\
	B_{0,1}^+ &=
	-(\tfrac{\pi}{2})^2\varepsilon-3
	(\tfrac{\pi}{2})^3 \varepsilon^2+\ldots\,;
	\label{eq:4.13} \\
	0> &B_{m,1}^-> B_{m,2}^- > B_{m,3}^-> \ldots
	\label{eq:4.14}
	\end{align}
For $m \geq 1$ and $j_m$:= $[(m+1)/2]$ we obtain the coefficients and expansions
	\begin{align}
	\alpha^-_{m,j_m} &= 2m, \qquad \alpha^+_{m,j_m+1} = 2(m+1)\,;
	\label{eq:4.15a} \\
	B_{m,j_m}^- &=
	-(\tfrac{\pi}{2})^2\varepsilon+
	(2m-1)(\tfrac{\pi}{2})^3 \varepsilon^2+\ldots\,;
	\label{eq:4.15} \\
	B_{m,j_m+1}^+ &=
	-(\tfrac{\pi}{2})^2\varepsilon-
	(2m+3)(\tfrac{\pi}{2})^3 \varepsilon^2+\ldots
	\label{eq:4.16}
	\end{align}
\label{cor:4.2}
\end{cor}
\begin{proof}[\textbf{Proof.}]
The proof follows from proposition~\ref{prop:4.1}, by explicit evaluation.
\end{proof}

\begin{cor}
The assumptions \eqref{eq:3.67} -- \eqref{eq:3.69} of corollary~\ref{cor:3.9} are satisfied for all $1 \leq m\leq m_0$ and sufficiently small $0<\varepsilon \leq \varepsilon_0$.
In particular, expansions \eqref{eq:4.12}, \eqref{eq:4.13} determine the $\varepsilon$-expansions \eqref{eq:1.20} of the boundaries $\underline{b}_k,\ \overline{b}_k$ of the Pyragas region, in our main theorem~\ref{thm:1.2}.
\label{cor:4.3}
\end{cor}

\begin{proof}[\textbf{Proof.}]
For $1 \leq m\leq m_0$, claim \eqref{eq:3.67} follows by comparison of the expansion \eqref{eq:4.16} for $B_{m,j_m+1}^+$ with expansion \eqref{eq:4.13} for $B_{0,1}^+$.
Likewise \eqref{eq:4.12}, \eqref{eq:4.15}, and \eqref{eq:4.14}, in this order, imply
	\begin{equation}
	B_{1,1}^- \leq B_{m,j_m}^- < B_{m,j_m-1}^- < \ldots < B_{m,1}^-\,.
	\label{eq:4.17}
	\end{equation}
This proves claim \eqref{eq:3.69}.
The remaining claim $B_{0,1}^+ <B_{1,1}^-$ of \eqref{eq:3.68} is immediate by comparison of the respective expansions \eqref{eq:4.12} and \eqref{eq:4.13}.
This proves the corollary.
\end{proof}

\begin{figure}[t!]
\centering \includegraphics[width=\textwidth]{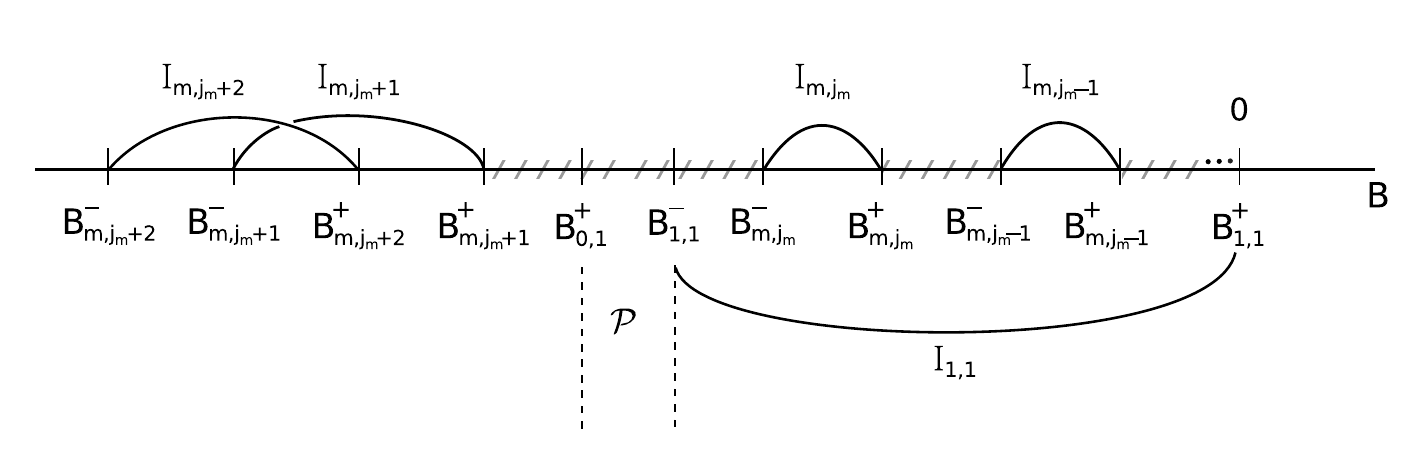}
\caption{\emph{
Stability windows (hashed) between intervals $I_{m,j} =(B_{m,j}^-, B_{m,j}^+)$ of Hopf-induced unstable eigenvalues with imaginary parts in the disjoint intervals designed by $m,j$.
Note how the first, leftmost, stability window between $I_{m,j_{m}+1}$ and $I_{m,j_m}$ contains the only Pyragas region $\mathcal{P} =(B_{0,1}^+, B_{1,1}^-)$ of stable supercritical Hopf bifurcation, for any $m$ such that $I_{m,j_m+1}$ still exists.
}}
\label{fig:4.1}
\end{figure}

For $m\geq 1$, there is an amusing characterization of the critical index $j=j_m=[(m+1)/2]$ as a \emph{Pyragas switch index}, in our expansions.
In fact our $\varepsilon$-expansions \eqref{eq:4.11} easily imply
	\begin{equation}
	\begin{aligned}
	&B_{m,j+1}^+< B_{m,j}^- \quad &\text{for} \quad &j=1, \ldots , j_m\,,\\
	&B_{m,j+1}^+> B_{m,j}^- \quad &\text{for} \quad &j>j_m\,.
	\end{aligned}
	\label{eq:4.18}
	\end{equation}
The interpretation is easy.
In the open interval $B\in I_{m,j}$:= $(B_{m,j}^-, B_{m,j}^+)$, consider the Hopf-induced unstable eigenvalue $\mu =\mu_R+i\tilde{\omega},\ \mu_R>0$, with frequency $\tilde{\omega}$ trapped in the disjoint intervals designated by $m,j$; see lemma~\ref{lem:3.1}, and \eqref{eq:4.5}, \eqref{eq:4.6} with $|\omega| \leq \tfrac{\pi}{2}$.
Then \eqref{eq:4.18} states that successive instability intervals $I_{m,j}$ and $I_{m,j+1}$ open a stability window in between, for $j \leq j_m$, but overlap for $j>j_m$.
See fig.~\ref{fig:4.1}.
In view of \eqref{eq:4.17}, therefore, the stability window
	\begin{equation}
	B_{m,j_m+1}^+ < B < B_{m,j_m}^-
	\label{eq:4.19}
	\end{equation}
is the very first stability window encountered, between the instability intervals $I_{m,j},\ m\geq 1$, as $j$ decreases from $j_m^\text{max}$ to $1$, i.e. as  $B<0$ increases towards zero from absent control at $B =-\infty$.

The instability inherited from $B =-\infty$, on the other hand, is only compensated for once $B_{0,1}^+<B<0,$ by the stabilizing series of Hopf bifurcations in the $m=0$ series at $B= B_{0,j}^+$.
The instability interval $I_{1,1} =(B_{1,1}^-, B_{1,1}^+)$ which starts from $B_{1,1}^- >B_{0,1}^+$, however, extends all the way to $B_{1,1}^+=0$.
Therefore any stability windows between the intervals $I_{m,j}$ of instability, for $m\geq 1$ and $j$ from $j_m+1$ down to $1$, remain ineffective. 
The only exception is the first such gap \eqref{eq:4.19} which, somewhat miraculously, contains the Pyragas region $\mathcal{P}$ of \eqref{eq:3.70} by \eqref{eq:3.67}, \eqref{eq:3.69}, as established above.

We are somewhat amazed how all these first stability windows align, simultaneously for all resonance orders $m$, at the same first order location
	\begin{equation}
	B =-(\tfrac{\pi}{2})^2\varepsilon + \ldots\,,
	\label{eq:4.20}
	\end{equation}
to contribute to Pyragas stabilization from $B=B_{0,1}^+$ to $B=B_{1,1}^-$ by a second order effect.
We will see next how such gaps also arise in the remaining limit $\delta \rightarrow 0$ of large $m\rightarrow +\infty$.


\section{Asymptotic expansions for large $\mathbf{\Omega_m = 1/\delta}$}
\label{sec:5}

In the previous section we have shown how the series of destabilizing Hopf intervals $B\in I_{m,j}$:= $ (B_{m,j}^-, B_{m,j}^+)$ skip the Pyragas candidate
	\begin{equation}
	\mathcal{P}=(B_{0,1}^+, B_{1,1}^-)\,,
	\label{eq:5.1}
	\end{equation}
for bounded values $m\leq m_0$, accordingly bounded $j\leq j_m^{\max}$, and small enough $\varepsilon=1/\omega_k \leq \varepsilon_0(\delta_0)$.
In other words, we have established assumptions \eqref{eq:3.67} -- \eqref{eq:3.69} of corollary~\ref{cor:3.9}, for arbitrarily bounded $m\leq m_0$ and $0<\varepsilon\leq\varepsilon_0(\delta_0)$.
In the present section we will complete this analysis, for large $m>m_0$.

In section \ref{sec:4}, we had fixed $m,j$, to study $\varepsilon$-expansions of $B_{m,j}^\pm$.
We also noticed the central role of
	\begin{equation}
	j=j_m:= [(m+1)/2]\,,
	\label{eq:5.2}
	\end{equation}
where the Pyragas switch \eqref{eq:4.19} actually occurs, between $B_{m,j_m+1}^+$ and $B_{m,j_m}^-$.
This suggests a somewhat delicate parametrization of the relevant expansions by $\varepsilon$ and $\delta$:= $1/\Omega_m = 1/(2m+1)$, both tending to zero in a region $\delta \geq \delta (\varepsilon)$.
Instead, we choose a parametrization of the problem by a rectangular region of $(\delta, \omega)$.
The $\varepsilon$-independent relations $\Omega^\pm= \Omega(\delta, \omega),\, B^\pm= B(\delta, \omega)$ will provide expansions with respect to small $\delta$.
At $j=j_m$ for $B_{m,j_m}^-$, and at $j=j_m+1$ for $B_{m,j_m+1}^+$, we will also obtain expansions for
	\begin{equation}
	\varepsilon = \varepsilon(\delta, \omega)\,,
	\label{eq:5.4}
	\end{equation}
from the hashing relation \eqref{eq:4.5}.
In other words, we determine $\varepsilon$ such that $B_{m,j_m}^-$ and $B_{m,j_m+1}^+$ arise at the frequency parameter $\omega$, for some small $\delta$.
The $\varepsilon$-expansions for the Pyragas boundary $B_{0,1}^+(\varepsilon),\ B_{1,1}^-(\varepsilon)$, in section~\ref{sec:4}, did not depend on $\delta$. 
They will allow us to compare the resulting locations, now, uniformly for small $0<\delta\leq \delta_0$.
This will prove our main theorem, via corollary~\ref{cor:3.9}.

We address the general case in lemma~\ref{lem:5.1}.
The limits $\omega \rightarrow \pm \tfrac{\pi}{2}$, for $B_{m,j}^\pm$, will be considered in lemma~\ref{lem:5.2}.
These results address the cases where $B_{m,j_m}^-,\ B_{m,j_m+1}^+$ actually exist, and $\omega_{m,j_m}^-<\pi\Omega_{m,j_m}^-$.
The remaining cases where $j_m \geq j_m^{\max}$ are prepared by expansions for $\min B$, in proposition~\ref{prop:5.3}, and are resolved in lemma~\ref{lem:5.4}.

As in section~\ref{sec:4}, we may restrict our attention to odd $k$, $|\omega |\leq \tfrac{\pi}{2}$.
See also \eqref{eq:4.3}, \eqref{eq:4.4}.
We mostly replace $m$ by $\delta = 1/\Omega_m= 1/(2m+1)$ and think of $0<\delta\leq \delta_0$ and $0<\varepsilon\leq \varepsilon_0(\delta_0)$ as small continuous, rather than discrete, real variables in all expansions.
For example
	\begin{align}
	2j_m &\phantom{:}= 2[(m+1)/2]=\tfrac{1}{2}(\delta^{-1}-(-1)^m)\,,
	\label{eq:5.5}\\
	a_{m,j_m} &:= 2j_m-1+\tfrac{1}{2}(-1)^m = \tfrac{1}{2}(\alpha^\pm_{m,j_m} \pm1) = \tfrac{1}{2}\delta^{-1}-1\,,
	\label{eq:5.6}
	\end{align}
for the Pyragas switch index $j=j_m$ of \eqref{eq:4.19}, \eqref{eq:5.2} and the associated shift $a=a_{m,j_m}$ in the hashing \eqref{eq:4.5}, \eqref{eq:4.6}.
See also \eqref{eq:4.15a}.

Our $\delta$-expansions are based on section~\ref{sec:2}.
From lemmata~\ref{lem:2.5}, \ref{lem:2.6} we recall how the 2-scale characteristic equation in the form
	\begin{align}
	0 &= (\delta^{-1}+\Omega) \sin (\tfrac{\pi}{2}\Omega) +
	\cos (\omega-\tfrac{\pi}{2}\Omega)\,,
	\label{eq:5.7}\\
	B &= -\sin^2(\tfrac{\pi}{2}\Omega)/\cos \omega\,,
	\label{eq:5.8}
	\end{align}
of \eqref{eq:2.27}, \eqref{eq:2.28} gives rise to unique functions $\Omega= \Omega(\delta, \omega),\ B=B(\delta, \omega)$, successively.
Insertion into the hashing
	\begin{equation}
	\varepsilon = \Omega/(\omega-\pi a_{m,j_m})
	\label{eq:5.9}
	\end{equation}
with \eqref{eq:5.5}, \eqref{eq:5.6} then provides $\varepsilon = \varepsilon^-(\delta, \omega)$, such that we encounter
	\begin{equation}
	B=B_{m,j_m}^-(\delta, \omega) \qquad \text{at} \qquad
	\varepsilon= \varepsilon^-(\delta, \omega)\,,
	\label{eq:5.10}
	\end{equation}
for $-\tfrac{\pi}{2}<\omega \leq \tfrac{\pi}{2} \underline{\Omega}_m$.
Similarly, the hashing
	\begin{equation}
	\varepsilon = \Omega/(\omega-\pi a_{m,j_m+1})\,,
	\label{eq:5.11}
	\end{equation}
with $a_{m,j_m+1} = a_{m,j_m}+2$ in \eqref{eq:5.6}, provides $\varepsilon = \varepsilon^+(\delta, \omega)$ such that we encounter
	\begin{equation}
	B= B_{m,j_m+1}^+(\delta, \omega)\qquad \text{at} \qquad
	\varepsilon= \varepsilon^+(\delta, \omega)
	\label{eq:5.12}
	\end{equation}
for $\tfrac{\pi}{2} \underline{\Omega}_m \leq \omega <\tfrac{\pi}{2}$.
The established $\varepsilon$-expansions of the Pyragas boundaries $B_{0,1}^+(\varepsilon),\ B_{1,1}^-(\varepsilon)$ of corollary~\ref{cor:4.2}, \eqref{eq:4.12} and \eqref{eq:4.13}, finally, at $\varepsilon= \varepsilon^\pm(\delta, \omega)$, respectively, allow us to compare these $\delta$-expansions as required in assumptions \eqref{eq:3.67} and \eqref{eq:3.69} of corollary~\ref{cor:3.9}.

Note how any possible nonuniqueness within the sets $B_{m,j}^+$ is remedied by our parametrization: we neither claim nor require injectivity  of $\omega \mapsto \varepsilon_{m,j}^+(\delta,\omega),$ for fixed $\delta$.

\begin{lem}
For odd $k$, and uniformly in $|\omega|\leq \tfrac{\pi}{2}$, we obtain the following expansions with respect to small $\delta$:
	\begin{align}
	\Omega(\delta,\omega) &=
	-\tfrac{c}{\pi/2}(\delta-s\delta^2+\ldots )
	\label{eq:5.13}\\
	B(\delta,\omega)&=
	-c(\delta^2-2s\delta^3+\ldots )
	\label{eq:5.14}\\
	\varepsilon^\pm(\delta,\omega)&=
	(\tfrac{\pi}{2})^{-2}c(\delta^2+ 
	(\tfrac{\omega}{\pi/2}\mp 2-s)\delta^3+\ldots)
	\label{eq:5.15}\\
	B_{0,1}^+(\delta,\omega)&=
	-c(\delta^2+(\tfrac{\omega}{\pi/2}-2-s)\delta^3+\ldots )
	\label{eq:5.16}\\
	B_{1,1}^-(\delta, \omega)&=
	-c(\delta^2+(\tfrac{\omega}{\pi/2}+2-s)\delta^3+\ldots )
	\label{eq:5.17}
	\end{align}
Here we use the abbreviations $c=\cos \omega,\ s=\sin\omega$.
Omitted terms are of the first omitted integer order in $\delta$.
\label{lem:5.1}
\end{lem}

\begin{proof}[\textbf{Proof.}]
To see why we obtain a uniform expansion of $\Omega$, we rewrite \eqref{eq:5.7} in the equivalent form
	\begin{equation}
	\sin (\tfrac{\pi}{2}\Omega)=
	-c\cos(\tfrac{\pi}{2}\Omega)\delta / (1+\delta(\Omega+s))\,.
	\label{eq:5.18}
	\end{equation}
For $0\leq \delta\leq \delta_0$, the implicit function theorem provides $\Omega= \Omega(\delta,\omega)$, uniformly in $\omega$.
Note the important uniform prefactor $c=\cos \omega$ in
	\begin{equation}
	\Omega = (\Omega/\sin(\tfrac{\pi}{2}\Omega)) \cdot
	\sin (\tfrac{\pi}{2}\Omega)=-c\delta \cdot(\ldots )\,,
	\label{eq:5.19}
	\end{equation}
because $\Omega/\sin(\tfrac{\pi}{2}\Omega)$ is regular nonzero.
To obtain the specific expansion \eqref{eq:5.13} for $\Omega$ we solve \eqref{eq:5.7} for $\delta$, expand for $\Omega$ at $\Omega =0$, and invert the resulting series for
	\begin{equation}
	\begin{aligned}
	\delta &=
	-\sin(\tfrac{\pi}{2}\Omega)/(\cos (\omega -\tfrac{\pi}{2}\Omega)
	+\Omega \sin (\tfrac{\pi}{2}\Omega))\\
	&=
	-\tfrac{\pi/2}{c}\Omega +\ldots \,.
	\end{aligned}
	\label{eq:5.20}
	\end{equation}

To obtain expansion \eqref{eq:5.14} for $B=-\sin^2(\tfrac{\pi}{2}\Omega)/c$ we insert \eqref{eq:5.13} into \eqref{eq:5.8} and observe cancellation of the denominator $c$.
Indeed the numerator picks up a factor $c^2$ from expansion \eqref{eq:5.13}.
This proves \eqref{eq:5.14}.
Of course \eqref{eq:5.14} applies to all $B_{m,j}^\pm = B(\delta, \omega)$, identically, at the appropriate values $\omega$.

The expansions \eqref{eq:5.15} for $\varepsilon^\pm$ follow from the hashings \eqref{eq:5.9} -- \eqref{eq:5.12}, if we replace $a_{m,j_m}$ and $a_{m,j_m+1} = a_{m,j_m}+2$ by their appropriate $\delta$-dependent values $\tfrac{1}{2}\delta^{-1}\pm 1$ from \eqref{eq:5.6}.
The prefactor $c$ remains inherited from $\Omega$.

To prove claims \eqref{eq:5.16}, \eqref{eq:5.17} it only remains to plug \eqref{eq:5.15} into the $\varepsilon$-expansions \eqref{eq:4.12} and \eqref{eq:4.13} for $B_{1,1}^-$ and $B_{0,1}^+$.
Note how the new term of order $\delta^3$, which roughly speaking corresponds to $\varepsilon^{3/2}$, arises from the leading term $-(\tfrac{\pi}{2})^2\varepsilon^\pm +\ldots $ alone.
This proves the lemma.
\end{proof}

\begin{lem}
Let $0<\delta \leq \delta_0$ be small enough and consider odd $k$.
Let $\Omega, B, \varepsilon^\pm, B_{0,1}^+, B_{1,1}^-$ be expanded as in lemma~\ref{lem:5.1}.

Then we obtain the inequalities
	\begin{equation}
	B_{1,1}^-(\delta, \omega)<B(\delta,\omega) < B_{0,1}^+(\delta, \omega)\,,
	\label{eq:5.21}
	\end{equation}
for all $0<\delta \leq \delta_0,\ |\omega|<\tfrac{\pi}{2}$.
For $j_m+1\leq j_m^{\max}$ and at $\varepsilon= \varepsilon^+(\delta,\omega)$ we conclude, in particular, 
	\begin{equation}
	B_{m,j_m+1}^+ < B_{0,1}^+(\delta,\omega)\,,
	\label{eq:5.22}
	\end{equation}
i.e. assumption \eqref{eq:3.67} of corollary~\ref{cor:3.9} holds. 
Likewise, for $j_m\leq j_m^{\max}$ and at $\varepsilon= \varepsilon^-(\delta, \omega)$ we conclude 
	\begin{equation}
	B_{1,1}^-(\delta,\omega)< B_{m,j_m}^-\,.
	\label{eq:5.22a}
	\end{equation}

Under the additional assumption
	\begin{equation}
	-\tfrac{\pi}{2} <\omega_* \leq \pi\Omega(\delta,\omega_*)\,,
	\label{eq:5.23}
	\end{equation}
we can also assert $-\tfrac{\pi}{2} <\omega \leq \pi\Omega(\delta,\omega),$ for all $-\tfrac{\pi}{2} \leq \omega\leq \omega_*.$ Moreover
	\begin{equation}
	\omega \mapsto B(\delta, \omega)<0
	\label{eq:5.24}
	\end{equation}
is then strictly decreasing, for fixed $\delta$ and all $-\tfrac{\pi}{2} \leq \omega\leq \omega_*.$

More specifically, suppose the additional assumption \eqref{eq:5.23} holds at $\Omega(\delta,\omega_*)= \Omega_{m,j_*}^-,\ \omega_* = \omega^-(\Omega_{m,j_*}^-)$, for some $1\leq j_*\leq j_m^{\max},\ m\geq 1$.
Then
	\begin{equation}
	B_{m,j_*}^-<\ldots < B_{m,1}^-\leq 0
	\label{eq:5.24a}
	\end{equation}
holds at $\varepsilon= \varepsilon(\delta,\omega_*)$, for all $1\leq j\leq j_*$.
For the above ranges of $\varepsilon,\delta$, this establishes the Pyragas region $\mathcal{P}$.
\label{lem:5.2}
\end{lem}

\begin{proof}[\textbf{Proof.}]
Claim \eqref{eq:5.21} may look somewhat paradoxical, at first sight.
To prove \eqref{eq:5.21}, nevertheless, we first divide \eqref{eq:5.21} by $c=\cos \omega >0$ and compare with expansions \eqref{eq:5.14}, \eqref{eq:5.16}, \eqref{eq:5.17} of lemma~\ref{lem:5.1}.
For $|\omega| <\tfrac{\pi}{2}-\eta_0$ bounded away from $\tfrac{\pi}{2}$, claim \eqref{eq:5.21} becomes equivalent to the obvious inequalities
	\begin{equation}
	\begin{aligned}
	\tfrac{\omega}{\pi/2} +2-s \ &> \ -2s &>&\  \tfrac{\omega}{\pi/2}-2-s\,,
	\quad \text{i.e.}\\
	\tfrac{\omega}{\pi/2}+2 \ &> \ -s &>&\  \tfrac{\omega}{\pi/2}-2\,.
	\end{aligned}
	\label{eq:5.25}
	\end{equation}
Equality holds for $\omega =-\tfrac{\pi}{2}$, on the left, and for $\omega =+\tfrac{\pi}{2}$, on the right.
Therefore $\eta_0=0$ does not seem an option for proving inequality \eqref{eq:5.21}, uniformly for small $\delta$, at first.
However, we obtain uniform expansions for the partial derivatives $\partial_\omega B,\ \partial_\omega B_{0,1}^+,\ \partial_\omega B_{1,1}^-$ as well, by differentiation of the coefficients of the $\delta$-expansions \eqref{eq:5.14}, \eqref{eq:5.16}, \eqref{eq:5.17}.
At $\omega =-\tfrac{\pi}{2}$ we obtain
	\begin{align}
	\partial_\omega((B-B_{1,1}^-)/c) &=
	\partial_\omega (s+\tfrac{\omega}{\pi/2}+2)\delta^3 +\ldots =
	\tfrac{1}{\pi/2}\delta^3+\ldots >0\,,
	\label{eq:5.26}\\
	\intertext{and at $\omega = +\tfrac{\pi}{2}$, similarly,}
	\partial_\omega ((B_{0,1}^+-B)/c) &=
	\partial_\omega(-\tfrac{\omega}{\pi/2}+2-s)\delta^3+\ldots =
	-\tfrac{1}{\pi/2}\delta^3+\ldots <0\,.
	\label{eq:5.27}
	\end{align}
This establishes positivity of the differences for all $|\omega|<\tfrac{\pi}{2}$, uniformly for all $0<\delta \leq \delta_0$, and settles claim \eqref{eq:5.21}.

To prove claim \eqref{eq:5.22} we invoke $B(\delta, \omega) < B_{0,1}^+(\delta, \omega)$ from \eqref{eq:5.21}.
In fact
	\begin{equation}
	B_{m,j_m}^- = B(\delta, \omega)\,,\qquad
	B_{m,j_m+1}^+= B(\delta,\omega)\,,
	\label{eq:5.28}
	\end{equation}
in our parametrization, for any $m\geq 1.$ Note how \eqref{eq:5.28} refers to possibly different $\varepsilon =  \varepsilon^\pm(\delta,\omega)$ given by \eqref{eq:5.9}, \eqref{eq:5.10} and \eqref{eq:5.11}, \eqref{eq:5.12}, respectively.
Indeed $\varepsilon$ has been eliminated in the 2-scale characteristic equations \eqref{eq:5.7}, \eqref{eq:5.8}, and only enters after the $(m,j)$-dependent hashings \eqref{eq:5.9}, \eqref{eq:5.11}.
Thus \eqref{eq:5.22} holds by definition of $\varepsilon^+(\delta, \omega)$.
The arguments for claim \eqref{eq:5.22a} via $\varepsilon^-(\delta,\omega)$ are completely analogous.

To complete the proof we only have to establish the strict monotonicity of $\omega \mapsto B(\delta, \omega)$ and of $j \mapsto B_{m,j}^-,\ j=1, \ldots ,j_*$ as claimed in \eqref{eq:5.24}, \eqref{eq:5.24a}.
We invoke lemmata~\ref{lem:2.5} and \ref{lem:2.6}.
By lemma~\ref{lem:2.5} we have strictly decreasing dependence $\Omega \mapsto \omega^-(\delta, \Omega)$, for fixed $\delta >0$.
By lemma~\ref{lem:2.6}, assumption \eqref{eq:5.23} implies that $\omega\mapsto B(\delta,\omega)$ decreases strictly, as long as $\omega \leq \pi\Omega(\delta,\omega)$.
By monotonicity of $\Omega$, we remain in this region, for all smaller $\omega \geq -\tfrac{\pi}{2}$, once we are ever inside.
This proves claim \eqref{eq:5.24}.

Specifically, suppose assumption \eqref{eq:5.23} holds for $\omega_* = \omega^-(\delta, \Omega_{m,j_*}^-)$, and hence for all $-\tfrac{\pi}{2}\leq \omega\leq \omega_*=\omega^-(\delta,\Omega_{m,j_*}^-).$
Then conclusion \eqref{eq:5.24} holds for all $-\tfrac{\pi}{2}\leq \omega\leq \omega_*=\omega^-(\delta,\Omega_{m,j_*}^-).$
Hashing \eqref{eq:5.9} implies strict monotonicity of the slow frequencies,
	\begin{equation}
	\Omega_{m,j_*}^-< \ldots < \Omega_{m,1}^-\leq 0\,.
	\label{eq:5.29}
	\end{equation}
In particular \eqref{eq:5.29} successively implies
	\begin{equation}
	\begin{aligned}
	\omega_*\quad &\ = \ \omega^-(\Omega_{m,j_*}^-) &> \ldots &>
	\ \omega^-(\Omega_{m,1}^-) &\geq &-\tfrac{\pi}{2}\,,\\
	B(\delta,\omega_*) &\ = \ \quad B_{m,j_*}^- &< \ldots &<
	\quad B_{m,1}^- &\leq & \quad 0\,,
	\end{aligned}
	\label{eq:5.30}
	\end{equation}
at $\varepsilon= \varepsilon(\delta,\omega_*).$ 
This proves claim \eqref{eq:5.24a}.

For the above ranges of $\varepsilon,\delta$, the inequalities \eqref{eq:5.22} and \eqref{eq:5.22a}, \eqref{eq:5.24a} also validate assumptions \eqref{eq:3.67}, \eqref{eq:3.69} of corollary~\ref{cor:3.9}, respectively.
This establishes the Pyragas region $\mathcal{P}$, in the above ranges of $\varepsilon,\delta$, and the lemma is proved.
\end{proof}

We now address the remaining segment of the piecewise strictly monotone curve $\omega \mapsto \Omega(\delta, \omega)$, where
	\begin{equation}
	\pi \Omega (\delta, \omega) \leq \omega\leq 
	\tfrac{\pi}{2}\Omega (\delta, \omega)<0\,.
	\label{eq:5.31}
	\end{equation}
By lemma~\ref{lem:2.6}, this segment contains
	\begin{equation}
	B_{\min}(\delta) = \min\limits_{|\omega |\leq \tfrac{\pi}{2}} B(\delta, \omega)\,.
	\label{eq:5.32}
	\end{equation}

\begin{prop}
Let $k$ be odd.
Uniformly, for $\omega$ satisfying \eqref{eq:5.31}, we have the expansions
	\begin{align}
	\Omega (\delta, \omega) &= 
	-\tfrac{1}{\pi/2} \delta +\mathcal{O}(\delta^3)\,;
	\label{eq:5.33}\\
	B(\delta, \omega) &=
	-\delta^2 +\mathcal{O}(\delta^4)\,.
	\label{eq:5.34}\\
	\intertext{In particular this implies}
	B_{\min}(\delta) &= 
	-\delta^2+\mathcal{O}(\delta^4)\,.
	\label{eq:5.35}
	\end{align}
\label{prop:5.3}
\end{prop}
\begin{proof}[\textbf{Proof.}]
We invoke expansions \eqref{eq:5.13} and \eqref{eq:5.14} of lemma~\ref{lem:5.1}.
Indeed \eqref{eq:5.13} and \eqref{eq:5.31} imply
	\begin{equation}
	-2c\delta + \ldots = \pi\Omega (\delta, \omega) \leq\omega\leq
	\tfrac{\pi}{2}\Omega(\delta, \omega)= -c\delta +\ldots
	\label{eq:5.36}
	\end{equation}
with $c=\cos \omega$.
In particular $\omega = \mathcal{O}(\delta)$, uniformly in the region \eqref{eq:5.31}.
Hence the $\delta$-expansion \eqref{eq:5.13} with $s= \sin \omega$ implies claim \eqref{eq:5.33}.
Similarly, \eqref{eq:5.14} and $\omega =\mathcal{O}(\delta)$ imply \eqref{eq:5.34}.
Claim \eqref{eq:5.35} follows from the uniform expansion \eqref{eq:5.34} by the remark preceding the proposition, and the proof is complete.
\end{proof}

\begin{lem}
Let $0<\delta \leq \delta_0$ be small enough and consider odd $k$.
Let $(\Omega_*,\omega_*)$ denote the unique intersection point of the line $\omega=\pi\Omega$ with the 2-scale curve $\Omega = \Omega(\delta, \omega)$.
Assume that the hashing line
	\begin{equation}
	\Omega = \varepsilon(\omega -\pi a_{m,j_m})\,,
	\label{eq:5.37}
	\end{equation}
of $j_m$:= $[(m+1)/2]$ intersects the line $\pi \Omega = \omega,\ \omega \leq 0$, to the left of $(\Omega_*, \omega_*)$.
Then
	\begin{equation}
	B_{1,1}^-(\varepsilon) < B_{\min} (\delta) <0\,.
	\label{eq:5.38}
	\end{equation}		
\label{lem:5.4}
\end{lem}

\begin{proof}[\textbf{Proof.}]
We recall $a_{m,j_m}=\tfrac{1}{2} \delta^{-1} -1$; see \eqref{eq:5.6}.
Insertion into hashing \eqref{eq:5.37} provides the intersection with $\pi \Omega = \omega$ at 
	\begin{equation}
	\begin{aligned}
	\varepsilon &=
	\Omega/(\omega-\pi a_{m,j_m})= \tfrac{1}{\pi}\Omega/
	(\Omega-\tfrac{1}{2} \delta^{-1}+1)=\\
	&=
	-\tfrac{1}{\pi/2}\delta\Omega/(1-2\delta-2\delta\Omega)\geq
	-\tfrac{1}{\pi/2}\delta\Omega(\delta,\omega_*)/
	(1-2\delta-2\delta\Omega(\delta, \omega_*))=\\
	&=
	(\tfrac{\pi}{2})^{-2}\delta^2/(1-2\delta)+\mathcal{O}(\delta^4)=
	(\tfrac{\pi}{2})^{-2}(\delta^2+2\delta^3+\mathcal{O}(\delta^4))\,.
	\end{aligned}
	\label{eq:5.39}
	\end{equation}
Here we have used intersection to the left of $(\Omega_*, \omega_*)$, i.e. $\Omega \leq \Omega_* = \Omega(\delta,\omega_*)< 0$, and expansion \eqref{eq:5.33} at $(\Omega_*, \omega_*)$.
The $\varepsilon$-expansion \eqref{eq:4.12} of corollary~\ref{cor:4.2} for $B_{1,1}^-= B_{1,1}^-(\varepsilon)$, the estimate $\varepsilon = \varepsilon^- = \mathcal{O}(\delta^2)$ of \eqref{eq:5.15}, and comparison with \eqref{eq:5.35} via \eqref{eq:5.39} then yield
	\begin{equation}
	\begin{aligned}
	B_{1,1}^- (\varepsilon) 
	&=	-(\tfrac{\pi}{2})^2\varepsilon+
	(\tfrac{\pi}{2})^3 \varepsilon^2 &+&\ \mathcal{O}(\varepsilon^3)  &\leq&\\
	&\leq \quad -\,\delta^2 -2\delta^3 &+&\ \mathcal{O}(\delta^4) &<& \ B_{\min}(\delta)<0\,.
	\end{aligned}
	\label{eq:5.40}
	\end{equation}
This proves claim \eqref{eq:5.38} and the lemma.
\end{proof}


\section{Proof of theorem~1.2}
\label{sec:6}

Our proof of theorem \ref{thm:1.2} is based on just a detailed stability analysis of the 2-delay characteristic equation  \eqref{eq:1.2ch}, with $\vartheta=0$, at the Hopf bifurcation points $\lambda = \lambda_k := (-1)^{k+1} \omega_k\,$ of \eqref{eq:1.3}.
Emphasis is on control-induced eigenvalues $\mu = i \tilde{\omega}$ in the limit $\varepsilon \rightarrow 0$ of large frequencies $1/\varepsilon = \omega_k = (k+\tfrac{1}{2})\pi$. 

We summarize the proof; see figs.~\ref{fig:3.1} and \ref{fig:3.2} for an illustration.
We only address the case of large odd $k$; the case of even $k$ is analogous.

In section~\ref{sec:2} we have studied the 2-scale characteristic equation \eqref{eq:2.8}, \eqref{eq:2.9} which eliminates $\varepsilon> 0$.
Instead, imaginary eigenvalues $\mu = \pm i\tilde{\omega}$ have been represented by a slow Hopf frequency $\tilde{\Omega} = \varepsilon\tilde{\omega}$, in addition to $\tilde{\omega}$ itself.
The case of real eigenvalues $\mu$, and their crossing at $\mu=0$ due to the scaled control parameter $b=2\varepsilon B$, was treated in lemma~\ref{lem:2.1} and corollary~\ref{cor:2.2}.
The hashing $\tilde{\Omega}= \varepsilon\tilde{\omega}$ was detailed, and normalized to $\Omega=\tilde{\Omega}-(2m+1)$, $\omega\equiv\tilde{\omega}\ ($mod $2 \pi),$ in lemma~\ref{lem:2.3}.
In lemma~\ref{lem:2.4} we rewrote the characteristic equation in normalized frequencies $\Omega, \omega$ instead of $\tilde{\Omega}, \tilde{\omega}$.
Lemma~\ref{lem:2.5} studied the resulting fundamental nonlinear 2-scale relation between $\Omega$ and $\omega$.
The resulting control parameters $B$ were addressed in lemma~\ref{lem:2.6}.
Emphasis there was on monotonicity properties with respect to $\omega$.

Section~\ref{sec:3} has been devoted to the resulting nontrivial control-induced Hopf bifurcations, at $B=B_{m,j}^\pm$, in contrast to the spectrum inherited from the uncontrolled case.
Here $m$ indicates the proximity of an $m:1$ resonance of $\tilde{\omega}$ with $\omega_k$.
The detailed analysis included simplicity of Hopf eigenvalues, and their transverse crossing directions; see theorem~\ref{thm:3.4}.
The analysis culminated in corollaries~\ref{cor:3.5} and \ref{cor:3.9}.
In corollary~\ref{cor:3.5} we established the absence of any Pyragas control, for control parameters $B>0$.
The central corollary~\ref{cor:3.9} established the control region
	\begin{equation}
	B\in\mathcal{P}=(B_{0,1}^+,\ B_{1,1}^-)
	\label{eq:6.1}
	\end{equation}
as the only Pyragas region, under assumptions \eqref{eq:3.67} -- \eqref{eq:3.69} of certain inequalities among the control-induced Hopf parameter values $B=B_{m,j}^\pm$.

Section~\ref{sec:4} collected $\varepsilon$-expansions for $B_{m,j}^\pm$, and auxiliary quantities like the fast and slow normalized Hopf frequencies $\omega_{m,j}^\pm$ and $\Omega_{m,j}^\pm$, see proposition~\ref{prop:4.1}.
These Taylor expansions in $\varepsilon$ amount to expansions in the limit of arbitrarily large unstable dimensions $k$, and arbitrarily rapid oscillation frequencies $\omega_k=(k+\tfrac{1}{2})\pi$, of the original Hopf bifurcations of \eqref{eq:1.3} in absence of any control.
Corollary~\ref{cor:4.3} then established the crucial Hopf inequalities \eqref{eq:3.67} -- \eqref{eq:3.69} for small enough $0<\varepsilon \leq \varepsilon_0(\delta_0)$, uniformly for bounded indices $m\leq m_0$ and $j \leq j_{m}+1$ in the Hopf series $B_{m,j}^+$, as well as for $B_{m,j}^-$ with $j\leq j_m$.
Here $j_m:=[(m+1)/2].$
In particular this settled assumption \eqref{eq:3.68}.
At the end of this section, we observed how our proof hinged on a miraculous gap property of instability intervals $B \in I_{m,j}=(B_{m,j}^-,\ B_{m,j}^+)$: the first gap occurs at $j=j_{m}$ and contains the Pyragas region \eqref{eq:6.1}.
See fig. \ref{fig:4.1}.

To complete the proof, the limit of large near-resonances $m>m_0$, alias small $\delta = 1/\Omega_m=1/(2m+1)$, was addressed in section~\ref{sec:5}.
For odd $k$, lemma~\ref{lem:5.1} collected $\delta$-expansions for $\Omega, B, \varepsilon^\pm\,$, and $B_{0,1}^+,\ B_{1,1}^-$ in terms of the (normalized) fast Hopf frequencies $\omega$.
The Taylor expansions with respect to $\delta$ address the limit of $m:1$ resonances, for $m \rightarrow \infty\,$.
In lemma~\ref{lem:5.2} this settled the remaining two assumptions \eqref{eq:3.67} and \eqref{eq:3.69} of corollary~\ref{cor:3.9}, up to one exceptional case.
The exceptional case was caused by our additional assumption, thus far, that all relevant Hopf points $B_{m,j}^-, \ j=1, \ldots , j_m$ in \eqref{eq:3.69} occur in the (normalized) frequency range
	\begin{equation}
	-\tfrac{\pi}{2}\leq \omega \leq \pi\Omega <0\,,
	\label{eq:6.2}
	\end{equation}
i.e. at, or to the right of, the unique intersection $(\Omega_*, \omega_*)$ of the 2-scale frequency relation $\Omega = \Omega(\delta,\omega)$ with the straight line $\omega = \pi\Omega$ in the $(\Omega, \omega)$-plane.
See \eqref{eq:5.7}, \eqref{eq:2.27}, \eqref{eq:2.36}.
Since Hopf points themselves originate from intersections of straight hashing lines with that 2-scale frequency relation, the only remaining case was that the hashing line at $j=j_m=[(m+1)/2]$ intersects $\omega = \pi\Omega$ to the left of $(\Omega_*, \omega_*)$.
This case was addressed, and settled, in lemma~\ref{lem:5.4}.
Indeed the minimal possible value $B_{\min}$ of control parameters induced by the 2-scale frequency relation then satisfies
	\begin{equation}
	B_{1,1}^-(\varepsilon)< B_{\min}(\delta)<0\,.
	\label{eq:6.3}
	\end{equation}
Since $B_{\min} \leq B_{m,j}^-$ for all $j$, this also established the remaining assumption \eqref{eq:3.69} in this one remaining configuration.
This completes the proof of our main theorem~\ref{thm:1.2}.


\end{document}